\documentclass[11pt]{article}

\usepackage{amsmath,amsthm,verbatim,amssymb,amsfonts,amscd, graphicx}
\usepackage{graphicx}
\usepackage{mathrsfs}
\usepackage{enumerate}
\usepackage{mathtools}
\usepackage{lipsum}         
\usepackage{xargs}                      
\usepackage{enumitem}
\usepackage{authblk}
\usepackage{cite}
\usepackage{hyperref}
\usepackage{varioref}  
\usepackage{cleveref}  
\usepackage{bbm}
\usepackage{float}
\usepackage[pdftex,dvipsnames]{xcolor}  
\usepackage{caption}
\usepackage{subcaption}
\usepackage{siunitx}
\usepackage[colorinlistoftodos,prependcaption,textsize=tiny]{todonotes}

\usepackage[a4paper, left=2.75cm, right=2.75cm, top=2.0cm, bottom=2.0cm]{geometry}

\theoremstyle{plain}

\newtheorem{lemma}{Lemma}

\newtheorem{proposition}{Proposition}

\theoremstyle{definition}
\newtheorem{definition}{Definition}

\usepackage{bm}

\newcommand{\N}{\mathbb{N}}
\newcommand{\C}{\mathbb{C}}
\newcommand{\R}{\mathbb{R}}

\newcommand{\Gkn}{\mathcal{G}_{\text{\tiny{KN}}}}
\newcommand{\Gw}{\mathcal{G}_{\text{\tiny{W}}}}
\newcommand{\Op}{\text{Op}}
\newcommand{\Opw}{\text{Op}_{\text{\tiny{W}}}}
\newcommand{\Opkn}{\text{Op}_{\text{\tiny{KN}}}}

\begin{document}

\title{Wave-Current-Bathymetry Interaction Revisited: Modeling, Analysis and Asymptotics}

\author[1]{Adrian Kirkeby}
\author[2]{Trygve Halsne}

\affil[1]{Simula Research Laboratory, Norway \\ \texttt{adrian@simula.no}}
\affil[2]{Norwegian Meteorological Institute, Norway \\ \texttt{trygveh@met.no}}

\maketitle

\begin{abstract}
Starting from the free surface Euler equations, we derive a leading-order system in terms of surface variables, depending on the surface current and on the bathymetry through the depth-dependent Dirichlet-to-Neumann (DN) operator. The resulting system is shown to be well-posed using the theory of hyperbolic systems of pseudo-differential operators.

We then consider wave propagation in slowly varying environments. As an explicit approximation to the DN operator, the semiclassical Weyl quantization of the symbol $g_b(X,\xi)=|\xi|\tanh(b(X)|\xi|)$ is shown to be both asymptotically accurate and consistent with the self-adjoint structure of the true operator, and to provide the natural framework for asymptotic analysis of the wave system. 

A central consequence of the resulting framework is that classical asymptotic models--including the wave action equation, the mild-slope equation, the Schrödinger equation, and the action balance equation--emerge systematically from a single formulation. By deriving these equations, we show how the simple leading order system with the Weyl quantization of the DN operator provides a unified and mathematically consistent framework for the asymptotic linear theory of wave–current–bathymetry interaction, hence providing a transparent, rigorous and accessible route from the primitive Euler equations to the mentioned asymptotic models. Throughout, numerical experiments are included to illustrate the analysis.
\end{abstract}

\section{Introduction}
The study of water waves interacting with ocean currents and bathymetry is a classical topic within applied mathematics and asymptotic analysis, and has produced important and useful insights into the physics of this immensely complex phenomenon. Much of the classical theory is described in works like \cite{mei1989applied,phillips1977upperocean,johnson1997modern,stoker2019water,peregrine1976interaction} and references therein, and some of its applications to areas like wave forecasting or coastal engineering are described in \cite{komen1994dynamics, holthuijsen2010waves,kamphuis2020introduction}. While this foundational theory is elegant and its insights powerful, classical methods were naturally developed before the advent of today's more rigorous mathematical frameworks. Consequently, the underlying modeling and asymptotic assumptions can sometimes be challenging to both understand and verify by modern standards, which may complicate their justification in applications. Furthermore, while modern mathematical analysis and numerical work on water waves now rely heavily on the surface formulation of the problem, the so-called Zakharov--Craig--Sulem formulation \cite{waterwavesprob}, the asymptotic analysis of classical wave-current-bathymetry interactions remains largely unexplored from this viewpoint. The aim of this article is therefore to bridge the gap between the classical and the modern and to provide a unified and mathematically sound linear theory for wave-current-bathymetry interaction. The paper is organized as follows:

\textbf{\Cref{sect: mathematical model}:}
We start by carefully constructing our mathematical model for smooth but arbitrary bathymetries and slowly varying background currents. Assuming that the vorticity of the background flow is small and localized in space, we show that the leading order wave velocity remains a potential flow. Appealing to Luke's variational principle as a selection criterion in the linearization, we then formulate the leading order system of partial differential equations (PDE) as a linear, nonlocal system of evolution equations for the wave amplitude and surface potential using the Dirichlet-to-Neumann (DN) operator.

\textbf{\Cref{sect: well-posed}:} In this section we employ techniques from the theory of hyperbolic systems of pseudo-differential operators to establish the well-posedness of the initial-value problem. The analysis relies on splitting the DN operator into an elliptic, well-behaved part and a smoothing operator depending on the bathymetry, and a diagonalization procedure. By proving the existence, uniqueness, and stability of solutions, we confirm that our model is mathematically reasonable.

\textbf{\Cref{sect: asymptotics}:} This section focuses on the asymptotic analysis of the wave system over slowly varying bathymetries and currents. We argue why the self-adjointness of the DN operator is an essential feature to obtain correct asymptotic model, and show that a self-adjoint and accurate approximation of the DN operator is obtained using the semiclassical Weyl quantization of the symbol $g_b(X,\xi) = |\xi|\tanh(b(X)|\xi|)$. From this result, the derivation of the various asymptotic models follow, and
the schematic below outlines the flow and structure of the section and the equations and techniques involved. 

\begin{center}
\makebox[\textwidth][c]{
\scalebox{0.76}{
\begin{tikzpicture}[
    ->, >=stealth, thick,
    level 1/.style={sibling distance=8cm, level distance=3.5cm},
    level 2/.style={sibling distance=5cm, level distance=2.5cm},
    every node/.style={draw, rounded corners, inner sep=6pt, font=\scriptsize, align=center},
    edge from parent/.style={draw, -stealth}
]
\node (top) { Surface system \\
    $\begin{cases}
        \partial_t \eta +  \nabla_X\cdot(\bar{U} \eta) = \mathcal{G}(b)\varphi, \\
        \partial_t \varphi + \bar{U}\cdot \nabla_{X} \varphi = - g\eta .
    \end{cases}$
}
    child { node {
        $\begin{aligned}
            \partial_t \mathcal{E}&= -  \nabla_X \cdot (\bar{U} \mathcal{E}) + \frac{1}{2}g\left( \overline{\eta}\mathcal{G}\varphi - \varphi \mathcal{G}\overline{\eta}\right) \\
            &\quad +\nabla_X \cdot \bar{U}\left(\frac{1}{2}\varphi \mathcal{G}\overline{\varphi} -\frac{1}{2}g  |\eta|^2\right)  + \frac{1}{2}\varphi [\bar{U}\cdot \nabla_X,\mathcal{G}] \overline{\varphi}
        \end{aligned}$ 
    }
        child[level distance=3.05 cm] { node {Total energy evolution \\ $\begin{aligned} \frac{\mathrm{d}}{\mathrm{d}t}\int \mathcal{E} \mathrm{d}X &= -\int\nabla_X \cdot \bar{U} \frac{1}{2}g|\eta|^2 \mathrm{d}X \\ &\quad + \int_{\Omega(0,b)}  \mathcal{P} \mathrm{d}X \end{aligned}$} 
            edge from parent node[above left, draw=none, pos = 0.85] {$\int \mathcal{E} \mathrm{d}X$}
        }
        child { node (wave_action) {Conservation of wave action \\ $\partial_t \mathcal{A} + \nabla_X\cdot ((\bar{U} + C_g) \mathcal{A}) = 0$} 
            edge from parent node[above right, draw=none, pos = 0.8] {WKB}
        }
        edge from parent node[above left, draw=none,pos = 0.6] {Energy density, $\mathcal{E} = \frac{1}{2}(g\eta^2 + \varphi \mathcal{G}\overline{\varphi})$}
    }
    child[level distance=5.5cm] { node (right_pde) {
        $\begin{cases}
            \partial_t^\mu \eta +  \nabla_X^\mu\cdot(\bar{U} \eta) = \Gw^\mu(b)\varphi, \\
            \partial_t^\mu \varphi + \bar{U}\cdot \nabla_{X}^\mu \varphi = - g\eta .
        \end{cases}$
    }
        child[level distance=3.5cm] { node {Mild-slope equation \\ $\nabla_X \cdot c \nabla \psi + c \kappa^2 \psi = 0$} 
            edge from parent node[above left, draw=none,pos = 0.7] {Time-harmonic waves}
        }
        child[level distance=3.5cm] { node {Schrödinger equation \\ $\begin{aligned} &\partial_t A_\varphi = -(\bar{U} + C_g)\cdot \nabla_X A_\varphi \\
        &-\frac{1}{2}\left(\nabla_X \cdot(\bar{U} + C_g)\right) A_\varphi \\ & - \frac{D_{\bar{U}+ C_g}\sigma }{2\sigma}A_\varphi  -\frac{i\mu}{2}\nabla_X \cdot (D \nabla_X A_\varphi) \end{aligned} $} 
            edge from parent node[left, draw=none] {WKB}
        }
        child[level distance=3.5cm] { node {Action balance equation \\ $\partial_t \mathcal{A} + \{\omega, \mathcal{A} \} = 0$} 
            edge from parent node[above right, draw=none, pos = 0.7] {Wigner distribution}
        }
        edge from parent node[above right, draw=none] {Slow horizontal coordinates $(\mu \ll 1)$ \\ $\mathcal{G} \mapsto \Gw^\mu(b)$}
    };

\path (top.north) ++(0, 1.5) node (euler) {Free surface Euler equations};

\draw[->] (euler) -- (top);

\draw[->] (right_pde) -- (wave_action);

\end{tikzpicture}
}}
\end{center}
We derive an evolution equation for the exact wave energy density $\mathcal{E}$, which is explored in detail; we find that it leads to a simple approximate evolution equation for the total wave energy depending on the bulk production---known from turbulence theory---and surface flow divergence of the current.  
We continue by considering the action of the Weyl operator on wave packets and derive an evolution equation for the phase-averaged energy density, and we show that this equation agrees with the classical wave action equation.

Next, we show how both the mild-slope equation and linear Schrödinger equation can be derived from the asymptotic formulation of the governing PDE. 

Last, we explore spectral energy dynamics and apply the Wigner transform to a diagonalized version of the system to derive a spectral energy evolution equation and the so-called energy balance equation. 

Although parts of the article are quite technical, we have tried to connect the results of the analysis to the classical theory in a clear manner so that it can hopefully be read and appreciated by a wider audience. The longer and more involved proofs are therefore moved to the Appendix. In addition, the paper is accompanied by a Jupyter Notebook/Python module to simulate wave--current--bathymetry interaction (available at \href{https://github.com/jfkirkeby/WaCuBa}{https://github.com/jfkirkeby/WaCuBa}), and we use this numerical framework to illustrate and verify several of our findings.

\section{Mathematical model}

\label{sect: mathematical model}
We denote the domain occupied by the fluid by $\Omega(\eta,b) = \{(X,z) \in \R^2 \times (-b(X),\eta(t,X))\}$, where $\eta(t,X)$ denotes the surface elevation and let $b(X)$ be the depth, and require $|\eta - b| > 0$. We assume the fluid is incompressible and satisfies the incompressible Euler equations, i.e., 
\begin{equation}
    \partial_t \bm{U} + (\bm{U}\cdot \nabla_{X,z})\bm{U} = -\frac{\nabla_{X,z}P}{\rho} -g\bm{e}_3, \quad \text{in} \quad \Omega(\eta,b), \quad \nabla_{X,z} \cdot\bm{U} = 0, \quad \text{in} \quad \Omega(\eta,b).
    \label{eq: euler}
\end{equation}
Here $\bm{U}$ denotes the fluid velocity vector and $\nabla_{X,z} = (\nabla_X,\partial_z)^\top$, $g$ is the gravitational acceleration and $\bm{e}_3 = (0,0,1)^\top$ We write $\bm{U} = (U,W)^\top$ with horizontal component $U = (U_1,U_2)$. With $\nu$ as the downward unit normal of the bottom surface, we assume the boundary conditions 
\begin{equation}
     \nu \cdot \bm{U}|_{z = -b(X)} = 0 \quad \text{and} \quad P|_{z = \eta} = P_{atm}.  
    \label{eq: bc euler}
\end{equation}
Incompressibility and conservation of mass leads to conservation law
\begin{equation}
    \partial_t \eta + \nabla_X \cdot\int_{-b}^\eta  U \mathrm{d}z = 0.
    \label{eq: mass cons}
\end{equation}
Using Leibniz' rule and $\nabla_X \cdot U = -\partial_z W $, one finds that  
\begin{align*}
     \nabla_X \cdot\int_{-b}^\eta  U \mathrm{d}z &= \nabla_X \eta \cdot U|_{z=\eta} + \nabla_X b \cdot U|_{z=-b} - \int_{-b}^\eta \partial_z W \mathrm{d}z  \\
     &= \nabla_X \eta \cdot U|_{z=\eta}  - W|_{z=\eta} + \nabla_X b \cdot U|_{z=-b}+ W|_{z=-b}.
\end{align*}
Due to \eqref{eq: bc euler} the bottom terms cancel and we get the so-called kinematic boundary condition:
\begin{equation}
    \partial_t \eta + \nabla_X \eta \cdot U|_{z=\eta} = W|_{z=\eta}
    \label{eq: kin full}
\end{equation}
The above equations constitute the model from which we will derive our linearized system. 
\begin{center}
\vspace{2mm}
    \underline{\textbf{Bathymetry, Current and Linearization}}
    \vspace{2mm}
\end{center}    
The bathymetry is considered stationary, and described by a smooth, bounded function $b(X)$ such that  $b_{\text{\tiny{min}}} \leq b(X) \leq b_{\text{\tiny{max}}} $ with $b_{\text{\tiny{min}}} > 0$. For technical reasons, we also assume that $b(X)$ eventually tends to $b_{\text{\tiny{max}}}$, i.e., there is some finite $R $ such that $b(X) = b_{\text{\tiny{max}}} $ for $|X| > R$.

Next, let  $\bar{\bm{U}}(t,X,z) = (\bar{U}, \bar{W})$ with $\bar{U} = (\bar{U}_1,\bar{U}_2)$ be the smooth background current,
characterized by the parameter $ \delta$ satisfying $0 < \delta < 1$. If $L$ is a typical length scale where $\mathcal{O}(1)$ changes in $U$ take place, we set $\delta = 1/L$ and enforce this by requiring $|D^\alpha \bar{U}_j| = \mathcal{O}(\delta^{|\alpha|})$ and $|D^\alpha \bar{W}| = \mathcal{O}(\delta^{|\alpha|})$ for any multi-index\footnote{Recall that for $\alpha = (\alpha_1,\alpha_2,\alpha_3)$,   $D^\alpha u = \partial_{x_1}^{\alpha_1}\partial_{x_2}^{\alpha_2}\partial_{z}^{\alpha_3} u$ and $|\alpha| = \sum \alpha_j.$} $\alpha$. 
We also assume that $\bar{\bm{U}}$ supports a background surface elevation $\bar{\eta}$ such that $\bar{\eta} \sim \delta$ and both $ |\partial_t \bar{\eta}| = \mathcal{O}(\delta)$ and $ |\nabla_X \bar{\eta}| = \mathcal{O}(\delta)$. Moreover, we assume the associated background pressure $\bar{P}$ satisfies $\bar{P}|_{z=\bar{\eta}} = P_{atm}$ for some given atmospheric surface pressure. 
\begin{figure}[htbp]
    \centering
    \includegraphics[width=0.9\linewidth]{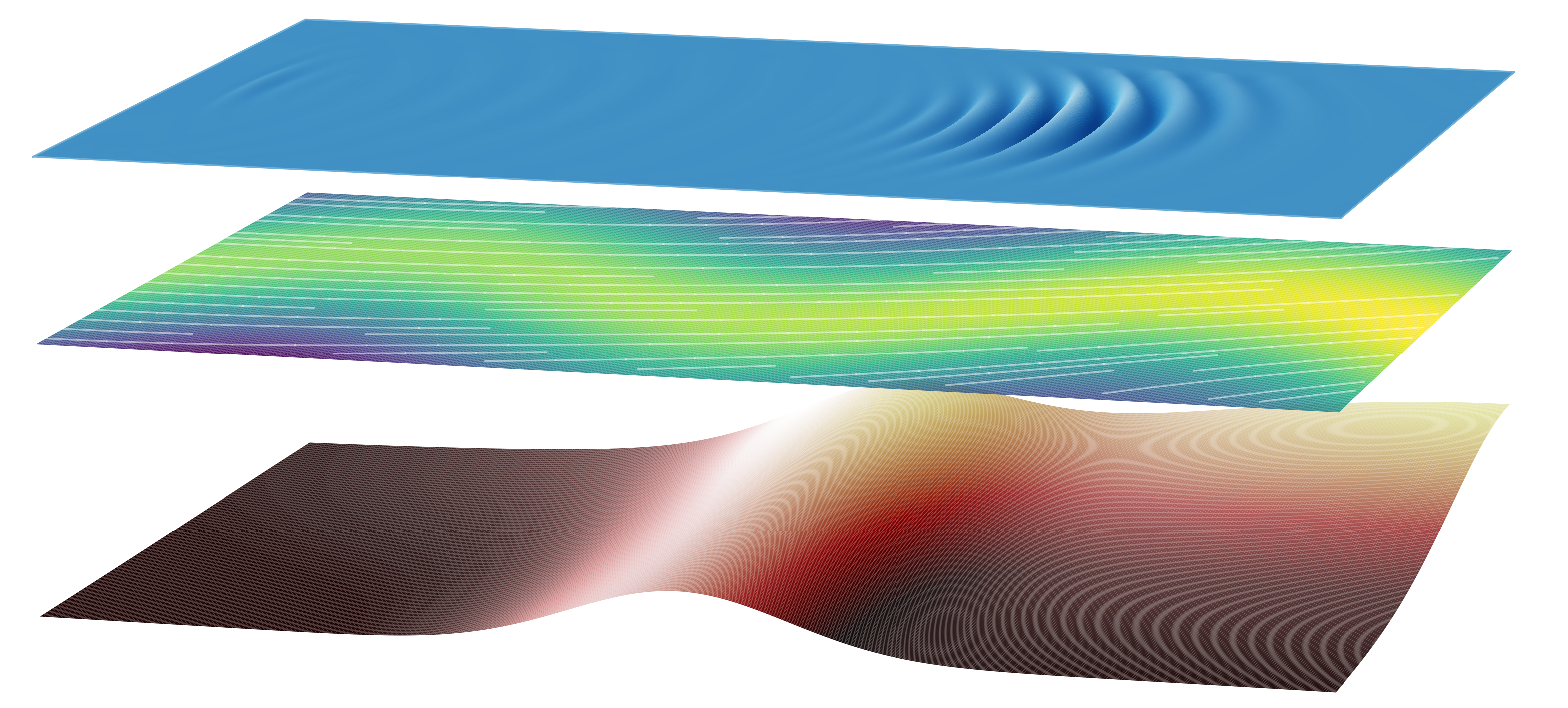}
    \caption{The figure illustrate the main components of our model; propagating waves (top), current (middle) and bathymetry (lower). Although the scales have been compressed for the purpose of illustration, the depicted current and bathymetry are used in the numerical simulation of the plotted wave in the numerical investigation of diffractive effects in \Cref{sect: diffraction}.}
    \label{fig: model}
\end{figure}

We now derive a linearized wave system around the mean surface $z = 0$, and use the shorthand notation $f_0 = f|_{z=0}$ for all quantities evaluated at $z = 0$. We consider a wave perturbation of the surface $\eta$ with corresponding velocity  $\bm{u} = (u,w)$ and pressure $p$. For a linearization parameter $0 < \varepsilon \ll 1 $, we assume small amplitude waves in the sense that $\eta \sim \varepsilon $, $ |\nabla_X \eta| \sim \varepsilon$ and $ |\bm{u}| = \mathcal{O}(\varepsilon).$ 
Taylor expanding $\bar{\bm{U}} + \bm{u} $ around $z = 0$  gives \[(\bar{\bm{U}} + \bm{u} )|_{z = \bar{\eta} + \eta} = (\bar{\bm{U}} + \bm{u})_0+ (\bar{\eta} +\eta)\partial_z (\bar{\bm{U}} + \bm{u} )_0+ \mathcal{O}(\varepsilon^2 +\delta^2).\]
The kinematic boundary condition linearized around $z=0$ consequently reads
\begin{equation}
    \partial_t \eta + \bar{U}_0\cdot \nabla_X \eta - w_0= -u_0\cdot \nabla_X \bar{\eta} + \eta (\partial_z \bar{W})_0+ \bar{\eta}(\partial_z w)_0+ \mathcal{O}(\varepsilon^2 + \delta^2).
    \label{eq: kin lin}
\end{equation}
We now consider the motion in the bulk. Linearizing the Euler equation \eqref{eq: euler} gives
\begin{equation}
    \partial_t \bm{u} + (\bar{\bm{U}}\cdot \nabla_{X,z})\bm{u}  = - \nabla_{X,z}p/\rho -(\bm{u}\cdot \nabla_{X,z})\bar{\bm{U}} + \mathcal{O}(\varepsilon^2), \quad \nabla_{X,z} \cdot \bm{u} = 0. 
\end{equation}
The following result shows that under the assumption that the vorticity of the background current is small and of finite extent in space, the wave perturbation $\bm{u}$ is a potential flow to leading order, i.e., $\bm{u} = \nabla_{X,z} \phi + \mathcal{O}(\delta \varepsilon)$ if it is initially a potential flow.  
\begin{proposition}
    Assume that $\nabla_{X,z} \times \bar{\bm{U}} = \bar{\bm{\omega}} = \mathcal{O}(\delta)$ and that there is some fixed $R \geq 0$ such that $\bar{\bm{\omega}} = 0$ for $|X| \geq R$. Moreover, assume the wave perturbation satisfies  $D^\alpha\bm{u} = \mathcal{O}(\varepsilon)$  for $|\alpha| \leq 3$ and $T = \mathcal{O}(1)$, and that it is initially irrotational, i.e, $\bm{\omega}|_{t=0} = 0$. For times $T = \mathcal{O}(1)$ it then holds  $\bm{u} = \nabla_{X,z} \phi + \tilde{\bm{u}}$, where $\nabla_{X,z} \cdot \tilde{\bm{u}} = 0$ and $\tilde{\bm{u}}$ is pointwise small in the sense that $\tilde{\bm{u}} = \mathcal{O}(\delta \varepsilon)$,$ \nabla_{X,z} \tilde{\bm{u}} = \mathcal{O}(\delta \varepsilon)$, and $\partial_t \tilde{\bm{u}} = \mathcal{O}(\delta \varepsilon)$ with constants depending on $R$ and $T$. 
    \label{prop: potpert}
\end{proposition}

\textbf{Note}: The validity of above result hinges on the apriori smallness assumption on the wave perturbation $\bm{u}$, typically valid for wavenumbers $k = \mathcal{O}(1)$. That such an assumption seems reasonable can for example be verified by throwing a rock in a calm river and watching the waves develop gently. In general, however, stability of the linearized Euler equations is highly non-trivial (cf. \cite{shvydkoy2005recent}). 
The proof can be found in the Appendix and relies on basic estimates for the linearized vorticity equation, the div-curl problem for $\tilde{\bm{u}}$ and boundary integral equation techniques. \\ 

By inserting $\bm{u} = \nabla_{X,z} \phi + \tilde{\bm{u}}$ and calculating the Bernoulli formulation, we get 
\begin{equation}
\begin{split}
\nabla_{X,z}\left(\partial_t \phi + \bar{\bm{U}} \cdot \nabla_{X,z} \phi + p \right) = {} & \nabla_{X,z} \phi \times \bar{\bm{\omega}} - \partial_t \tilde{\bm{u}} \\
& + (\bar{\bm{U}} \cdot \nabla_{X,z} )\tilde{\bm{u}} + (\tilde{\bm{u}} \cdot \nabla_{X,z} )\bar{\bm{U}}  + \mathcal{O}(\varepsilon^2)
\end{split}
\label{eq: grad bernoulli}
\end{equation}
and the incompressibility constraints $\Delta_{X,z} \phi = 0$ and $\nabla_{X,z} \cdot \tilde{\bm{u}} = 0$. 
Before we evaluate the above equation at $z = 0$ and derive our leading order system of equations, we make a structural observation. 
\begin{center}
\vspace{2mm}
\underline{\textbf{A leading order system}}
\vspace{2mm}
\end{center}
All terms appearing on the right hand sides of both equation \eqref{eq: kin lin} and \eqref{eq: grad bernoulli} are $\mathcal{O}(\varepsilon \delta)$ and can be truncated to yield a consistent leading order $\mathcal{O}(\varepsilon)$ system. However, in the simplified case when $\bar{\bm{U}}$ is a potential flow, appealing to Luke's variational principle (Ch.13, \cite{whitham2011linear}) shows that only one of these $\mathcal{O}(\varepsilon \delta)$ terms are needed to keep the variational structure intact for the $\mathcal{O}(\varepsilon)$ equations. To see this, we consider the Lagrangian density
\begin{equation}
    \mathcal{L} =  \int_{-b}^\eta \left(\partial_t \Phi + \frac{1}{2}|\nabla_{X,z}\Phi|^2 \right)\mathrm{d}z  + \frac{1}{2}g\eta^2. 
\end{equation}
Inserting $\bar{\bm{U}} = \nabla_{X,z}\Phi$ and $\bar{\eta}$ and linearizing for $\eta, \phi$ around $z = 0$ while discarding all $\mathcal{O}(\varepsilon \delta)$ terms gives
\[\mathcal{L}^{(2)} = \int_{-b}^0 \frac{1}{2}|\nabla_{X,z} \phi|^2 \mathrm{d}z + \eta(\partial_t \phi_0 + \bar{U}_0\cdot \nabla_X \phi_0) + \frac{1}{2}g\eta^2. \]
Computing the first variation of the action functional $\mathcal{S} = \int_{t_0}^{t_1} \int_{\R^2} \mathcal{L}^{(2)} \mathrm{d}X \mathrm{d}t$, we get the leading order $O(\varepsilon)$ system 
\begin{alignat*}{2}
    &\partial_t \eta + \nabla_X\cdot (\bar{U}_0\eta) = \partial_z \phi_0, \quad && \partial_t \phi_0 + \bar{U}_0\cdot \nabla_X \phi_0 = -g \eta, \\
    &\Delta_{X,z} \phi = 0 \text{ in } \Omega(0,b), \quad && \partial_\nu \phi|_{z=-b} = 0. 
\end{alignat*}
Comparing the kinematic condition \eqref{eq: kin lin} with the one above and using that incompressibility gives $\nabla_X \cdot \bar{U} = - \partial_z \bar{W}$, we recover the above kinematic equation by keeping the $\mathcal{O}(\varepsilon \delta)$ term $\eta \partial_z \bar{W}_0 = - \eta\nabla_X \cdot \bar{U}_0 $, while discarding the remaining $\mathcal{O}(\varepsilon \delta) $ terms. By the same logic, we recover the dynamic equation above from \eqref{eq: grad bernoulli} by discarding all $\mathcal{O}(\varepsilon \delta)$ terms and taking the surface trace (and using that $p_0 =  -g\eta + \mathcal{O}(\varepsilon \delta)$). 

Therefore, motivated by both structural reasons (adhering to the variational principle), and for practical modeling reasons (stripping the model for typically unknown data like $\bar{\eta}, \nabla_X \bar{\eta}$ and $\bar{\bm{\omega}}$), we take as a leading order linear system the equations

\begin{equation}
    \begin{cases}
        \partial_t \eta +  \nabla_X\cdot (\bar{U}_0 \eta) - \partial_z \phi_0= \mathcal{O}(\varepsilon \delta), \\
        \partial_t \phi_0+ \bar{U}_0\cdot \nabla_{X} \phi_0+ g\eta = \mathcal{O}(\varepsilon \delta),\\
         \Delta_{X,z} \phi = 0 \quad \text{in} \quad \Omega(0,b), \quad \partial_\nu \phi|_{z=-b} = 0.  
    \end{cases}
    \label{eq: linerized1}
\end{equation}
We now introduce the surface velocity potential $\varphi(t,X) = \phi|_{z=0}$ and the Dirichlet-to-Neumann (DN) operator $\mathcal{G}(b)$ defined by 
\begin{equation}
    \mathcal{G}(b)\varphi = \partial_z \phi|_{z=0}, \quad \text{where } \phi \text{ solves } \quad \begin{cases}
        \Delta_{X,z} \phi = 0 \quad \text{in} \quad  \Omega(0,b), \\
        \phi = \varphi \quad \text{at} \quad z = 0, \\
        \partial_\nu \phi = 0 \quad \text{at} \quad z = -b(X).
    \end{cases}
    \label{eq: DN-def}
\end{equation}
Writing $\bar{U} = \bar{U}_0$ we express \eqref{eq: linerized1} in terms of $\eta, \varphi, \bar{U}$ and $ \mathcal{G}(b)$:
\begin{equation}
    \begin{cases}
        \partial_t \eta +  \nabla_X\cdot(\bar{U} \eta) = \mathcal{G}(b)\varphi, \\
        \partial_t \varphi + \bar{U}\cdot \nabla_{X} \varphi = - g\eta .\\  
    \end{cases}
    \label{eq: linerized2d}
\end{equation}
This form is known as the Zakharov--Craig--Sulem formulation (cf. \cite{waterwavesprob}), and casts the wave-current interaction problem as a 1D$+$2D problem, at the cost of introducing the nonlocal DN operator. The above system is a 2D version of the 1D system considered in the influential paper \cite{longuet1961changes} that also includes time-dependent currents with weak vorticity and variable bathymetry. As we will see in \Cref{sect: asymptotics}, equation \eqref{eq: linerized2d} has significant explanatory powers despite its simplicity.

\section{Well-Posedness of the Wave-Current-Bathymetry System}

\label{sect: well-posed}
We now consider the question of well-posedness of the leading order system for wave-current-bathymetry interaction. We consider the Cauchy problem 
\begin{equation}
    \begin{cases}
    \partial_t \eta + \nabla_X \cdot(\bar{U} \eta)- \mathcal{G}(b)\varphi = 0, \\
    \partial_t \varphi +  \bar{U} \cdot \nabla_X \varphi + g\eta   = 0, \\
    \left(\eta(0,X), \varphi(0,X)\right) = \left(\eta_0,\varphi_0\right).
    \end{cases}
    \label{eq: cauchy}
\end{equation}
Above, $\left(\eta_0,\varphi_0\right)$ are suitable initial conditions for the wave and surface potential, and we investigate the existence, uniqueness and stability of solutions. Equation \eqref{eq: cauchy} is a coupled, nonlocal wave-transport system, and even though the equation is linear, the spatially varying coefficients and nonlocality of $\mathcal{G}(b)$ makes the analysis more involved. An operator perturbation argument for $\mathcal{G}$ allows us to symmetrize the principal term in the equation and treat it as a first order hyperbolic system with matrix valued pseudo-differential symbol, and from this, energy estimates and well-posedness of the PDE can be established. We obtain the following result: 

\begin{proposition}
 Assume $\bar{U}(t,X) \in C^\infty_b(\R \times \mathbb{R}^2;\mathbb{R}^2)$ and that the bathymetry \\ $b(X) \in C^\infty_b(\R^2)$ satisfies the assumptions in \Cref{sect: mathematical model}. Let $(\eta_0,\varphi_0) \in H^s \times H^{s+1/2}  $ for $s \in \mathbb{R}$. Then equation \eqref{eq: cauchy} has a unique solution $(\eta,\varphi) \in C(\mathbb{R};H^s \times H^{s+1/2})  $. 
 \label{prop: well-posed}
\end{proposition}

We now introduce the necessary tools and give a proof of the above result. 
\subsection{Pseudodifferential Analysis}

We briefly recall the definition of a pseudodifferential operator (PDO) (see, e.g., \cite{alinhac2007pseudo,taylor2013partial} for a proper introduction). For $m \in \mathbb{R}$ we denote by $S^m$ the set of  $a(X,\xi) \in C^\infty(\mathbb{R}^n \times \mathbb{R}^n)$ taking values in $\C$ and satisfying  
\[ |D_X^\alpha D_\xi^\beta a(X,\xi) | \leq C_{\alpha,\beta}(1 + |\xi|)^{m-|\beta|}, \quad \forall \alpha,\beta \in \mathbb{N}_0^n.\]
A function $a \in S^m$ is called a symbol, and we write $S^{-\infty} = \cap_{m} S^m$. We denote by $\text{Op}(a)$ the PDO associated with the symbol $a$. For suitable functions $g$, $\text{Op}(a)$ is defined as 
\[\text{Op}(a)g (X) =    \frac{1}{(2\pi)^n} \iint e^{i(X-Y)\cdot\xi} a(X,\xi) g(Y) \mathrm{d}Y\mathrm{d}\xi =\int e^{iX\cdot \xi} a(X,\xi) \widehat{g}(\xi) \mathrm{d}\xi, \]
where $\widehat{g}(\xi)$ is the Fourier transform of $g$. Note that this is the Kohn-Nirenberg (or standard) quantization of the symbol.  Moreover, the notation $\text{Op}(a) \in \mathrm{OPS}^m$ means $a \in S^m$, and for matrix valued pseudo-differential operators with symbol ${a_M = [a_{i,j}(X,\xi)]_{i,j = 1}^M}$ we say that \[a_M \in S^m \quad \text{if}  \quad  \| [D_X^\alpha D_\xi^\beta a(X,\xi)]_{i,j = 1}^M\| \leq C_{\alpha,\beta}(1 + |\xi|)^{m-|\beta|}, \quad \forall \alpha,\beta \in \mathbb{N}_0^n, \] where $\|\cdot\|$ is any proper matrix norm. 
We also allow for symbols with a simple, smooth time-dependence; for a function $a(t,X,\xi) \in C^\infty(\R^{2n + 1}) $, we say, with a slight notational overloading, that $a \in S^m$ if $\partial_t^k a(t,X,\xi) \in S^m$ for all $t \in \R$ and $k \in \N_0$.
We will also use the family of $L^2$-Sobolev spaces $H^s$ with $s\in \mathbb{R}$, defined as 
\[ H^s = \left\{ g \in \mathcal{S}': \int_{\mathbb{R}^2} (1 + |\xi|^2)^s|\widehat{g}(\xi)|^2 \mathrm{d}\xi \right\}, \]
where $\mathcal{S}'$ is the space of tempered distributions. Last, $C^\infty_b$ denotes the space of scalar or vector valued functions with the property that all derivatives are bounded and continuous. \\
\newline 
We proceed by constructing a useful modification of the constant depth DN operator. 
\begin{definition}
Fix $0 < \delta < 1 $, and let $h \in C^\infty(\mathbb{R})$ be a smooth characteristic function such that $h(\tau) \geq 0$, $h(\tau) = 0$  for $|\tau|> \delta$ and $h(\tau) = 1$ for $ |\tau| \leq \delta/2$. For a constant $B > 0$, set \begin{equation}
\gamma(\tau,B) = h(\tau) + \tau\tanh(B\tau) \quad \text{for } \tau \geq 0. 
\label{eq: mod symbol}
\end{equation}
We now define a family of PDOs as follows:
\[ \mathcal{G}^p(B) = \text{Op}( \gamma^p(|\xi|),B), \quad \text{for} \quad  p \in \mathbb{R}.\]
Also, we define the spectral cut-off operator $\mathcal{C}$ by $\mathcal{C}  = \text{Op}(h(|\xi|)) . $
\label{def: modDN}
\end{definition}
\noindent
We give some useful facts about the operators $\mathcal{G}^p$:  
\begin{lemma}
    The following holds for any $p\in \mathbb{R}$.
    \begin{itemize}
        \item[1:] The function $\gamma^p(|\xi|,B):\mathbb{R}^2 \to \mathbb{R}$ is a symbol and belongs to $S^p$. Consequently, \[\mathcal{G}^p(B)= \text{Op}(\gamma^p) :H^s \to H^{s-p} \quad \text{for all} \quad s \in \mathbb{R}.\] 
         \item[2:] For all $\alpha, p\in \mathbb{R}$, $\mathcal{G}^\alpha(B) \mathcal{G}^p(B)= \mathcal{G}^{\alpha + p}(B)$. In particular,  $\mathcal{G}^p(B): H^s \to H^{s-p}$ is invertible with inverse $\mathcal{G}^{-p}(B)$. 
         \end{itemize}
         \label{prop: G_pert}
\end{lemma}
We proceed with a key result that allow us to split the DN operator into a sum of a constant depth operator and an infinitely smoothing operator that depends on the depth. This result is a version of the fact that the DN operator depends analytically on the bathymetry (cf. \cite{waterwavesprob} Ch. 3). 
\begin{proposition}
Let the bathymetry $b$ be as in \Cref{sect: mathematical model} and $\varphi \in H^s(\R^2)$. Then 
\[ \mathcal{G}(b)\varphi = \mathcal{G}(b_{\text{\tiny{max}}})\varphi + \mathcal{K}(b)\varphi ,\quad \text{where} \quad \mathcal{K}(b) \in \mathrm{OPS}^{-\infty}.\]
\label{prop: G + K}
\end{proposition}
The above result follows from analyzing the difference $\mathcal{G}(b)\varphi - \mathcal{G}(b_{\text{\tiny{max}}})\varphi$ using boundary integral equations, and its proof, together with that of  \Cref{prop: G_pert}, can be found in the Appendix.

\subsection{Analysis of the Wave System}
\label{sect: well-posed analysis}
Using \Cref{prop: G + K} and the fact that $\mathcal{G}^1 = \mathcal{G}(b_{\text{\tiny{max}}}) + \mathcal{C}$, we have that $\mathcal{G}(b) = \mathcal{G}^1 - \mathcal{C}+ \mathcal{K}(b)$, where $\mathcal{C}, \mathcal{K}(b) \in \mathrm{OPS}^{-\infty}$. We define
\begin{equation}
    V = \begin{bmatrix}
        \eta \\ \varphi 
    \end{bmatrix}, \quad A = \begin{bmatrix}
        -\bar{U}\cdot \nabla_X -  \nabla_X \cdot \bar{U}  \quad & -\mathcal{C} + \mathcal{K}(b) \\ 0 \quad  &-\bar{U}\cdot \nabla_X 
    \end{bmatrix}, \quad
    B = \begin{bmatrix}
        0 \quad &\mathcal{G}^1 \\
        -g \quad &0
    \end{bmatrix},
\end{equation}
 and write \eqref{eq: cauchy} as  
\begin{equation}
    \partial_t V = AV + BV, \quad V(0) = V_0 = (\eta_0,\varphi_0)^\top.
    \label{eq: U sys}
\end{equation}
We want to diagonalize the above system, and we start by diagonalizing $B$. Due to proposition \Cref{prop: G_pert}, we can write $B = S\Lambda S^{-1}$ with
\begin{equation}
\Lambda = \begin{bmatrix}
        -i\sqrt{g}\mathcal{G}^{1/2} \quad &0 \\
        0 \quad &i\sqrt{g}\mathcal{G}^{1/2}
    \end{bmatrix}, \quad S = \begin{bmatrix}
        i\frac{\mathcal{G}^{1/2}}{\sqrt{g}} \quad &  -i\frac{\mathcal{G}^{1/2}}{\sqrt{g}} \\ 1 \quad &1 
    \end{bmatrix}, \quad  S^{-1} = \frac{1}{2}\begin{bmatrix}
        -i \sqrt{g}\mathcal{G}^{-1/2}\quad  & 1  \\
        i \sqrt{g}\mathcal{G}^{-1/2} \quad &1 
    \end{bmatrix}. 
\end{equation}
We now set $Q = S^{-1}V$. Applying $S^{-1}$ to \eqref{eq: U sys}, we get    
\begin{equation}
     \frac{\partial}{\partial t} Q = (\Lambda + S^{-1}AS)Q, \quad Q(0) = S^{-1}V_0. 
     \label{eq: diag}
\end{equation}
The operator $\widetilde{A}=S^{-1}AS$ takes the form $\widetilde{A} = \widetilde{A}_1 + \widetilde{A}_0$ with
\begin{align*}
\widetilde{A}_1 &= \begin{bmatrix}
   -\frac{1}{2}\mathcal{G}^{-1/2}\bar{U}\cdot \nabla_X \mathcal{G}^{1/2} - \frac{1}{2}\bar{U}\cdot \nabla_X \quad 
   &\frac{1}{2}\mathcal{G}^{-1/2}\bar{U}\cdot \nabla_X \mathcal{G}^{1/2} - \frac{1}{2}\bar{U}\cdot \nabla_X \\
   \frac{1}{2}\mathcal{G}^{-1/2}\bar{U}\cdot \nabla_X \mathcal{G}^{1/2} - \frac{1}{2}\bar{U}\cdot \nabla_X \quad 
   &-\frac{1}{2}\mathcal{G}^{-1/2}\bar{U}\cdot \nabla_X \mathcal{G}^{1/2} - \frac{1}{2}\bar{U}\cdot \nabla_X  
\end{bmatrix}, \\
\widetilde{A}_0 &= \begin{bmatrix}
    -a_0 \quad &\overline{a}_0 \\ \hspace{2mm}a_0 \quad &-\overline{a}_0
\end{bmatrix}, \quad  a_0(X,\xi) =\mathcal{G}^{-1/2} \nabla_X \cdot \bar{U} \mathcal{G}^{1/2} + i\sqrt{g} \mathcal{G}^{-1/2}(-\mathcal{C} + \mathcal{K}(b)) 
\end{align*}
Following  Ch.  7.7 in \cite{taylor2013partial} we say that a system 
\[ \partial_t V = K V \]
is symmetric hyperbolic if the $n\times n$ matrix-valued, time-dependent PDO $K \in \mathrm{OPS}^1$ and $K^* + K \in \mathrm{OPS}^0$. Here $K^*$ denotes the Hermitian adjoint of $K$ with respect to the $L^2(\mathbb{R}^n;\mathbb{C}^d)$ inner product. For a scalar PDO $\text{Op}(a) \in \mathrm{OPS}^m$, the symbol of $\text{Op}(a)^*$ is given by $a^* =  \bar{a} + r(X,\xi)$ for $r \in S^{m -1}$, where $\bar{a}$ denotes the complex conjugate. Hence, the statement that $K^* + K \in \mathrm{OPS}^0$ can be stated as $K = K_1 + K_0$, where $K_1 \in \mathrm{OPS}^1$ is symmetric and has a purely imaginary symbol, while $K_0 \in \mathrm{OPS}^0$. The following proposition shows that $\Lambda + \widetilde{A}$ is symmetric hyperbolic. 

\begin{proposition}
    Let $\kappa_\pm(X,\xi) = i(\bar{U}\cdot\xi \pm \sqrt{g}\gamma^{1/2}(|\xi|))$. The operator $K = \Lambda + \widetilde{A}$ takes the form
    \[ K = \begin{bmatrix}
        \text{Op}(\kappa_+) \quad & 0 \\
        0 \quad &\text{Op}(\kappa_-)
    \end{bmatrix} + R, 
    \]
    where $R  \in \mathrm{OPS}^0$. 
    \label{prop: imag}
\end{proposition}

\begin{proof}
For $a \in S^{m_1}, b \in S^{m_2} $, the symbol $c$ of the composition $C = \text{Op}(a)\text{Op}(b)$ takes the form $c = a b + r$ for some $r \in S^{m_1 + m_2 -1}$. 
Writing  $\text{Op}(u) = \bar{U}\cdot \nabla_X$, we have $u(t,X,\xi) = -i\bar{U}(t,X)\cdot\xi \in S^1$. The diagonal terms $\kappa$ of $\widetilde{A}_1$ has the form
\[\text{Op}(\kappa) = -\frac{1}{2}\mathcal{G}^{-1/2}\text{Op}(u) \mathcal{G}^{1/2} - \frac{1}{2}\text{Op}(u), \]
and we therefore get
\begin{align*} \kappa &= -\frac{1}{2}\gamma^{-1/2}(|\xi|)u(t,X,\xi)\gamma^{1/2}(|\xi|) -\frac{1}{2} u(t,X,\xi) + r(t,X,\xi) \\
&= -i\bar{U}(t,X)\cdot\xi + r(t,X,\xi), \quad r(t,X,\xi) \in S^0.  
\end{align*}
By the same computation, the off-diagonal terms of $\widetilde{A}_1$ have symbol $r \in S^0$. Moreover, both $\mathcal{C}$ and $\mathcal{K}(b)$ are in $\mathrm{OPS}^{-\infty}$ and $\mathcal{G}^{-1/2} \nabla_X \cdot \bar{U} \mathcal{G}^{1/2} \in \mathrm{OPS}^0$, and therefore $\widetilde{A}_0 \in \mathrm{OPS}^0$. As the diagonal terms of $\Lambda$ are 
$\pm i\sqrt{g}\gamma^{1/2}(|\xi|) $, the result now follows. \end{proof}

Existence and uniqueness of \eqref{eq: U sys} now follows from energy estimates for hyperbolic systems. The following result can be found in  Ch. 7.7 in \cite{taylor2013partial}, or, with a less compressed proof, in Ch. 8 in \cite{hintzmicro}. 

\begin{proposition}
Let $K $ be $n\times n$ matrix valued PDO with smooth time dependence and such that $K \in \mathrm{OPS}^1$ and $K^* + K \in \mathrm{OPS}^0 $. Assume $V_0 \in H^s $ (here $H^s = H^s(\R^n;\C^n) $ denotes the space of vector-valued $H^s$ functions)  and $f \in C(\mathbb{R};H^s)$. Then there exists a unique solution $V$ to the problem 
\[\partial_t V = KV + f,  \quad  V(0) = V_0 , \]
and $V \in C(\mathbb{R};H^s)$ satisfies $\|V(t)\|_{H^s} \leq C(t)\left(\|U_0\|_{H^s} + \|f\|_{C([0,t];H^s)} \right)$. 
\label{prop: hyperbolic}
\end{proposition}

The well-posedness of the Cauchy problem \eqref{eq: cauchy} now follows.  
\begin{proof}
(\Cref{prop: well-posed})
We consider the diagonalized version of \eqref{eq: cauchy}, 
\[ \partial_t Q = (\Lambda + \widetilde{A}) Q, \quad Q(0) = Q_0. \]
We have $Q_0 = S^{-1}V_0 \in H^{s+1/2}\times H^{s+1/2}$, and so by \Cref{prop: imag} and the assumption that $\bar{U} \in C^\infty_b(\R^{n+1};\R^2)$, it follows that this system satisfies the assumptions of \Cref{prop: hyperbolic}. Hence there exists a unique solution $Q \in C(\mathbb{R};H^{s+1/2}\times H^{s+1/2})$. Since $\mathcal{G}^{1/2}$ is bijective, we recover $V$ uniquely:   
\[ \begin{bmatrix}
    \eta \\ \varphi
    \end{bmatrix} = V = SQ = \begin{bmatrix}
        &\frac{i}{\sqrt{g}}\mathcal{G}^{1/2}(Q_1 -Q_2) \\
        & Q_1 + Q_2 
    \end{bmatrix}, \]
Hence we loose one half derivative in Sobolev space regularity for $\eta$ relative to $\varphi$ and get $(\eta,\varphi) \in  C(\mathbb{R};H^{s}\times H^{s+1/2})$. 
\end{proof}

\section{Asymptotic Analysis: Energy, Wave Action and Diffraction}
\label{sect: asymptotics}
Having established the general properties of the wave-current-bathymetry system, we now want to develop an understanding of how waves evolve in non-uniform environments. To this end, we conduct a Wentzel–Kramers–Brillouin (WKB) analysis (cf. \cite{whitham2011linear,buhler2014waves}) of the system \eqref{eq: linerized2d}. The WKB method provides a robust framework for studying waves where the medium varies slowly compared to the wavelength. The underlying premise is that in a linear regime, waves locally resemble plane waves even as they adapt to their surroundings. This is captured by the wave packet ansatz:
\[\eta(t,X) = e^{iS(t,X)}A(t,X).\] Here, $S$ is a phase function and $A$ is an amplitude (or envelope) function of the wave packet. The real valued wave amplitude is   $\eta = \operatorname{Re}\left\{e^{iS}A \right\}$, but working with with complex valued wave simplifies notation and calculations. Defining the local frequency and wavenumber vector as 
\begin{equation} \omega(t,X) = -\partial_t S(t,X), \quad \bm{k}(t,X) = \nabla_X S(t,X),
\label{eq: freq wavenumber}
\end{equation}
it is clear that for well-behaved $S$ and $A$, $\eta$ approximates a plane wave solution locally, i.e.,  $\eta(t,X) \approx e^{i(\bm{k} \cdot X - \omega t) + i\alpha}A_{\text{const.}}$, with $\alpha \in[0,2\pi)$. For this observation to be reasonable, we want the change of $A$ and $S$ to be small over the distance of a wavelength $\lambda = \frac{2\pi}{|\bm{k}|}$. Therefore, if $L$ is a typical length scale where the current or depth changes significantly, we want $\lambda/L \ll 1$. Hence, we introduce the small parameter\footnote{$\mu$ must not be confused with the shallow water parameter $\mu_s = H/L $(or $(H/L)^2$. } $\mu = \lambda /L$ and assume $0<\mu \ll 1 $. Defining the slow coordinates $X' = \mu X $ and $t' = \mu t$, the slow variation in $A, \bm{k}$ and $\omega$ is achieved by writing 
\[ e^{iS(t',X')}A(t',X') = e^{i\tilde{S}(\mu t,\mu X)/\mu}A(\mu t,\mu X), \quad \text{for} \quad \tilde{S},A \in C^\infty_b(\R^2).\]
\noindent
We enforce the coordinate transform by scaling the differential operators (omitting the superscript)  
\[ \partial_{t} \mapsto \partial_{t}^\mu =  \mu \partial_t  \quad \text{and} \quad \partial_{x_j} \mapsto \partial_{x_j}^\mu = \mu \partial_{x_j}, \]
and we assume that the slowly varying current $\bar{\bm{U}}$ and bathymetry $b$ scale with $\mu$. 

To proceed, we also need to find a suitable asymptotic expression of the DN operator that respects the underlying mathematical/physical structure, as this is a key component in the asymptotic analysis of equation \eqref{eq: linerized2d}, and the next subsection is dedicated to this. We then proceed to derive several important models from the asymptotic version of equation \eqref{eq: linerized2d}.

\subsection{Asymptotics of the DN-operator}
When the bottom is flat, $b(X) = b_0$, an explicit representation of the DN operator is given by the Fourier multiplier $\mathcal{G}(b_0) \varphi = |D|\tanh(b_0|D|)\varphi$. It is therefore natural to try to generalize this expression to slowly varying depth $b(X)$. Hence we define the symbol $g_b(X,\xi) = |\xi|\tanh(b(X)|\xi|)$. There are several ways to associate the symbol $g_b$ with an operator, and the procedure of constructing an operator from a symbol is called quantization \cite{zworski2012semiclassical}. The most common quantizations are the Kohn-Nirenberg (or standard) and Weyl quantizations.
To also incorporate the asymptotic scaling by the slow parameter $\mu > 0$, we will use the so-called semiclassical scaling of the associated PDOs. 
The semiclassical Kohn-Nirenberg quantization of the DN operator takes the form 
\begin{equation}
    \Gkn^\mu(b) \varphi (X) = \frac{1}{(2\pi\mu)^2} \iint e^{i(X-Y)\cdot\xi/\mu} g_b(X,\xi) \varphi(Y) \mathrm{d}Y\mathrm{d}\xi, 
\end{equation}
while the semiclassical Weyl quantization is
\begin{equation}
    \Gw^\mu(b) \varphi (X) = \frac{1}{(2\pi\mu)^2} \iint e^{i(X-Y)\cdot\xi/\mu} g_b\left(\dfrac{X+Y}{2},\xi\right) \varphi(Y) \mathrm{d}Y\mathrm{d}\xi. 
\end{equation}
The semiclassical scaling implies that $\Op^\mu((i\xi)^\alpha) = \mu^{|\alpha|}D^\alpha $, i.e., it scales derivatives by $\mu$, which is what we want for the DN operator. However, the two quantizations differ in a crucial way that ultimately dictates our choice for the asymptotic analysis:
\begin{center}
\vspace{2mm}
\begin{minipage}{0.88\textwidth}
A central property of the DN operator  $\mathcal{G}(b)$ is its self-adjointness (Ch. 3 and A.2, \cite{waterwavesprob}). This property is, for example, essential for establishing energy conservation for \eqref{eq: linerized2d} in the absence of currents. As our asymptotic analysis focuses on energy dynamics, we require an approximation of $\mathcal{G}(b)$ that preserves this self-adjoint structure. Since the Weyl quantization of real-valued symbols is (formally) self-adjoint on $L^2$, $\Gw^\mu(b)$ therefore becomes the natural choice.
\end{minipage}
\vspace{2mm}
\end{center}

We note that the Weyl quantization of the DN operator also appears in the works \cite{akrish2023mild, smit2013evolution}, although without rigorous justification. Both the Weyl and Kohn-Nirenberg quantizations are quite accurate approximations to the true operator for slowly varying bathymetry, and we refer to Ch. 4 in \cite{zworski2012semiclassical} on the self-adjointness of Weyl quantization (and the lack thereof for Kohn-Nirenberg). Moreover, the Weyl quantization has additional approximation properties that are favorable in asymptotic analysis (see  \Cref{prop: w-expansion}).

The next proposition makes this quantitative. Below, $\mathcal{G}^\mu(b)$ is the DN operator associated with the (horizontally) scaled version of the boundary value problem for $\varphi$ in \eqref{eq: DN-def}, with  $\nabla_X^\mu = \mu \nabla_X$, $\Delta_{X,z}^\mu = \mu^2 \Delta_X + \partial^2_z$ and\footnote{We omit the normalization in the normal derivative, as the Neumann condition is either 0 or equal to some function with the same normalization.} $\partial_{\nu}^\mu = \partial_z +  \nabla_X^\mu b(X) \cdot \nabla_X^\mu.$

\begin{proposition}
Let $0<\mu < 1$ and assume the depth function $b$ satisfies $b_{\text{\tiny{min}}} \geq 1$ and $\|\nabla_X b\|_{L^\infty} \leq \delta $ and $\Delta_X b = \mathcal{O}(1)$, and write $M(b) = \frac{\delta(1+ \delta)}{b_{\text{\tiny{min}}}}$. For $\varphi \in H^1\cap C^1_b$ it holds that 
\[ \|\mathcal{G}^\mu(b)\varphi - \Gw^\mu(b)\varphi \|_{L^\infty} 
 \leq C\mu^2 M(b) \|\nabla_X \varphi \|_{L^\infty} + \mathcal{O}(\mu^3).  \] 
Moreover, $\Gw^\mu(b)$ is self-adjoint on $L^2$ with domain $H^1$. 
\label{prop: G est}
\end{proposition}
\noindent
The proof is quite lengthy and can be found in the Appendix. \\
\newline
\textbf{Note:} 
Large bottom variations have small effects on the surface when they appear at a depth sufficiently large, and the above result reflects this. In coastal wave modeling, it is common to assume a so-called mild slope condition. For $0 < \delta \ll 1$, it is usually stated as $\nabla_X b/b \leq |\bm{k}|\delta$, where $|\bm{k}|$ is the local wavenumber, and invoking this assumption gives rise to the mild-slope equation (cf. \cite{mei1989applied}).\\

We now continue towards the WKB analysis of \eqref{eq: linerized2d}, and the following result gives the action of $\Gw^\mu(b)$ on a wave packet.
\begin{proposition}
    Let $A,S \in C^\infty_0$ and $\eta = Ae^{iS/\mu}$. For $\bm{k} = \nabla_X S$, we write 
    \[ \sigma(X,\bm{k}) = \sqrt{g|\bm{k}|\tanh(b(X)|\bm{k}|)}, \quad C_g(X,\bm{k}) = \nabla_{\bm{k}} \sigma(X,\bm{k}). \]
We then have
\begin{equation}
    g\Gw^\mu(b) \eta = e^{iS/\mu}\left(\sigma^2 A - i\mu  2\sigma C_g \cdot \nabla_X A - i\mu \left(\nabla_X \cdot \sigma C_g \right) A  \right)+ \mathcal{O}(\mu^2).
\end{equation}

\label{prop: w-expansion}
\end{proposition}

\noindent
\textbf{Note:} The term $\sigma$ is often referred to as the \emph{intrinsic} frequency of the waves, while the term $C_g$ is the group velocity \cite{buhler2014waves}. The term $\frac{i\mu }{2}\left(\nabla_X \cdot C_g \right)$ in the above expansion is unique to the Weyl quantization of the DN operator.
\begin{proof}
In the following, $a = b  + \mathcal{O}_{S^m}(\mu^k)$ will mean that  $ a = b + \mu^k r$ for $r\in S^m$, and $\Op(a) = \Op(b) + \mathcal{O}_{S^m}(\mu^k) $ means $\Op(a) = \Op(b) + \mu^k\Op(r)$.
We first note that for  $S \in C^\infty_0$ and real valued and $a \in S^m$ with $m \leq 1$, we have
        \[ e^{-iS/\mu} \Opw^\mu(a) e^{ iS/\mu} u = \Opw^\mu(p) u + \mathcal{O}_{S^{m-2}}(\mu^2) \quad \text{with} \quad  p(X,\xi) = a(X,\xi + \nabla_X S). \]
This result can be found in (the proof of) Theorem 10.6 in \cite{zworski2012semiclassical} and is obtained Taylor expanding the phase with $m = (X+Y)/2$  and $d = X -Y$, so that $S(X) - S(Y) = -d \cdot \nabla_X S(m) + \mathcal{O}(|d|^3)$. 
 We now use that for a symbol $a$, the Weyl and Kohn-Nirenberg quantizations are related through the formula  $Op_{\text{\tiny{W}}}^\mu(a) = Op_{\text{\tiny{KN}}}^\mu(e^{-\frac{i\mu}{2}(\nabla_X,\nabla_\xi)} a)$ (Ch. 4,\cite{zworski2012semiclassical}). Hence \[ \Opw^\mu(a) = Op_{\text{\tiny{KN}}}^\mu(a - \frac{i\mu}{2}(\nabla_\xi \cdot \nabla_X) a ) + \mathcal{O}(\mu^2)_{S^{m-2}}.\]
 The action of $\operatorname{Op}_\text{\tiny{KN}}(a)$ on a smooth function $u$ in terms of differential operators follows by standard Taylor expansion of the symbol (cf. Theorem 9.5 in \cite{zworski2012semiclassical}), and we have  
\begin{equation}  \operatorname{Op}_\text{\tiny{KN}}(a) u = \sum_{|\alpha| \leq N} \frac{1}{\alpha!} \partial^\alpha_\xi a(X,\xi)|_{\xi = 0}(\mu D_X)^\alpha u(X) + \mathcal{O}_{S^{m - N -1}}(\mu^{N+1}).
\label{eq: KN expansion}
\end{equation}
Using this and the formula for $ e^{-iS/\mu} \Opw^\mu(a) e^{ iS/\mu} u$ and keeping terms of order $\mu^1$ or lower, we get
\begin{align*} 
g\Gw^\mu(b) \eta &= ge^{iS/\mu}\bigg( g_b(X,\nabla_X S + \xi)|_{\xi = 0} A  - i \mu (\nabla_\xi g_b(X,\nabla_X S + \xi)|_{\xi = 0} )\cdot \nabla_X A \\
& - \frac{i\mu}{2} \left(\nabla_X \cdot( \nabla_\xi g_b(X,\nabla_X S + \xi)|_{\xi = 0})\right) A\bigg)  + \mathcal{O}(\mu^2),\end{align*}
where the result holds pointwise ($\mathcal{O}_{S^{-1}}(\mu^2) \to \mathcal{O}(\mu^2) $)  by \Cref{prop: pointwise} (in the Appendix) for smooth WKB-functions. 
\end{proof}
\noindent
We are now in a position to study the asymptotic properties of the water waves system \eqref{eq: linerized2d}. 

\subsection{Asymptotic Energy Dynamics}
\label{sect: energy and action}
An important feature of the WKB analysis of wave phenomena is that it provides a simple, interpretable approximation of the wave energy dynamics. However, the presence of a variable current and bathymetry complicates the situation, and for water waves, energy dynamics is usually studied using the concept of wave action. We now briefly recall some elementary notions of energy for water waves and establish an evolution equation for the non-asymptotic energy density. From this equation, we first compute an evolution equation for the total energy, and second, we derive an equation for the asymptotic energy density. Last, we revisit the wave action equation and show that it agrees with the asymptotic energy equation.  \\

In the absence of a current, total kinetic energy of the water is given by 
\[E_{K} =  \frac{\rho}{2}\int_{\mathbb{R}^2} \int_{-b(X)}^\eta |\nabla_{X,z} \phi |^2 \mathrm{d}z \mathrm{d}X = \frac{\rho}{2}\int_{\mathbb{R}^2} \varphi \mathcal{G}(b)\varphi  \mathrm{d}X + \mathcal{O}(\varepsilon^2).  \]
Consequently, the kinetic energy density is $e_K = \varphi \mathcal{G}(b)\overline{\varphi} + \mathcal{O}(\varepsilon^2)$. 
Similarly, the total potential energy of the wave is 
\[\frac{\rho g}{2}\int_{\mathbb{R}^2} \int_{0}^\eta z \mathrm{d}z \mathrm{d}X = \frac{\rho g}{2}\int_{\mathbb{R}^2}  \mathrm{d}X + \mathcal{O}(\varepsilon^2)\]
Hence, the potential energy density is $e_P = \frac{g\rho}{2} |\eta|^2+\mathcal{O}(\varepsilon^2)$. In the continuation, we set the constant density $\rho = 1$ and denote the total energy density by $\mathcal{E} =    \frac{1}{2}\left(g |\eta|^2 + \varphi\mathcal{G}\overline{\varphi}\right)$. Note that we allow for complex $\eta$ and $\varphi$.  \\
We now derive a equation for  $\mathcal{E}$ from the system \eqref{eq: linerized2d}.
We write $L = \bar{U}\cdot \nabla_X$ and $\mathcal{G} = \mathcal{G}(b)$ and $D = \nabla_X \cdot \bar{U}$, and compute 
\begin{align*}
\partial_t \mathcal{E} &= \frac{1}{2}\left(g(\overline{\eta}\partial_t \eta  +  \eta \partial_t \overline{\eta}) + \partial_t \varphi \mathcal{G}\overline{\varphi} + \varphi \mathcal{G}\partial_t\overline{\varphi} \right) = \frac{1}{2}\left( g( (- L\eta - D \eta + \mathcal{G}\varphi) \overline{\eta} \right) \\
&+ \frac{1}{2}\left( g\eta (-L\overline{\eta} - D\overline{\eta}+ \mathcal{G}\overline{\varphi})) + (-L\varphi - g\eta)\mathcal{G}\overline{\varphi} + \varphi \mathcal{G}(-L\overline{\varphi} - g\overline{\eta}) \right) \\
& = \frac{1}{2}g\left( \overline{\eta}\mathcal{G}\varphi - \varphi \mathcal{G}\overline{\eta}\right)- \nabla_X \cdot (\bar{U}\frac{1}{2}g  |\eta|^2) - D \frac{1}{2}g|\eta|^2 - \frac{1}{2}\mathcal{G}\overline{\varphi}L\varphi - \frac{1}{2}\varphi \mathcal{G} L\overline{\varphi}.
\end{align*} 
Using the commutator $[L,\mathcal{G}] = L\mathcal{G} - \mathcal{G}L$, we write $\varphi\mathcal{G} L\overline{\varphi} = \varphi L \mathcal{G}\overline{\varphi} - \varphi [L,\mathcal{G}] \overline{\varphi}.$
After rearrangement we now have 
{\small
\begin{equation}
    \partial_t \mathcal{E} + \nabla_X \cdot (\bar{U} \mathcal{E})= \frac{1}{2}g\left( \overline{\eta}\mathcal{G}\varphi - \varphi \mathcal{G}\overline{\eta}\right) +\nabla_X \cdot \bar{U}\left(\frac{1}{2}\varphi \mathcal{G}\overline{\varphi} -\frac{1}{2}g  |\eta|^2\right)  + \frac{1}{2}\varphi [\bar{U}\cdot \nabla_X,\mathcal{G}] \overline{\varphi}.
    \label{eq: energy density}
\end{equation} }
This equation has some interesting consequences, which we now explore. 
\begin{center}
\vspace{2mm}
\underline{\textbf{Evolution of total energy}}
\vspace{2mm}
\end{center}
An interesting first interpretation of this equation is found by considering the total energy $\int_{\R^2} \mathcal{E} \mathrm{d}X$. In the following we assume all quantities are real and have sufficient smoothness and decay at infinity. Integrating \eqref{eq: energy density} over $\R^2$, we first note that the term $\int_{\R^2} \left( \eta\mathcal{G}\varphi - \varphi \mathcal{G}\eta\right) \mathrm{d}X = 0$ due to $\mathcal{G}$ being self-adjoint. Next, integrating by parts and using the self-adjointness of $\mathcal{G}$ gives
\[ \int_{\R^2} \varphi [L,\mathcal{G}] \overline{\varphi} \mathrm{d}X = -2\int_{\R^2} (\bar{U}\cdot \nabla_X \varphi) \mathcal{G} \varphi  \mathrm{d}X -  \int_{\R^2} \nabla_X \cdot \bar{U} \varphi \mathcal{G}\varphi \mathrm{d}X.   \]
Recalling now our bulk potential $\phi$ with $\bm{u} = \nabla_{X,z} \phi$, the divergence theorem gives 
\[ \int_{\Omega(0,b)} \nabla_{X,z} \cdot \left(( \bar{\bm{U}} \cdot \bm{u} )\bm{u}\right) \mathrm{d}z \mathrm{d}X = \int_{\R^2}  (\bar{U}\cdot \nabla_X \varphi) \mathcal{G} \varphi \mathrm{d}X + \mathcal{O}(\varepsilon^2\delta).\]
Expanding the bulk integrand gives after some vector calculus
\[ \nabla_{X,z} \cdot \left(( \bar{\bm{U}} \cdot \bm{u} )\bm{u}\right) = \nabla_{X,z} \cdot \left(\frac{1}{2}|\bm{u}|^2 \bar{\bm{U}}\right) + \bm{u} \cdot S(\bar{\bm{U}})\bm{u}, \quad S(\bar{\bm{U}}) = \frac{1}{2}\left( \nabla_{X,z} \bar{\bm{U}} + (\nabla_{X,z} \bar{\bm{U}})^\top \right),\]
where $S(\bar{\bm{U}}) $ is the strain tensor of $\bar{\bm{U}}$. Inserting these expressions into the integrated \eqref{eq: energy density}, we apply again the divergence theorem and find  $ \int_{\Omega(0,b)} \nabla_{X,z} \cdot \left(\frac{1}{2}|\bm{u}|^2\bar{\bm{U}}\right) = \mathcal{O}(\varepsilon^2 \delta)$. Next, the $\nabla_X \cdot \bm{U} \varphi \mathcal{G}\varphi$ terms cancel, and the integral of $\nabla_X \cdot (\bar{U}\mathcal{E})$ vanish, and we are left with
\begin{equation*}
    \frac{\mathrm{d}}{\mathrm{d}t}\int_{\R^2} \mathcal{E} \mathrm{d}X = -\int_{\R^2}\nabla_X \cdot \bar{U} \frac{1}{2}g|\eta|^2 \mathrm{d}X - \int_{\Omega(0,b)}  \bm{u} \cdot S(\bar{\bm{U}})\bm{u} \mathrm{d}z \mathrm{d}X + \mathcal{O}(\varepsilon^2\delta)
\end{equation*}
In turbulence theory, the term $\mathcal{P} = -\bm{u} \cdot S(\bar{\bm{U}})\bm{u}$ is known as the production term (cf.  Ch. 5.3 in \cite{Pope2000TurbulentFlows} or  Ch. 5.6 in \cite{henningson2001StabilityTransition}), and is the term responsible for energy transfer from the mean flow/background flow to the turbulent/oscillatory motion and vice versa. The above equation therefore shows that the change in total wave energy results from two sources, each having a clear interpretation. 
\vspace{2mm}
\begin{itemize}
    \item[$\circ$] The term $-\int_{\R^2}\nabla_X \cdot \bar{\bm{U}} \frac{1}{2}g|\eta|^2 \mathrm{d}X$ represents geometric effect of the surface current in the wave amplitude.  Since $\eta$ satisfies the continuity equation $\partial_t \eta + \nabla_X\cdot (\bar{U} \eta) = \mathcal{G}\varphi$, we get (after momentarily discarding the DN term) the classical total energy evolution 
    \[\frac{\mathrm{d}}{\mathrm{d}t} \int_{\R^2} \eta^2 \mathrm{d}X = - \int_{\R^2} \nabla_X \cdot \bar{U} \eta^2 \mathrm{d}X.\]
    Locally, this means that a diverging current ($\nabla_X \cdot \bar{U} > 0$) spreads out the wave, leading to a decrease in amplitude and hence a decrease in potential energy, and vice versa for a converging current ($\nabla_X \cdot \bar{U} < 0$). 
    \item[$\circ$] The term  $\int_{\Omega(0,b)} \mathcal{P} \mathrm{d}z \mathrm{d}X$ represents the change in kinetic wave energy due to the interaction with the background current $\bar{\bm{U}}$ in the bulk. The strain tensor $S(\bar{\bm{U}})$ is symmetric and may be decomposed into principal directions and eigenvalues $\{D_i,\lambda_i\}_{i = 1}^3$, where $\lambda_i > 0$ or $\lambda_i < 0$ represents, respectively, the volumetric expansion or contraction of $\bar{\bm{U}}$ in direction $D_i$.
    Hence, for $\bm{u}$ parallel with $D_i$, we have 
    \[ \mathcal{P}= -\lambda_i |\bm{u}|^2,\]
    and we therefore arrive at a similar mechanism as for the wave amplitude: an expanding flow in the direction of the motion causes a decrease of wave kinetic energy, while a contracting flow in the direction of the wave motion causes an increase 
    kinetic wave energy. Following \cite{Pope2000TurbulentFlows}, we may interpret this as the background flow expansion resulting from the wave doing work on the background flow and losing energy, or the background flow doing work on the wave and the wave gaining energy. 
    
\end{itemize}
\vspace{2mm}
To the best of our knowledge, this specific form of the energy evolution in wave-current interaction has not been presented before, and we therefore summarize it in a proposition. 
\begin{proposition}
    Under the assumptions of \Cref{sect: mathematical model}, the total energy of the wave system \eqref{eq: linerized2d} satisfies the evolution equation 
    \begin{equation}
    \frac{\mathrm{d}}{\mathrm{d}t}\int_{\R^2} \mathcal{E} \mathrm{d}X = -\int_{\R^2}\nabla_X \cdot \bar{U} \frac{1}{2}g|\eta|^2 \mathrm{d}X + \int_{\Omega(0,b)}  \mathcal{P} \mathrm{d}X 
\end{equation}
where $\mathcal{P} = -\bm{u} \cdot S(\bar{\bm{U}})\bm{u}$ is the production term related to the bulk interaction of the background current $\bar{\bm{U}}$ and the wave velocity $\bm{u}$. 
\label{prop: total energy evo}
\end{proposition}
\begin{center}
\vspace{2mm}
\underline{\textbf{Energy dynamics}}
\vspace{2mm}
\end{center}
We now change perspective and study the evolution of the energy density in the asymptotic regime. We first define the WKB amplitude and surface potential \[\eta = A_\eta e^{iS/\mu}\quad \text{ and }  \varphi = A_\varphi e^{iS/mu}.\]
Replacing now $\mathcal{G}(b)$ by the asymptotic representation $\Gw^\mu(b)$ and using  \Cref{prop: w-expansion}, we get $\mu^0$ expression for kinetic and potential energy densities are \[e_K = \varphi \Gw^\mu(b) \overline{\varphi} = \frac{1}{2g}\sigma^2 A_\varphi^2 + \mathcal{O}(\mu) \quad \text{and} \quad e_P = \frac{g}{2} |\eta|^2 = \frac{g}{2}A_\eta^2.\]
Note that these expressions are usually obtained by considering the real part of $\eta$ and $\varphi$ for $e_K$ and $e_P$ and then phase averaging the result, and are therefore referred to as phase averaged energy densities. Recalling that $\omega = -\partial_t S$ and $\bm{k} = \nabla_X S$, we get from \eqref{eq: linerized2d} that \begin{equation}
    (\partial_t^\mu + \bar{U}\cdot \nabla_X^\mu)\varphi = -g \eta \implies -i(\omega - \bar{U}\cdot \bm{k})A_\varphi e^{iS/\mu}  = -g A_\eta e^{iS/\mu} + \mathcal{O}(\mu).
    \label{eq: A rel}
\end{equation}
Using the dispersion relation $\omega = \sigma + \bar{U}\cdot \bm{k}$ we therefore have that $\sigma^2 A_\varphi^2 = g^2 A_\eta^2 + \mathcal{O}(\mu) $, and consequently we have the classical equipartition of phase averaged energy density, i.e., 
\[e_K = \frac{1}{2g}\sigma^2 A_\varphi^2 = \frac{g}{2}A_\eta = e_P.\]
In the following, we denote the total phase averaged energy density by $E = e_K + e_P$. 

We now consider the asymptotic version of equation \eqref{eq: energy density}. We have immediately that $\mathcal{E}^\mu = \frac{1}{2}(g|\eta|^2 + \varphi \Gw^\mu \overline{\varphi}) = E + \mathcal{O}(\mu)$, and after a careful expansion of the different terms, (which we shall soon consider), we find that $E$ satisfies the evolution equation
    \begin{equation}
    \partial_t E + \nabla_X \cdot (\bar{U} + C_g)\cdot \nabla_X E =  \frac{E}{\sigma} \left(  \bar{U}\cdot\nabla_X \sigma  -  C_g\cdot\nabla_X (\bar{U}\cdot \bm{k}) \right) + \mathcal{O}(\mu). 
    \label{eq: alt energy}
\end{equation}
We compute the terms on the right hand side with $\eta = A_\eta e^{iS/\mu}$ and $\varphi = -i\frac{g}{\sigma}A_\eta e^{iS/\mu}$. Using the expansion in \Cref{prop: w-expansion},  have 
{\small
\begin{align*}
    \frac{1}{2}g\left( \overline{\eta}\mathcal{G}\varphi - \varphi \mathcal{G}\overline{\eta}\right) &= \frac{1}{2}\left( A_\eta \left(\sigma^2  - i\mu \left(\nabla_X \cdot \sigma C_g\right)  -i\mu  2\sigma C_g \cdot \nabla_X  \right) \left(\frac{-ig}{\sigma}A_\eta\right) \right) \\
    &- \frac{1}{2}\left(\left(\frac{-ig}{\sigma}A_\eta\right)\left(\sigma^2  + i\mu \left(\nabla_X \cdot \sigma C_g\right)  +i\mu  2\sigma C_g \cdot \nabla_X  \right)A_\eta  \right) \\
    &= \frac{1}{2}\left(  -2\mu \left(\nabla_X \cdot \sigma C_g\right)\frac{g}{\sigma}A_\eta^2 -  2\mu gA_\eta C_g \cdot \nabla_X A_\eta - 2\mu gA_\eta \sigma C_g \cdot \nabla_X \left(\frac{A_\eta}{\sigma}\right) \right)
\end{align*}
}
Using that \[\sigma C_g \cdot \nabla_X \left(\frac{A_\eta}{\sigma}\right) = C_g \cdot \nabla_X A_\eta - \frac{A_\eta}{\sigma} C_g \cdot \nabla_X \sigma,\]
and $E = gA_\eta^2$, we therefore have that 
\begin{equation*}
    \frac{1}{2}g\left( \overline{\eta}\mathcal{G}\varphi - \varphi \mathcal{G}\overline{\eta}\right) = - \nabla_X^\mu\cdot (C_gE) + \mathcal{O}(\mu^2). 
\end{equation*}
For the commutator term, we have 
{\small
\begin{align*}
    \bar{U}\cdot \nabla_X^\mu ( \Gw^\mu \overline{\varphi} )&= \bar{U}\cdot \nabla_X^\mu\left( \frac{e^{-iS/\mu}}{g}(\sigma^2 + i2\mu \sigma C_g \cdot \nabla_X  + i \mu \nabla_X \cdot (C_g \sigma) )A_\varphi \right),   \\
    & =  -i\bar{U}\cdot \bm{k} \left( \frac{e^{-iS/\mu}}{g}(\sigma^2 + i2\mu \sigma C_g \cdot \nabla_X + i \mu \nabla_X \cdot (C_g \sigma))A_\varphi \right) \\
    &+ \frac{e^{-iS/\mu}}{g} \bar{U}\cdot \left( (2\mu \sigma \nabla_X \sigma)A_\varphi + \mu\sigma^2 \nabla_X A_\varphi \right) + \mathcal{O}(\mu^2),  
\end{align*}
}
and 
{\small
\begin{align*}
    \Gw^\mu (\bar{U}\cdot \nabla_X^\mu \overline{\varphi}) &= \Gw^\mu( e^{-iS/\mu}(\bar{U}\cdot (-i\bm{k}A_\varphi + \mu\nabla_X A_\varphi)), \\
    & =  \frac{e^{-iS/\mu}}{g}(\sigma^2 + i2\mu \sigma C_g \cdot \nabla_X  + i \mu \nabla_X \cdot (C_g \sigma) )(\bar{U}\cdot (-i\bm{k}A_\varphi + \mu\nabla_X A_\varphi)), \\
    & = \frac{e^{-iS/\mu}}{g}\left( \sigma^2 (\bar{U}\cdot (-i\bm{k}A_\varphi + \mu\nabla_X A_\varphi)) + (2\mu \sigma C_g \cdot \nabla_X  +  \mu \nabla_X \cdot (C_g \sigma))(\bar{U}\cdot \bm{k} A_\varphi  \right) + \mathcal{O}(\mu^2).
\end{align*}
}
Summing up, we find that 
\begin{align*}
    \varphi[\bar{U}\cdot\nabla_X^\mu,\Gw^\mu(b) ] \overline{\varphi} &= \frac{\sigma A_\varphi^2}{g}\left(2 \bar{U}\cdot\nabla_X^\mu \sigma - 2 C_g \cdot \nabla_X^\mu (\bar{U}\cdot \bm{k}) \right) + \mathcal{O}(\mu^2) \\
       & = \frac{2E}{\sigma}\left( \bar{U}\cdot\nabla_X^\mu \sigma -  C_g \cdot \nabla_X^\mu (\bar{U}\cdot \bm{k}) \right) +  \mathcal{O}(\mu^2).
\end{align*}
Last, the term  $\nabla_X^\mu \cdot \bar{U}\left(\frac{1}{2}\varphi \Gw^\mu \overline{\varphi} -\frac{1}{2}g  |\eta|^2\right) = \mathcal{O}(\mu^2)$, and we get \eqref{eq: alt energy} after dividing by $\mu$.
\begin{center}
\vspace{2mm}
\underline{\textbf{Wave Action}}
\vspace{2mm}
\end{center}
If one consults the literature to learn about the interaction of water waves with currents and bathymetry, one soon finds that the go-to equation is the conservation law for wave action (cf. \cite{mei1989applied,buhler2014waves,phillips1977upperocean}). Wave action is defined as the phase averaged energy density divided by the intrinsic wave frequency, i.e.,  
\[\mathcal{A} = \frac{E}{\sigma},\] 
and the canonical origin for the conservation of wave action is the paper \cite{bretherton1968wavetrains}, showing that $\mathcal{A}$ satisfies the continuity equation
\begin{equation}
    \partial_t A + \nabla_X \cdot \left((\bar{U} + C_g)\mathcal{A}\right) = 0.
    \label{eq: wave action 1}
\end{equation}
This result is found through Whitham's averaged Lagrangian method\cite{whitham2011linear}, and although this method is elegant, the actual application to water waves (Ch. 4.2 in \cite{bretherton1968wavetrains}) is rather high-level and brief. Although the wave action equation has considerable historical clout, it seems to lack a standard, transparent WKB type derivation from a set of primitive PDEs; in the standard works \cite{phillips1977upperocean,buhler2014waves,komen1994dynamics}, equation \eqref{eq: wave action 1} is just stated, and in similar works \cite{mei1989applied,johnson1997modern}, the derivations in the setting of both varying currents and bathymetry are incomplete. For the sake of completeness, we first derive \eqref{eq: wave action 1} using Whitham's method of the averaged Lagrangian, and then show that it agrees with asymptotic energy equation \eqref{eq: alt energy} coming from \eqref{eq: linerized2d}. \\

After integrating by parts the volume term in the linearized Lagrangian from \Cref{sect: mathematical model}, the action of \eqref{eq: linerized2d} becomes 
\begin{equation}
    \mathcal{S} = \int_{t_1}^{t_2} \int_{\R^2} \mathcal{L} \mathrm{d}X \mathrm{d}t, \quad  \mathcal{L} = \frac{1}{2}\varphi \mathcal{G}\overline{\varphi} + \frac{1}{2}g|\eta|^2 + \overline{\eta}(\partial_t \varphi + \bar{U}\cdot \nabla_X \varphi). 
\end{equation}
Inserting WKB solutions for $\eta$ and $\varphi$ and the using asymptotic scaling, we apply  \Cref{prop: w-expansion} and equation \eqref{eq: A rel} to obtain an $\mathcal{O}(1)$, phase-averaged\footnote{Again, due to over use of complex quantities, there is no need to actually do the phase averaging.} approximation  $\overline{\mathcal{L}}$:
\[\overline{\mathcal{L}} = \frac{\sigma^2}{g}A_\varphi^2 - \frac{1}{g}A_\varphi^2(\partial_t S + \bar{U}\cdot \nabla_X S)^2 + \mathcal{O}(\mu).  \]
Crucially, the phase-averaged Lagrangian depends only on derivatives of the phase $S$. A stationary point of the averaged action functional is now found by computing the Euler-Lagrange equations (cf. Ch. 8, \cite{evans2022partial}): 
\begin{equation}
    \partial_{A_\varphi} \overline{\mathcal{L}} = 0 \quad  \text{and} \quad  \partial_t \left(\frac{\partial \overline{\mathcal{L}} }{\partial (\partial_t S)}\right) + \partial_{x_1}\left(  \frac{\partial \overline{\mathcal{L}} }{\partial(\partial_{x_1}S) }\right) + \partial_{x_2}\left(  \frac{\overline{\mathcal{L}} }{\partial (\partial_{x_1}S) }\right)= 0.  
\end{equation}
The first equation gives the Doppler shifted dispersion relation: 
\begin{equation}
    \partial_{A} \overline{\mathcal{L}} = 0 \implies \sigma^2(\bm{k}) - (\bar{U}\cdot \bm{k} - \omega )^2 = 0 \implies \omega = \bar{U}\cdot \bm{k} \pm \sigma(\bm{k}).
    \label{eq: doppler}
\end{equation} 
The second equation takes the form 
\begin{equation*}
    \partial_t \left( (-\omega + \bar{U}\cdot \bm{k})A^2_\varphi \right) + \nabla_X \cdot ( \sigma\nabla_{\bm{k}} \sigma A^2_\varphi + \bar{U} (\bar{U}\cdot \bm{k} -  \omega) A^2_\varphi) = 0. 
\end{equation*}
Using that $ \nabla_{\bm{k}} \sigma = C_g$ and $\mathcal{A} = E/\sigma = \sigma A^2_\varphi$  we therefore get the conservation law for wave action, 
\begin{equation*}
      \partial_t \mathcal{A} + \nabla_X ((\bar{U} + C_g) \mathcal{A}) = 0.
\end{equation*}

To compare the wave action equation with our asymptotic energy equation \eqref{eq: alt energy}, we re-write it as an energy equation:
\begin{equation}
    \partial_t E + \nabla_X \cdot \left((\bar{U} + C_g)  E\right) = \frac{E}{\sigma}\left(\partial_t \sigma  + (\bar{U} + C_g)\cdot \nabla_X \sigma \right).  
    \label{eq: wave action energy form}
\end{equation}
Using now the Hamilton-Jacobi equation $\partial_t \bm{k} + \nabla_X \omega = 0$ (which is exact) and the dispersion relation \eqref{eq: doppler}, we get  \[\partial_t \sigma = \nabla_{\bm{k}}\sigma \cdot \partial_t \bm{k} = - C_g \cdot \nabla_X \omega = - C_g \cdot \nabla_X (\sigma + \bar{U}\cdot \bm{k}).\]    
Inserting this into the right hand side of \eqref{eq: wave action energy form}, we see that it agrees with equation \eqref{eq: alt energy}. In other words, the asymptotic energy equation \eqref{eq: alt energy} is equivalent to the wave action equation, and so there is nothing new under the sun. Still, we hope that our alternative route from the Euler equations to phase averaged energy dynamics have illuminated some of the details of the mechanisms involved. 

\subsubsection{Numerical Examples}
\label{sect: num energy}
We now do a computational verification of the total energy evolution in  \Cref{prop: total energy evo} and the above asymptotic energy theory. We use the computational method described in the Appendix to compute the exact energy density $\mathcal{E}$ from \eqref{eq: linerized2d}, and we compute the asymptotic energy density $E$ from \eqref{eq: wave action energy form}. To avoid diffractive effects/decay dominating our comparison (see \Cref{sect: diffraction}), we use 1D simulations. 

\begin{center}
\vspace{2mm}
\underline{\textbf{Total energy evolution}}
\vspace{2mm}
\end{center}
We take the computational domain to be $\Omega_c = [0,L]$ with $L = 2000 \ m$, and the constant depth $b = 9 \ m$.  For our current, we set 
\begin{align*}
\bar{U}_1(x_1,z) &= 1 + \frac{1}{2}\tanh((x_1 - 2L/3)/300)\cos(\pi z/b), \\
\bar{W}(x_1,z) &= -\frac{b}{2\pi 300}\frac{1}{\cosh^2((x_1 - 2L/3)/300)}\sin(\pi z/b), 
\end{align*}
with units $m/s$. We define the gaussian $B(x_1; x_c,\sigma) = \exp\left(-(x_1 - x_c)^2/(2\sigma^2)\right)$ and solve the initial value problem for \eqref{eq: linerized2d} with $\bar{U} = \bar{U}_1(x_1)$ and initial condition $\varphi_0 = 0$ 
\[\eta_0(x_1) = \frac{1}{2}B(x_1;L/2,0.04L)\cos(k_0(x_1 - L/2)), \quad k_0 = \frac{2\pi }{(L/50)}.\]
In each time step we compute the terms 
\[ I_s(t) = -\int_{\Omega_c} \partial_{x_1} \bar{U}_1(x_1) g\eta^2(t,x_1) \mathrm{d}x_1 \quad \text{and} \quad  I_b(t) = \int_{\Omega_c} \int_{-b}^0 \mathcal{P}(t,x_1,z) \mathrm{d}z \mathrm{d}x_1.\]
Defining the total energy $E_T(t) = \int_{\Omega_c} \mathcal{E}(t,x_1) \mathrm{d}x_1$, we also compute the approximation 
\[ \tilde{E}_T(t) = E_T(0) + \int_0^\top I_s(\tau) + I_b(\tau) \mathrm{d}\tau,\]
and compare it with $E_T(t)$. 
\Cref{fig: total energy evo} illustrates the setup and shows the results. The numerical results shows excellent agreement between $E_T$ and $\tilde{E}_T$. Moreover, it shows that the main driver of total energy decrease is the surface divergence term $I_s(t)$ in our specific example.   

\begin{figure}[htbp]
    \centering
    \includegraphics[width=1\linewidth]{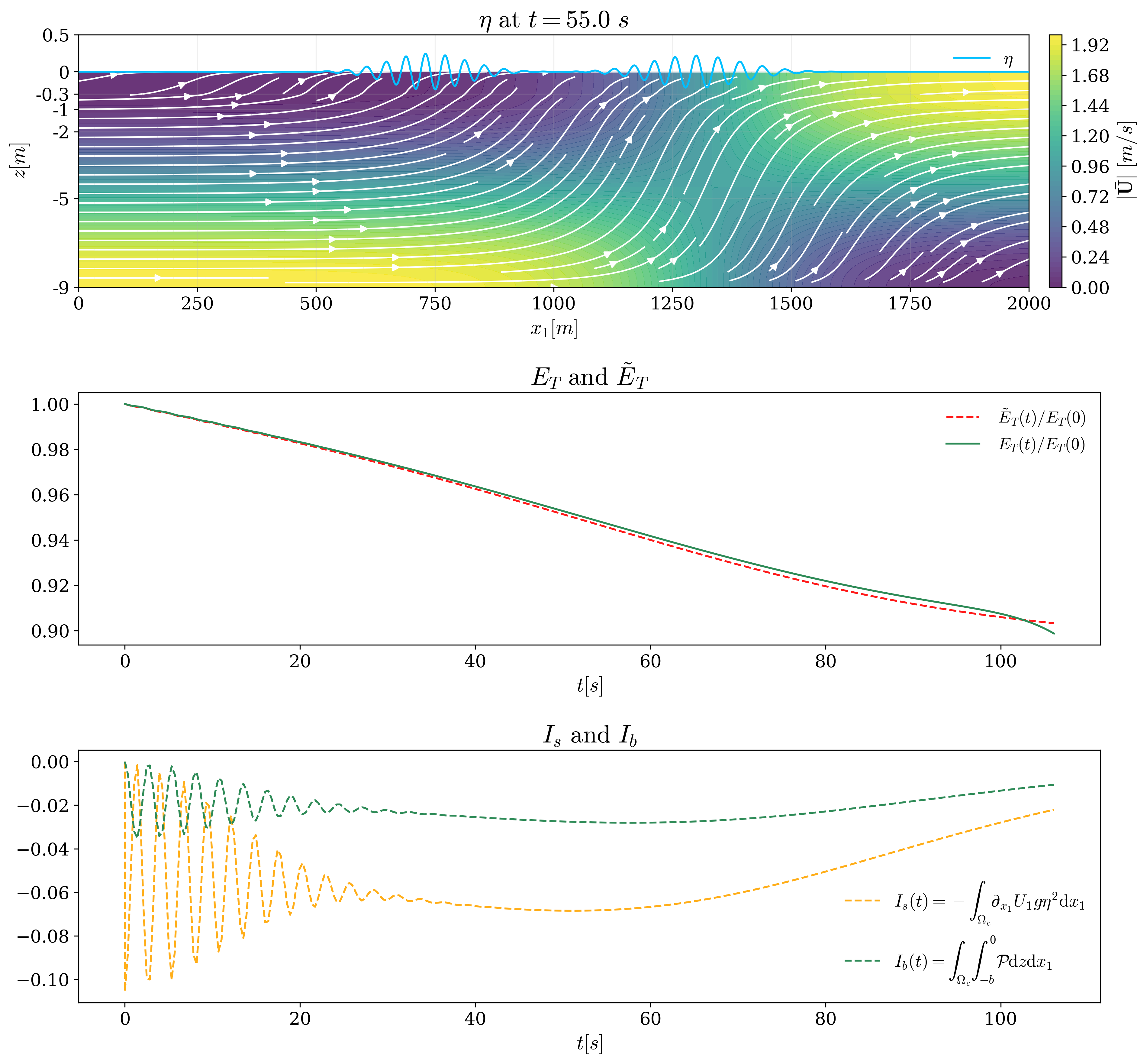}
    \caption{The first panel shows a snapshot of the propagating wave $\eta$ on top of the bulk current. In the bulk the $z$-direction is compressed for economic reasons, and the vertical velocity is slightly amplified for visibility. The second panel shows evolution of $E_T(t)$ and $\tilde{E}_T(t)$ throughout the simulation, while the third panel shows the computed values of $I_s(t)$ and $I_b(t)$. }
    \label{fig: total energy evo}
\end{figure}
    
\begin{center}
\vspace{2mm}
\underline{\textbf{Energy density dynamics over a bumpy bathymetry and current}}
\vspace{2mm}
\end{center}
We now consider the case of a moderately sloping, bumpy bathymetry and variable current. The depth profile is \[b(x_1) = 24 - 18B(x_1;2L/3,0.08L) -  14.5B(x_1;L/2,0.02L) -  3.6B(x_1;2.2L/3,0.01L).\]
We set the current to be {\small \[\bar{U}_1(x_1) = 1.2 + 0.5\left(B(x_1;2L/3,0.12L) +  0.8B(x_1;L/2,0.02L)+ 0.02B(x_1;2.2L/3,0.01L)\right), \]}
and $b$ and $\bar{U}_1$ has units $m$ and $m/s$, respectively. Such a current could, approximately, result from mass conservation over the bump for current with a suitable depth dependence. The maximal gradient of $b$ is now $\mathrm{max} |\partial_{x_1} b| \approx 0.18$ and for $\bar{U}$,  $\mathrm{max} |\partial_{x_1} \bar{U}_1| \approx 0.005$. Hence the asymptotic parameter $\mu$ is $\mu \approx 0.18$, i.e., not very small. For the initial condition we take $\varphi_0 = 0$ and \[\eta_0(x_1) = \frac{1}{2}B(x_1; x_0,0.04L)\cos(k_0(x_1-x_0))  \text{ with }  k_0 = \frac{2\pi}{(L/60)}, \quad x_0 = 3L/10. \]
In the full simulation, the initial wave splits into two components traveling in opposite direction. We track only the energy of the right-traveling component, and compare it to the wave action energy $E$. As initial condition $E_0$ for $E$, we note that for each of the right/left components, $E_0 \approx \frac{1}{2}\mathcal{E}_0 $, and so we set $ E_0 = \frac{1}{2}g(\frac{1}{2}B(x_1; 3L/10,0.04L))^2$. Moreover, some time is needed before $\mathcal{E}$ settles into an envelope-like function comparable to $E$. We therefore let the waves propagate and stabilize, and compare $\mathcal{E}$ and $E$ the subdomain $\mathcal{D}_m = [0.4 L,0.85L]$,

In this scenario the numerical simulations presented in \Cref{fig: current and bathy} shows that $E$ is an excellent approximation to $\mathcal{E}$, both in terms of pointwise accuracy and total energy.  We note that the energy density amplitudes decays as the wave propagates over the bump. This is perhaps counterintuitive, but can be explained by the fact that the group velocity is not a strictly decreasing function of depth. For stationary waves in the absence of a current, we have $ \partial_{x_1}(C_g E) =0$ and so $E_b/E_a = C_{g,a}/C_{g,b}$, where $E_a,E_b$ and $C_{g,a},C_{g,b}$ are the energy densities and group velocities at points $x_a,x_b$, respectively. As $C_g(b,k)$ takes on a maximum when $kb \approx 1.2$, which is (accidentally) approximately the value of $kb$ at the peak of the bump in our simulated example, this explains the decay in amplitude. \\
\begin{figure}[htbp]
    \centering
    \includegraphics[width=1\linewidth]{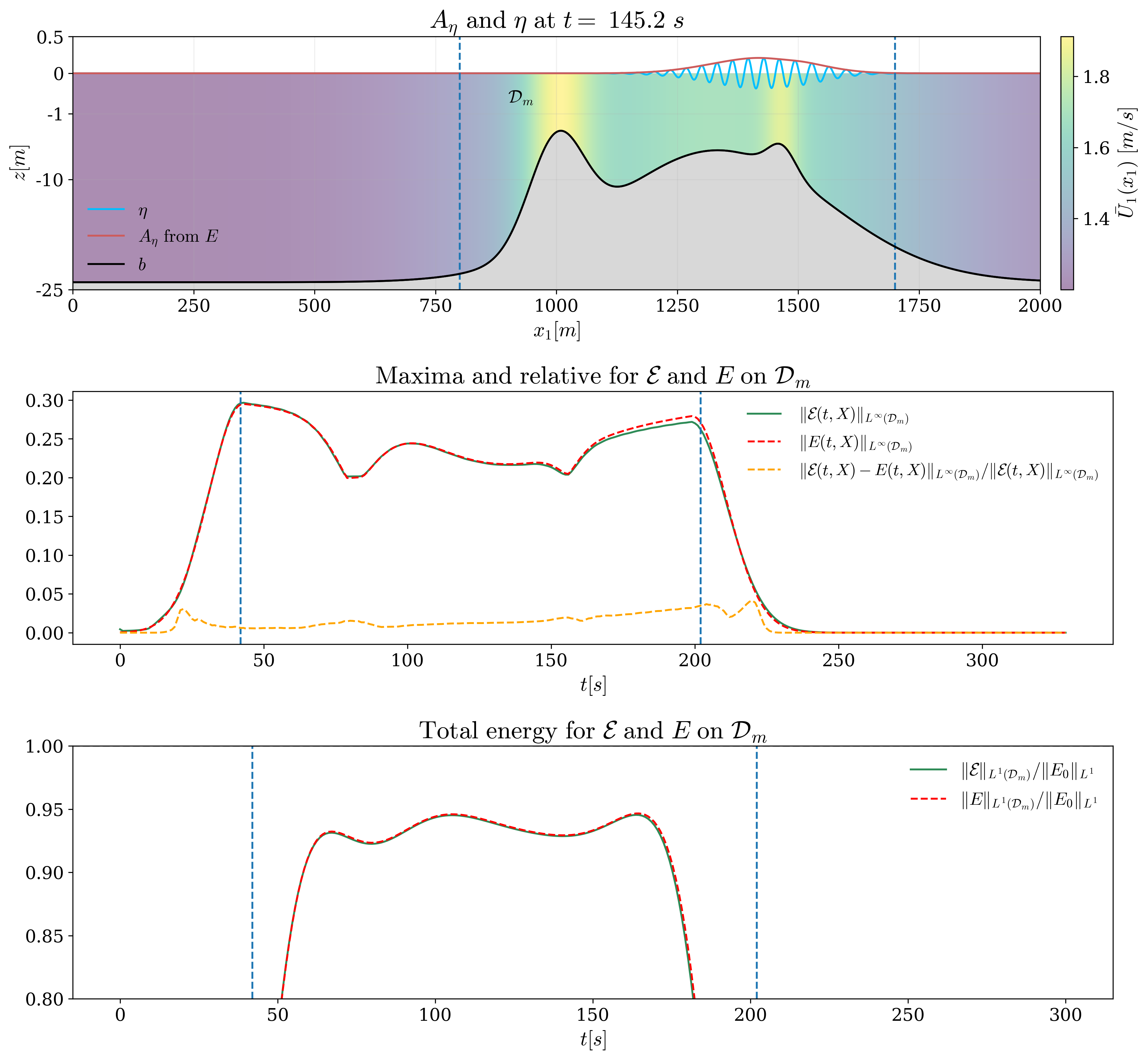}
    \caption{In the first panel we see the amplitude $A_\eta(t,X) = \sqrt{E/g}$ and the simulated wave packet $\eta(t,X)$ together with the bathymetry and 1D current. The vertical dashed lines indicate the measurement domain $\mathcal{D}_m$. In the second panel we plot the maxima and relative difference of $E$ and $\mathcal{E}$ in $\mathcal{D}_m$ as a function of time. The vertical dashed lines indicates approximately when the peak of the wave packet enters and leaves $\mathcal{D}_m$. The third panel shows the time evolution of the total energy of both $E$ and $\mathcal{E}$, both normalized by $\|E_0\|_{L^1}$}
    \label{fig: current and bathy}
\end{figure}

\subsection{Diffractive models}
\label{sect: diffraction}
The wave action equation derived in the previous section is fundamental in the understanding of large-scale wave evolution. An important reason for its applicability is that it allows using geometric wave theory. Geometric wave theory, which is often referred to as ray tracing when applied to water waves, yields an elegant and simple method for wave prediction \cite{buhler2014waves,halsne2023ocean}. 
To arrive at the ray tracing equations, we note that by the definition given in \eqref{eq: freq wavenumber}, the local wavenumber $\bm{k}$ and frequency $\omega$ of a wave packet $\eta = Ae^{iS}$ must satisfy $\partial_t \bm{k} + \nabla_X \omega = 0$. Coupling this equation with the dispersion relation $\omega(\bm{k},X)$, we get an evolution equation for $\bm{k}$, the Hamilton-Jacobi equation:
\begin{equation}
    \partial_t \bm{k} + \nabla_X \omega(\bm{k},X) = 0 \quad  \implies \quad \partial_t \bm{k} + \nabla_{\bm{k}} \omega \cdot \nabla_X \bm{k} + (\nabla_X \omega )\bm{k} = 0.
\end{equation}
Next, the characteristics of the wave action equation is given by $\dot{X}(t) = \bar{U} + C_g = \nabla_{\bm{k}} \omega$, and therefore get the system 
\begin{align*}
     \dot{X}(t) &=  \nabla_{\bm{k}}\omega(\bm{k},X),\\
      \dot{\bm{k}}(t,X(t)) &= -\nabla_X \omega(\bm{k},X) \bm{k}(t,X(t)), \\
      \dot{\mathcal{A}}(t,X(t)) &= - \nabla_X\cdot(\nabla_{\bm{k}}\omega(\bm{k},X))\mathcal{A}(t,X(t)).
\end{align*}
These equations constitute the ray tracing equations. They reduce the complex wave behavior to coupled system of ordinary differential equations for $\bm{k}$ and $\mathcal{A}$. Remarkably, this description of waves is fairly accurate, and, as we have seen, works in variable media. 
\newline 
While ray tracing is able to model effects of refraction, i.e., the effects of slowly varying currents and bathymetry, it cannot model the effect of diffraction \cite{buhler2014waves}. Diffraction is, loosely speaking, the spread of waves due to geometry; it can either be due to the geometry of the domain in which the wave propagates, i.e, due to the interaction with edges or corners, or due to the geometry of the wave itself\footnote{Or, in the words of the inventor of diffractive optics J.B. Keller, in the context of light, ``Diffraction is the process whereby light propagation differs from the predictions of geometrical optics.'' \cite{mcnamara1990introduction}}.
We now show how two standard models that that account for both refraction and diffraction can be derived from the system \eqref{eq: linerized1}. We start by deriving the so-called mild-slope equation (cf. \cite{dingemans1997water}) in the absence of currents, and continue by deriving a diffractive extension of the wave action equation in the form of a linear Schrödinger-type equation. 
\begin{center}
\vspace{2mm}
\underline{\textbf{The mild-slope equation}}
\vspace{2mm}
\end{center}
To keep the derivation simple and close to the standard framework (cf. \cite{mei1989applied}), we assume $\bar{U} = 0$ and set the asymptotic parameter $\mu = 1$. Moreover, we assume the mild-slope type condition is imposed according to  \Cref{prop: G est}, i.e., that $M(b) = \frac{|\nabla_X b|(1+|\nabla_X b|)}{b_{\text{\tiny{min}}}} \ll 1$. 
In the absence of current, \eqref{eq: linerized2d} reduces to the second order equation $\partial_t^2 \varphi + g\Gw(b)\varphi = 0$, and we now solve for a time harmonic solution $\varphi = e^{i\omega t}\psi(X)$ to the second order to this equation.  It follows that $\psi$ has to satisfy the eigenvalue problem
\begin{equation}
    g\Gw(b) \psi = \omega^2\psi.
    \label{eq: mild-slop orig}
\end{equation}  
The following proposition shows that for mild-sloping bathymetries, the solution to the mild-slope equation is an approximate solution to the above equation. 
\begin{proposition}
    For $\delta \ll 1 $, we assume the bathymetry $b(X)$ satisfies \[|\partial_{x_i}^m b| \leq \delta^m, m = 0,1,2,..., \quad \text{and} \quad  M(b) = \frac{\delta(1 +\delta)}{b_{\tiny{\text{min}}}} \ll 1. \]
    Let $\kappa_0(X) = |\bm{k}_0(X)|$ and set $\Omega(X,\kappa) = \sqrt{g\kappa\tanh(b(X)\kappa)}$. For a fixed temporal wave frequency $\omega$, let $\kappa_0(X)$ satisfy $\omega^2 = \Omega(X,\kappa_0(X))$. Moreover, set $c = \frac{\Omega(X,\kappa_0)}{\kappa_0}\partial_\kappa \Omega(X,\kappa_0)$.   If $\psi$ satisfies the mild-slope equation 
    \begin{equation}
        \nabla_X \cdot c \nabla_X \psi + \kappa_0^2c \psi = 0
        \label{eq: mild-slope}
    \end{equation}
    then 
    \[g\Gw(b) \psi = \omega^2\psi + \mathcal{O}(\delta^2).\]
\end{proposition}
\begin{proof}
   The eigenvalue problem \eqref{eq: mild-slop orig} imposes the dispersion relation on $\psi$; for $u \in L^2$ satisfying $ {g\Opw(g_b -\omega/g)u = 0}$, $u$ is localized in phase-space in the characteristic set \[\{(X,\xi) : g|\xi|\tanh(b(X)|\xi|) - \omega^2 = 0\}, \]  (cf. Ch. 2.9 in \cite{martinez2002introduction}). 
Therefore, and in line with the wavenumber shift property for wave packets in  \Cref{prop: w-expansion}, we want to Taylor expand the symbol $g_b$ around a spatially varying wavenumber $\kappa_0(X) = |\xi_0(X)|$ satisfying the dispersion relation ${\omega^2 = \Omega^2(X,\kappa_0(X))}$. We write $\kappa^2 = \xi^2$ and $g_b(X,\kappa) = g\sqrt{\kappa^2}\tanh(b(X)\sqrt{\kappa^2})$, and compute \[ g_b(X,\kappa) = g_b(X,\kappa_0) + (\kappa^2 - \kappa_0^2)\left(\partial_\tau g\sqrt{\tau}\tanh(b(X)\sqrt{\tau})\right)\big|_{\tau = \kappa_0^2} + \mathcal{O}((\kappa^2 - \kappa_0^2)^2).\]
Next, we have that \[\left(\partial_\tau g\sqrt{\tau}\tanh(b(X)\sqrt{\tau})\right)\big|_{\tau = \kappa_0^2} = \partial_\tau \Omega^2(X,\sqrt{\tau})\big|_{\tau = k_0^2} = \frac{\Omega(X,k_0)}{k_0}\partial_k \Omega(X,k_0) = c,\]
Inserting the truncated symbol in the Weyl operator, we find by the same method as that in the proof of  \Cref{prop: w-expansion} that
\begin{align*}
     \Opw(g_b(X,\kappa_0(X)) + (\kappa^2 - \kappa_0^2)c)\psi = (\Omega^2(X,\kappa_0) - \kappa_0^2c)\psi  - \nabla_X \cdot c \nabla_X \psi + \mathcal{O}(\delta^2),
\end{align*}
where the $\mathcal{O}(\delta^2)$ term comes from the term $\frac{1}{4}\Delta_X c $ and the assumption $|\Delta_X b| \sim |\nabla_X b|^2 \sim \delta^2$. 
Inserting this into \eqref{eq: mild-slop orig}, we get
\[ g\Gw^\mu(b) \psi - \mu^2\omega^2 \psi = 0 \implies    \nabla_X \cdot c \nabla_X \psi + k_0^2c\psi  = 0. \]
This is the standard form of the mild-slope equation (cf. \cite{mei1989applied}). A complete error analysis is more involved, and we only outline the approach. Assuming $\psi$ solves the mild-slope equation we have $L_1 \psi = \Opw(c(\kappa^2 - \kappa_0^2)) \psi = \mathcal{O}(\delta^2)$. Now write $c_1 = \frac{1}{2}\partial^2_\tau \Omega^2(X,\sqrt{\tau})\big|_{\tau = k_0^2}$, $K = \kappa^2 - \kappa_0^2$ and $c_2 = c_1/c$. By the composition formula, the next term in the expansion of $g_b$ is
\begin{align*}
R_1 &= \Opw(c_1 K^2) = \Opw(c_2K cK) = \Opw(c_2K) L_1 + \frac{1}{2i}\Opw( \{c_2K,cK\} ) + \Opw(r_1),   
\end{align*}
where $\{c_2K,cK\} = 2\xi \cdot(c_2\nabla_X c - c\nabla_X c_2)K $. We therefore get $\Opw(\{c_3K,c_1K\} ) = \Opw(c_3\xi)L_1 + \Opw(r_2)$. The reminder terms consist of spatial derivatives of order $n\geq 2$ of functions composed with $b(X)$, and so $\Opw(r_1) \sim \Opw(r_2) = \mathcal{O}(\delta^2)$. Therefore we have  $R_1\psi =\mathcal{O}(\delta^2)$. This argument may now be recursively expanded to $R_2 \approx \Opw(K^3)$ in terms of $R_1$, so that $R_2 \psi = \mathcal{O}(\delta^3)$ and so forth. 
 \end{proof}

For a detailed analysis of the bathymetry scattering problem for the mild-slope equation, we refer to \cite{kirkeby2025imaging}.
The mild-slope equation is an elliptic wave equation, and may be considered as a time-averaged wave equation, and it retains its spatial oscillatory wave nature. This, together with the ellipticity, makes the equation challenging to solve numerically \cite{radder1979parabolic}. To overcome this, a further approximation is often imposed; by assuming a constant wavenumber $\bm{k}_0 = (k_0,0)$, i.e., unidirectional propagation, the simplest parabolic approximation of the mild-slope equation is 
\[ \partial_{x_1} A_\varphi = \frac{i}{2k_0}\left((k^2(X) - k_0^2) + \partial_{x_2}^2\right) A_\varphi.\]
We refer to (Ch. 5, \cite{Hunt1997}) or the papers \cite{radder1979parabolic,kirby1986higher} for details of the derivation. The parabolic approximation is the simplest phase-averaged model that includes the effects of diffraction, and it allows for effective numerical approximations \cite{kirby1986higher}. Moreover, the equation can be expanded to yield the non-linear Schrödinger equation, a fundamental equation for the amplitude dynamics of weakly non-linear, dispersive waves \cite{ablowitz2011nonlinear,kirby1983parabolic}. 
\begin{center}
\vspace{2mm}
    \underline{\textbf{The linear Schrödinger equation}}
    \vspace{2mm}
\end{center}
Using our formalism, we now derive a parabolic approximation for \eqref{eq: linerized2d}. After rearrangement, $\varphi$ satisfies the second order equation  
\[(\partial_t^\mu + \bar{U}\cdot \nabla_X^\mu)^2 \varphi + g\Gw^\mu \varphi - (\nabla_X^\mu \cdot \bar{U})(\partial_t^\mu + \bar{U}\cdot \nabla_X^\mu) \varphi = 0. \]
Obtaining a Schrödinger equation for the amplitude $A_\varphi$ in $\varphi = A_\varphi e^{iS/\mu}$ is now straight forward. Writing $  \mu D_{\bar{U}} = \partial_t^\mu + \bar{U}\cdot \nabla_X^\mu$, we 
have \[\mu^2 D_{\bar{U}} \varphi = e^{iS/\mu}\left( -(D_{\bar{U}} S)^2 A_\varphi - i \mu( 2(D_{\bar{U}} S) D_{\bar{U}} A_\varphi + AD_{\bar{U}}^2 S ) +\mu^2 D_{\bar{U}}^2 A_\varphi \right)\]
Defining $ D(X,\bm{k}) = \left[ \frac{1}{2}\partial^2_{k_i,k_j} \sigma(X,\bm{k}) \right]_{i,j = 1}^2$, we follow  \Cref{prop: w-expansion}, and include $\mu^2$ differential terms in the WKB expansion of $g\Gw^{\mu}$. We get 
\begin{align*}
g\Gw^\mu(b) \eta &= e^{iS/\mu}\left(\sigma^2 A - i\mu  2\sigma C_g \cdot \nabla_X A - i\mu \left(\nabla_X \cdot \sigma C_g \right) A  \right) \\
&- \mu^2 e^{iS/\mu} \sigma \left( \sum_{i,j=1}^2 D_{i,j}\partial^2_{x_i,x_j} A + \sum_{i,j=1}^2 (\partial_{x_i} D_{i,j}) \partial_{x_j} A \right) + \mathcal{O}(\mu^2),  
\end{align*}
which  contracts into the conservative divergence form:
\begin{align*}
g\Gw^\mu(b) \eta &= e^{iS/\mu}\left(\sigma^2 A - i\mu  2\sigma C_g \cdot \nabla_X A - i\mu \left(\nabla_X \cdot \sigma C_g \right) A  \right) \\
&- \mu^2 e^{iS/\mu} \sigma \nabla_X \cdot (D \nabla_X A) + \mathcal{O}(\mu^2).
\end{align*}
Using that $D_{\bar{U}}S = -\omega + \bar{U}\cdot \bm{k}$ we get at order $\mu^0$ the dispersion relation $(-\omega + \bar{U}\cdot \bm{k})^2 = \sigma^2 $. At order $\mu^1$ (but keeping the diffraction term), we get after some rearrangement the  Schrödinger equation 
\begin{equation}
\partial_t A_\varphi + (\bar{U} + C_g)\cdot \nabla_X A_\varphi + \frac{1}{2}\left(\nabla_X \cdot(\bar{U} + C_g)\right) A_\varphi + \frac{D_{\bar{U}+ C_g}\sigma }{2\sigma}A_\varphi  +\frac{i\mu}{2}\nabla_X \cdot (D \nabla_X A_\varphi) = 0. 
\label{eq: eta-schrödinger}
\end{equation} 
The above equation appears to be the natural extension of the wave action equation. An immediate consequence of the above formulation is that it recovers the stationary phase approximation in the setting of flat bottom and constant currents; for initial condition $A_\eta = A_0(X)$ we apply the Fourier transform and find
{\small 
\begin{align}
    A_\eta(t,X) &= \frac{1}{(2\pi)^2}\int_{\R^2}e^{i \bm{k}\cdot (X - (\bar{U} + C_g)t) + \frac{it}{2}\bm{k}^\top D\bm{k} } \hat{A}_0(\bm{k}) \mathrm{d}\bm{k}  = \int_{\R^2} K_{D,V}(X-Y,t) A_0(Y) \mathrm{d}Y
    \label{eq: schrödinger solution op}
\end{align}
}
where \[K_{D,V}(X-Y,t) = \frac{1}{4\pi i t \sqrt{\operatorname{det} D}}  e^{\frac{i}{4t}(X-Vt -Y)^\top D^{-1} (X-Vt -Y)}, \quad V = \bar{U} + C_g,\]
is the fundamental solution to the anisotropic Schrödinger equation with transport $V$ (cf. \cite{ortner2015fundamental}). This is the solution found by applying the method of stationary phase to the the linear Cauchy problem (cf. \cite{buhler2014waves,ablowitz2011nonlinear}). 
The Schrödinger propagator conserves the $L^2$-norm of the amplitude, i.e., $\|A_\eta(t,\cdot)\|_{L^2} = \|A_0\|_{L^2} $, which, in the light of  \Cref{sect: energy and action}, is the conservation of total phase-averaged energy.  
In addition, the above solution formula shows that the inclusion of the Schrödinger term gives the correct pointwise decay in time, i.e., the asymptotic equation now satisfies a so-called dispersive estimate:
\[ \|\eta\|_{L^\infty} \leq Ct^{-1}\|\eta_0\|_{L^1}, \quad t > 0. \]
This follows from Young's inequality applied to the convolution operator in \eqref{eq: schrödinger solution op}, and the pointwise temporal rate of decay rate $t^{-1}$ is true in general for linear water waves \cite{deneke2023dispersive}. It is, however, not true in general for continuity equations like the wave action equation. For constant current and depth there is pure transport and no pointwise decay. This indicates that the inclusion of diffraction via the Schrödinger term significantly changes the of the nature of the asymptotic approximation, even at the $\mu^1$ accuracy.

However, to use equation \eqref{eq: eta-schrödinger} together with the ray tracing equations still faces problems: the first is the unavoidable problem of caustics, i.e., points where rays emancipating from different locations collide (cf. \cite{buhler2014waves} for a very readable account of geometric optics, ray tracing and the problem of caustics). An accurate numerical solution to \eqref{eq: eta-schrödinger} requires knowing the wavenumber $\bm{k}$ in the domain of propagation. Moreover, diffraction is really a geometric effect, and it becomes significant whenever the propagating wave encounters edges, i.e., when the waves propagates in a domain with boundaries, or when wave fronts are curved. For water waves, this could be islands, fjords, harbors etc. Extending the ray tracing equations to domains with boundaries is non-trivial, but can be achieved by the method of diffractive geometric optics \cite{buhler2014waves,mcnamara1990introduction}, and we will not pursue this direction further in the current paper.

\subsubsection{Numerical examples}

We now illustrate the effect of diffraction due to a curved wavefront for waves propagating over a variable bathymetry and current. To do this, we compute the true energy density $\mathcal{E}$ from the full system \eqref{eq: linerized2d}, the phase averaged energy density $E$ from equation \eqref{eq: wave action 1}, and the diffractive energy density $E_S = \frac{g}{2}|A_\varphi|^2$, where the complex amplitude $A_\varphi$ is obtained by solving the Schrödinger equation \eqref{eq: eta-schrödinger}. 

Let $\Omega_c = [0,L_{x_1}] \times [0, L_{x_2}]$ with $L_{x_1} = 1500 \ m$ and  $L_{x_2} = 800 \ m$ . The bathymetry is given by the function 
\[ b(X) = 24 - 18B(x_1;L_{x_1},L_{x_1}/4)B(x_2;L_{x_2},L_{x_2}/4) - 10B(x_1;L_{x_1}/2,0.06L_{x_1}) \]
Inspired by the model in \cite{rypina2007lagrangian}, we set the surface current to be a type of meandering jet 
\[\bar{U} = (\mathrm{sech}^2\left(\frac{x_2 - 0.3\sigma_U\sin(3\pi x_1/L_{x_2})}{\sigma_U}\right), \frac{x_1}{2L_{x_2}})^\top, \quad \sigma_U = 0.7 L_{x_2}.\]
As before, $b$ and $\bar{U}$ has units $m$ and $m/s$, respectively. For initial conditions, we set $\varphi_0 = 0$ and 
\[\eta_0(X) = \frac{1}{2}B(x_1;L_{x_1}/5,0.04L_{x_1})B(x_2;L_{x_2}/2,0.04L_{x_1})\cos(k_0(x_1 -L_{x_1}/5)),  \]
with $k_0 = 2 \pi/(L_{x_1}/36)$. A snapshot of the propagating wave together with the current and the bathymetry (suitably rescaled for illustrative purposes) is shown in \Cref{fig: model}. Again, we track the right-moving wave energy density $\mathcal{E}$ and compare it with the asymptotic energy densities $E$ and $E_S$ on the subdomain $\mathcal{D}_m$. Initial condition for $E$ used in equation \eqref{eq: wave action energy form} is (as in \Cref{sect: num energy}) $E_0 = \frac{1}{2}g(\frac{1}{2}B(x_1;L_{x_1}/5,0.04L_{x_1})B(x_2;L_{x_2}/2,0.04L_{x_1}))^2 $, and the initial condition for $A_\varphi$ in equation \eqref{eq: eta-schrödinger} is $A_0 = \frac{1}{4}B(x_1;L_{x_1}/5,0.04L_{x_1})B(x_2;L_{x_2}/2,0.04L_{x_1})$.
Pointwise plots of the differences $\mathcal{E} - E$ and $\mathcal{E} - E_S$ are shown in \Cref{fig: dispersion}. The numerical results shown in the figure shows that $E_S$ is a significantly better approximation to $\mathcal{E}$ when dispersive effects are present; while the maximum difference $\|\mathcal{E} - E\|_{L^\infty(\mathcal{D}_m)}$ steadily increases, the corresponding $\|\mathcal{E} - E_S\|_{L^\infty(\mathcal{D}_m)}$ stays approximately a factor $10^{-1}$ smaller. This is in stark contrast to the example in \Cref{sect: num energy}, where diffractive effects are negligible and $E$ remains precise approximation to $\mathcal{E}$ throughout the simulation.

\begin{figure}[htbp]
    \centering
    \includegraphics[width=1\linewidth]{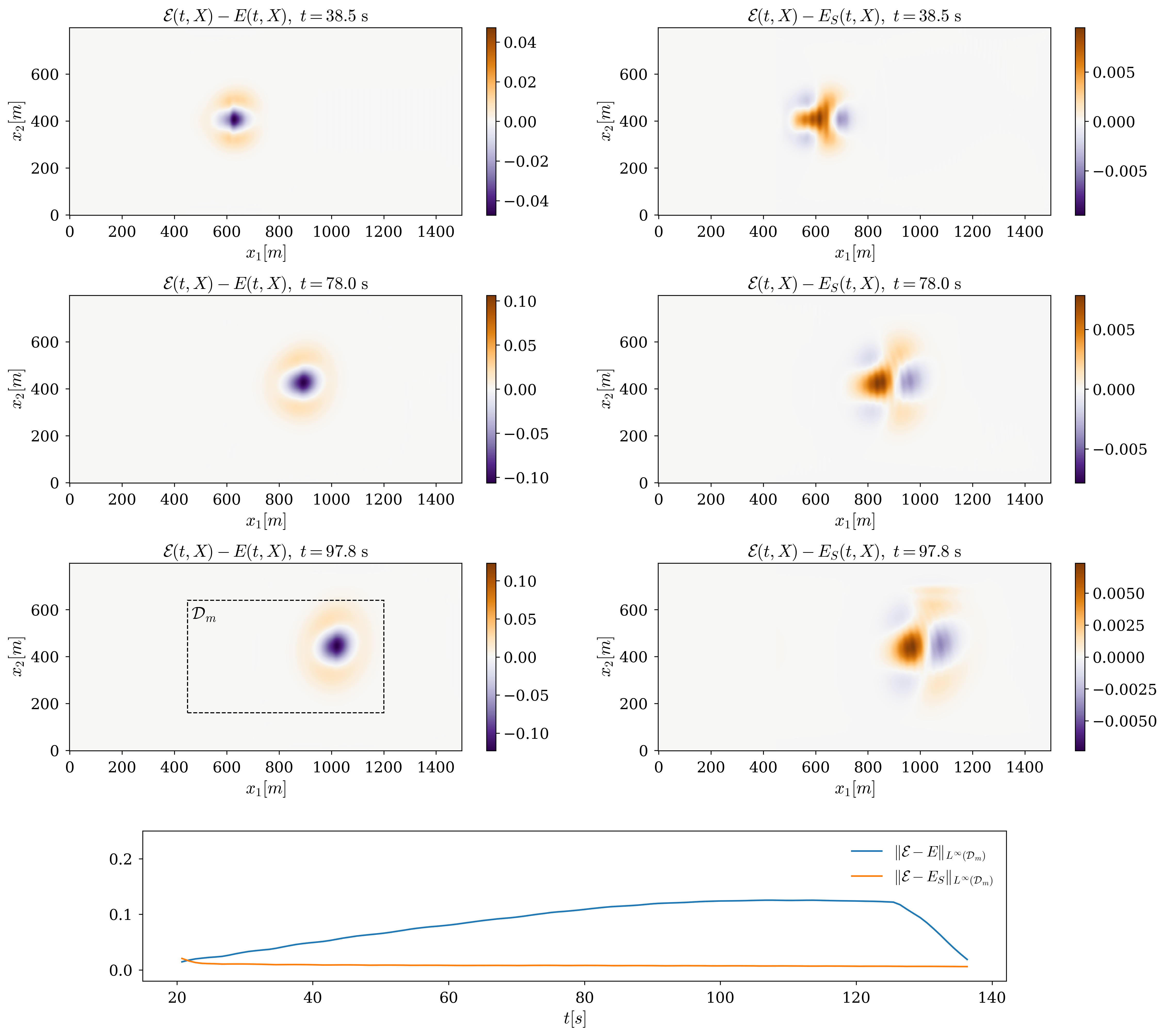}
    \caption{The left panels show the pointwise difference $\mathcal{E}(t,X) -E(t,X)$ at three different times. The lack of diffraction in $E$ manifests itself as overestimation of $\mathcal{E}$ around the peak, while underestimating it away from the peak. The pointwise difference $\mathcal{E}(t,X) -E_S(t,X)$ for the Schrödinger energy density shown in the right panels does not have a clear interpretation, but is about a factor $10^{-1}$ smaller. The lower panel shows the maximum evolution of approximation error as the wave propagates through $\mathcal{D}_m$ for the two densities. }
    \label{fig: dispersion}
\end{figure}

\subsection{Beyond WKB: phase-space dynamics}
Although WKB analysis is a powerful tool for analyzing wave propagation, it is limited by the presence of caustics in the ray tracing equations and the lack of generality in the WKB ansatz.  One approach to overcoming these limitations is to extend the object of study from the energy of a single-phase WKB solution to a more general phase-space wave energy. A way to achieve this generalization is through the introduction of the Wigner distribution, and we now apply this approach to our water waves system. For a comprehensive account of WKB analysis and Wigner distributions in the context of the Schrödinger equation we refer to \cite{jin2011mathematical}, and for an application to water waves similar in spirit to our approach, see \cite{bal2002capillary}.  \\

To measure the spectral energy density of the water waves system, we introduce the Wigner distribution $W^\mu(u,v)(X,\xi)$. For $u,v \in L^2$, the Wigner distribution is given by 
\begin{equation}
    W^\mu(u,v)(X,\xi) = \frac{1}{(2\pi \mu)^2} \int_{\R^2} e^{-iY\cdot \xi/\mu}u(X + Y/2)\overline{v}(X - Y/2)\mathrm{d}Y.
    \label{eq: wigner}
\end{equation}
When $u = v$ we write $W^\mu(u) = W^\mu(u,u) $. 
To see why the Wigner distribution is a suitable spectral energy measure for water waves, we note the following properties, valid for $\mu > 0$:
\vspace{2mm}
\begin{itemize}
    \item[$\circ$] For\footnote{Here $\hat{u}_\mu = (2\pi \mu)^{-n/2}\int_{\R^n} e^{-i X\cdot \xi/\mu} u(X) \mathrm{d}X$ denotes the semiclassical Fourier transform.} $u, \hat{u}_\mu  \in L^1 \cap L^2$ the Wigner distribution $W^\mu(u)(X,\xi)$ is real-valued and it holds that 
    \begin{equation*}
        \int_{\R^2} W^\mu(u)(X,\xi)\mathrm{d}\xi = |u(X)|^2 \quad \text{and}\quad  \int_{\R^2} W^\mu(u)(X,\xi)\mathrm{d}X = |\hat{u}_\mu(\xi)|^2.
    \end{equation*}
    For a closed, convex set $U$ such that $\text{supp}(u)  \subset U $, it holds that $W^\mu(u)(X,\xi) = 0$ when $X \notin U$. Conversely, if $\text{supp}(\hat{u}_\mu)  \subset U $, then $W^\mu(u)(X,\xi) = 0$ for  $\xi \notin U$.
    \item[$\circ$] For a WKB solution $u = Ae^{iS/\mu}$ we have (in the sense of distributions) that \[W^\mu(u)(X,\xi) = |A(X)|^2\delta(\xi -\nabla_X S(X)) \quad \text{as} \quad \mu \to 0. \]    
\end{itemize}
Together, these properties show how the Wigner distribution generalizes the WKB approach and yields a spectral energy measure: if a wave $\eta$ is of WKB type, the Wigner distribution of $\eta$ is concentrated around the phase-space location $(X,\bm{k}) = (X,\nabla_X S) \in \R^{2}\times \R^2$, where $W^\mu(u)(X,\bm{k})|_{\bm{k} = \nabla_X S(X)} \approx E(X)$. More generally, the marginalization and support formulas show that $W^\mu(u)$ provides a phase space representation such that marginals recover the spatial and spectral densities.

To apply the Wigner distribution to the water waves system, we will transform our asymptotic system and work with the energy variables $\psi = \frac{1}{2}(\sqrt{g}\eta + i \sqrt{\Gw^\mu}(b) \varphi)$, where $\sqrt{\Gw^\mu}(b) = \Opw^\mu(g_b^{1/2})$. As we will soon verify, applying $W^\mu$ to $\psi$ gives the decomposition 
\begin{equation}
    W^\mu(\psi) = \frac{1}{4}gW^\mu(\eta) + \frac{1}{4}W^\mu(\varphi,\Gw^\mu \varphi) +\frac{\sqrt{g}}{2}\text{Im} W^\mu(\eta,\sqrt{\Gw^\mu} \varphi) + \mathcal{O}(\mu).
\end{equation}
For $\eta, \varphi$ on WKB form we use the expansion of $\sqrt{\Gw^\mu}(b)$ together with the marginalization property and WKB localization of $W^\mu$ to find that $\int_{\R^2} W^\mu(\psi) \mathrm{d}\xi = \int_{\R^2} E \mathrm{d}X + \mathcal{O}(\mu)$ and that $W^\mu(\psi) = E(X)\delta(\bm{k} -\nabla_X S(X)) $ as $\mu \to 0$. 
The main drawback of the Wigner distribution is that it may take on negative values, and for this reason it is often referred to as a pseudo-energy density. We refer to \cite{grochenig2001foundations, bal2006lecture, gerard1997homogenization, jin2011mathematical} for more details on the Wigner distribution and proofs of the properties listed above.
Still, based on these considerations, we define the spectral energy density and spectral action density as, respectively, 
\begin{equation}
    E_W(X,\bm{k}) = W^{\mu}(\psi)(X,\bm{k}) \quad \text{and} \quad \mathcal{A}_W(X,\bm{k}) = \frac{E_W(X,\bm{k})}{\sigma}.
\end{equation} 
The following proposition gives the asymptotic phase-space dynamics of $E_W$ and $\mathcal{A}_W$ governed by the water waves system \eqref{eq: linerized2d}. We recall that the phase-space Poisson bracket is defined by $\{u,v\} = \nabla_{\bm{k}}u \cdot \nabla_X v - \nabla_X u \cdot \nabla_{\bm{k}} v$ and that $\sigma = \sqrt{g |\bm{k}|\tanh(b(X)|\bm{k}|)}$ and $\omega = \sigma + \bar{U}\cdot \bm{k}$.
\begin{proposition}
The spectral energy density $E_W$ of the water waves system \eqref{eq: linerized2d} satisfies the phase-space evolution equation 
\begin{equation}
    \partial_t E_W + \big\{\omega,E_W \big\} = \frac{ \partial_t \sigma + \{\omega, \sigma\}}{\sigma} E_W + \mathcal{O}(\mu),
    \label{eq: spectral energy}
\end{equation}
Consequently, the spectral action density $\mathcal{A}_W$ satisfies the equation 
\begin{equation}
    \partial_t \mathcal{A}_W + \big\{\omega,\mathcal{A}_W \big\} = \mathcal{O}(\mu). 
\label{eq: action balance}
\end{equation}
\label{prop: action balance}
\end{proposition}
Equation \eqref{eq: action balance} is often referred to as the action balance equation in oceanography, and equipped with suitable source terms, it is the main model used in wave forecasting \cite{komen1994dynamics}. A recent article discussing its origin, interpretation and applications (together with a formal derivation using the Wigner distribution) is \cite{akrish2023mild}.

\subsubsection{Numerical examples}
Using \eqref{eq: action balance}, we expand the ray tracing approach to phase-space $(X,\bm{k})$. The characteristics of \eqref{eq: spectral energy} are given by 
\[ \dot{X} = \nabla_{\bm{k}} \omega(X,\bm{k}), \quad \dot{\bm{k}} = -\nabla_X \omega(X,\bm{k}).  \]
As $\nabla_{\bm{k}} \omega$ and $\nabla_X \omega$ are smooth vector fields, the characteristic curve $(X(t),\bm{k}(t))$ is a smooth path in $\R^{2} \times \R^2$, and we now apply the phase space formulation to study wave blocking.  
\begin{center}
\vspace{2mm}
    \underline{\textbf{Wave blocking}}
    \vspace{2mm}
\end{center}
When waves propagate on an opposing, accelerating current, an interesting phenomenon known as \emph{wave blocking} occurs: at the point where the opposing current velocity becomes equal in magnitude to the waves group velocity, the ray tracing equations predicts a stationary point $\dot{X}(t) = \bar{U} + C_g = 0$, i.e., a point which the wave cannot pass. Consequently, the ray tracing equations predict an unbounded increase of energy/wave action at this point \cite{peregrine1976interaction,phillips1977upperocean}. This behavior is then interpreted as causing wave breaking (and therefore being far outside the scope of linear wave theory) \cite{phillips1977upperocean}, or analyzed further using asymptotic matching methods \cite{peregrine1976interaction}. The latter analysis shows that under certain scenarios, the wave increases in steepness and eventually turns around and propagates in the same direction as the current, while continually getting steeper. It is claimed that such behavior is realistic only for incoming waves with very small initial steepness. \\

We now examine this phenomena numerically, using the wave system \eqref{eq: linerized2d} and the spectral energy equation. 
We consider the following setup: let $\Omega_c = [0,L]$ with $L= 2000 \ m $ m, a constant depth $b = 20 \ m $, and set the current to be 
\[\bar{U}_1(x_1) = -5\frac{x_1^2}{L^2},\]
with units $m/s$. 
As in the wave action simulation, we take as initial conditions for the wave simulation $\varphi_0 = 0$ and 
\[\eta_0(x_1) = \frac{1}{2}B(x_1; x_0,0.04L)\cos(k_0(x_1-x_0))  \text{ with }  k_0 = \frac{2\pi}{(L/60)}, \quad  x_0 = 3L/10.  \]
Moreover, we solve phase-space flow for the spectral action equation, 
\begin{equation}
    \dot{x} = \partial_k \omega(x,k), \quad x(0) = x_0,  \quad \dot{k} = -\partial_x \omega(x,k), \quad k(0) = k_0.  
\end{equation}
The initial conditions are set so that we trace the evolution of the peak energy density $\mathcal{E}$ propagating to the left. Along the trajectory $(x(t),k(t))$ we compute the phase averaged energy by $E(t) = \sigma(k(t),x(t))\mathcal{A}_0$. Throughout the propagation, we compute the Wigner distribution $W(\psi)$ with $\psi = \sqrt{g}\eta + i\sqrt{\Gw}\varphi $ numerically using the trapezoidal rule and a truncated version of \eqref{eq: wigner}. At each time-step, we extract the position $(x,k)$ of the maximum of $W(\psi)(x,k)$. The numerical results are shown in  \Cref{fig:phase space}. The numerical simulation shows that the waves eventually turns and starts propagating backwards, while undergoing a dramatic increase in both wavenumber and amplitude both right before and after the turning point, and that this behavior is well represented by the phase space evolution. The dramatic wave transformation indicates that the physics is well beyond the limitations of linear theory. \\
\newline 
\begin{figure}[htbp]
    \centering
    \includegraphics[width=1\linewidth]{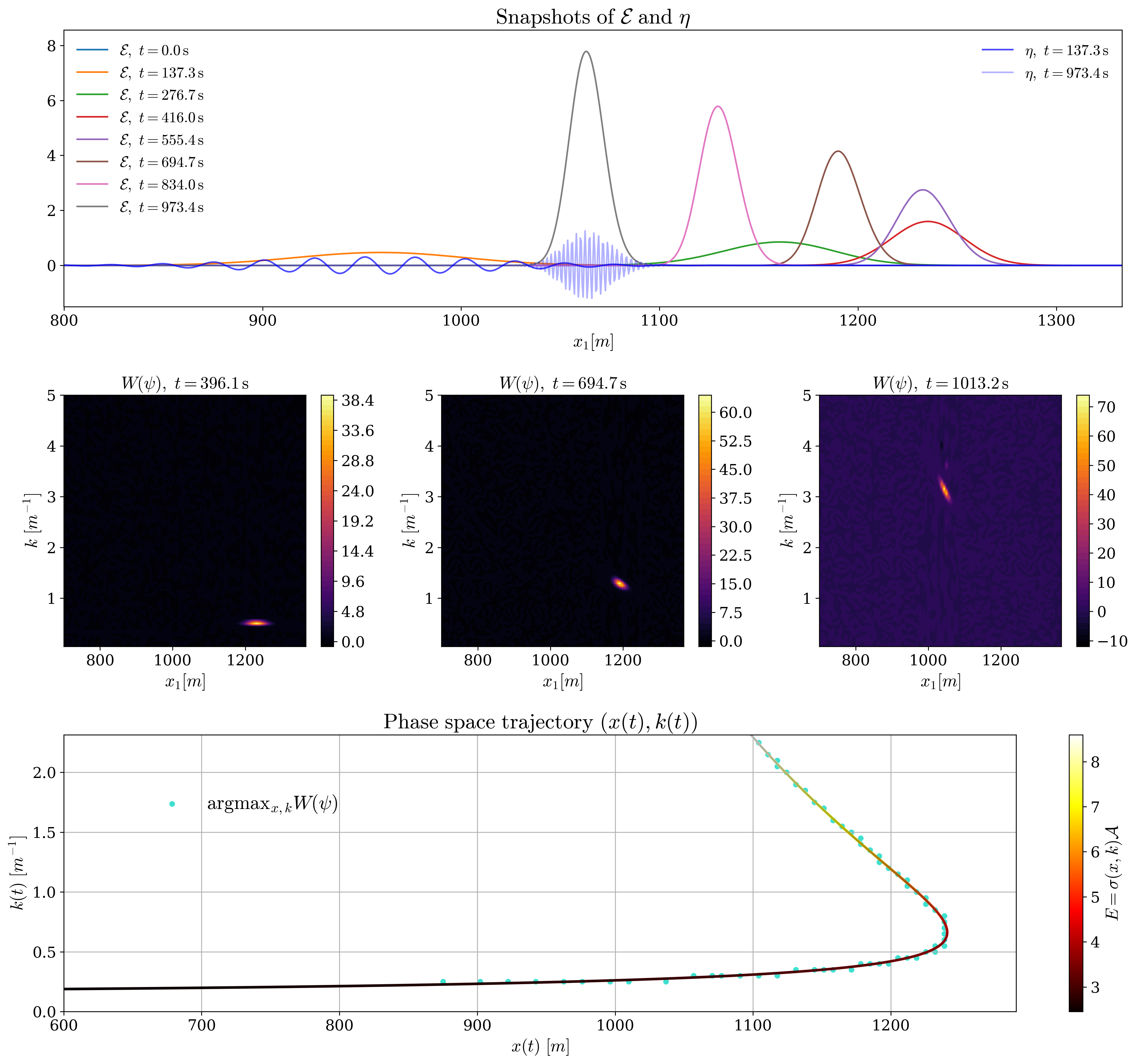}
    \caption{The top panel shows several snapshots of the evolution of $\mathcal{E}$. Initially, the wave propagates to the right, before it stops and eventually starts propagating backwards at $x_1 \approx 1240 \ m$. The energy density $\mathcal{E}$ grows rapidly around the stopping point, and continues to grow as the wave starts propagating to the left. In the lower panel, the behavior is seen in phase space; the trajectory $(x(t),k(t))$ changes direction at the stopping point, and the wavenumber keeps increasing as the wave propagates against the current. The color represents the value of $E$ along the trajectory, and matches the simulated $\mathcal{E}$ quite well. The middle panels shows the Wigner distribution at three different times; one can clearly see how it is concentrated around the phase space coordinates $(x(t),k(t))$ from the characteristics, and this becomes even more apparent in the lower panel, where we have plotted the coordinates of $\text{argmax}_{x,k} W(\psi)$.   }
    \label{fig:phase space}
\end{figure}

We give a rigorous derivation of \Cref{prop: action balance} from the asymptotic water waves system. 
\begin{proof}
The semiclassical Weyl quantization of the symbol  $i\bar{U} \cdot \xi$ is $\Opw^\mu(i\bar{U} \cdot \xi) = \bar{U} \cdot \nabla_X^\mu + \frac{1}{2}\nabla_X^\mu \cdot \bar{U}$. Therefore, we may write \eqref{eq: linerized2d} asymptotically as 
\begin{equation}
    \partial_t^\mu V + \Opw^\mu(s) V = 0,  \quad s(X,\xi,t) = \begin{bmatrix}
        i\bar{U}\cdot \xi + \frac{\mu}{2}\nabla_X \cdot \bar{U} & -g_b(X,\xi) \\
        g & i\bar{U}\cdot \xi - \frac{\mu}{2}\nabla_X \cdot \bar{U}
    \end{bmatrix}
\end{equation}
We now want to diagonalize the above system to order $\mathcal{O}(\mu)$. In the following, we will work with powers of $g_b(X,\xi)$. By the same reasoning as in \Cref{prop: G_pert}, this can be made rigorous by adding to $g_b(X,\xi)$ a bump function supported on a fixed, arbitrarily small ball around $\xi = 0$. To avoid additional notation, we assume such a modification but keep the notation $g_b$. The symbol then becomes positive and $\Opw^\mu(g_b^{\alpha})$ is elliptic and invertible  for $\alpha \geq 0$ (Thm. 29, \cite{zworski2012semiclassical}). Moreover, the composition formula for semiclassical Weyl operators is \[\Opw^\mu(a)\Opw^\mu(b) = \Opw^\mu(a\#b) \quad \text{with} \quad  a \# b = ab + \frac{\mu}{2i}\{ a,b\} + \mathcal{O}(\mu^2).\] It  follows that  $\Opw^\mu(g_b^{\alpha})\Opw^\mu(g_b^{-\alpha}) = I + \mathcal{O}(\mu^2)$ and that $\Opw^\mu(g_b^{1/2})\Opw^\mu(g_b^{1/2}) = \Gw^{\mu}(b)  + \mathcal{O}(\mu^2)$. We will use the notation $\sqrt{\Gw^\mu}(b) = \Op_{\text{\tiny{W}}}^\mu(\sqrt{g_b(X,\xi)})$

We first want to transform our equation into the energy variables $\psi_\pm = \sqrt{g}\eta \mp i \sqrt{\Gw^\mu}(b) \varphi$.
We therefore introduce the operator
\begin{equation}
     M = \Opw^\mu(m), \quad m(X,\xi) = \begin{bmatrix}
        \sqrt{g} & -i\sqrt{g_b} \\
        \sqrt{g} & i\sqrt{g_b}
    \end{bmatrix}, \quad M^{-1} = \Opw^\mu(m^{-1}), 
\end{equation}
and note that  $M^{-1}M = I + \mathcal{O}(\mu^2)$.
Defining $Q = \Opw^\mu(m) V$, we get 
\[\partial_t^\mu Q  + M\Opw^\mu(s) M^{-1} Q - (\partial_t^\mu M)M^{-1}Q = \mathcal{O}(\mu^2).\]
We now make the decomposition
\[s = i\bar{U}\cdot \xi I + \frac{\mu}{2} \nabla_X \cdot \bar{U}\begin{bmatrix}
    1 & 0 \\
    0 & -1
\end{bmatrix} + \begin{bmatrix}
    0 & -g_b \\g & 0  
\end{bmatrix} = s_1 + \mu s_0 + s_g, \text{ (respectively)}. \]
Using the composition estimates for $\sqrt{\Gw^\mu}$ and $\Gw^\mu$, one finds that 
\[M \Opw^\mu(s_g)M^{-1} = \Opw^\mu(\text{diag}(i\sqrt{gg_b},-i\sqrt{g g_b})) + \mathcal{O}(\mu^2)\]
and that 
\[\mu M \Opw^\mu(s_0)M^{-1} =  \mu\frac{1}{2}\nabla_X \cdot \bar{U} J + \mathcal{O}(\mu^2), \]
where $J$ is the $2\times 2$ anti-diagonal identity matrix.
We therefore get 
\[ M \Opw^\mu(s)M^{-1} = \Opw^\mu(D_1 + \mu R) + \mathcal{O}(\mu^2), \quad  D_1 = \begin{bmatrix}
    i(\bar{U}\cdot \xi - \sqrt{gg_b}) & 0 \\ 0 & i(\bar{U}\cdot \xi + \sqrt{gg_b})
\end{bmatrix}, \]
and \[R =  \frac{1}{2}\nabla_X \cdot \bar{U} J+ \frac{1}{2i}\left( \{m,s_1\}m^{-1} + \{ms_1,m^{-1}\} \right). \]
For $2\times 2$ matrix symbols we have $\{a,b\}_{i,j} = \sum_{k = 1}^2 \{a_{i,k},b_{k,j} \}$. 
Hence $\{m,m^{-1}\}_{i,j} = 0$ for $i,j = 1,2$. Consequently, $\{m s_1,m^{-1}\} = m\{s_1,m^{-1}\} +\{m,m^{-1}\}s_1 = m\{s_1,m^{-1}\}$. Moreover, {\small \[\{s_1,mm^{-1}\} = 0 \quad \implies \quad  \{s_1,m\}m^{-1} = - m \{s_1,m^{-1}\} \quad \implies \quad \{m,s_1\}m^{-1} = m\{s_1,m^{-1}\}.  \]  }
All this now allows for a brief calculation:
\begin{align*}
\frac{1}{2i}\left( \{m,s_1\}m^{-1} + \{ms_1,m^{-1}\} \right) &= \frac{1}{i} \{m,s_1\}m^{-1} = \frac{\{ \sqrt{g_b}, \bar{U}\cdot \xi \}}{2\sqrt{g_b}}\begin{bmatrix}
    1 & -1 \\ -1 & 1
\end{bmatrix}.
\end{align*}
Similarly, 
\begin{equation*}
(\partial_t M)M^{-1} = \Opw^{\mu}\left(\frac{\partial_t \sqrt{g_b} }{2\sqrt{g_b} }\right) \begin{bmatrix}
    1 & -1 \\ -1 & 1
\end{bmatrix} + \mathcal{O}(\mu^2).
\end{equation*}
Setting $\gamma =  \frac{1}{2\sqrt{g_b}}(\{ \sqrt{g_b}, \bar{U}\cdot \xi \} - \partial_t \sqrt{g_b}) $, we have now the system 
\[\partial_t^\mu Q + \Opw^\mu( D_1 + \mu \gamma I + \mu(\nabla_X \cdot \bar{U} - \gamma)J ) Q = \mathcal{O}(\mu^2). \]
To get rid of the $\mathcal{O}(\mu)$ off-diagonal terms, we make a second diagonalization. For the diagonalizer, we set $K = \Opw^{\mu}(I + \mu k) $ for some off-diagonal matrix symbol $k$ to be determined, and note that $K^{-1} = \Opw^{\mu}(I - \mu k) + \mathcal{O}(\mu^2).$  With $\tilde{Q} = KQ$, we get the conjugated system
\[ \partial_t \tilde{Q} + K \Opw^\mu(D_1 + \mu (\gamma I + (\nabla_X \cdot \bar{U} - \gamma)J) )K^{-1} \tilde{Q} = \mathcal{O}(\mu^2).\]
Since all non-trivial commutator terms above are $\mathcal{O}(\mu^2)$, the $\mathcal{O}(\mu)$ symbol of the above operator becomes
\[ D_1 + \mu (\gamma I + (\nabla_X \cdot \bar{U} - \gamma)J) + \mu(kD_1 - D_1k). \]
We now want to choose $k$ such that the off-diagonal terms in this symbol vanish. This is achieved by requiring that 
\[ (kD_1 - D_1k) + (\nabla_X \cdot \bar{U} - \gamma) J = 0 \implies \begin{cases}
    -2ik_{1,2}\sqrt{g g_b} = -\nabla_X \cdot \bar{U} + \gamma, \\
    2ik_{2,1}\sqrt{g g_b} = -\nabla_X \cdot \bar{U} + \gamma. 
\end{cases} 
    \]
Choosing $k$ accordingly, our $\mathcal{O}(\mu)$ diagonalized equation becomes 
\begin{equation}
    \partial_t^\mu \tilde{Q}_\pm = \Opw^\mu(d_\pm) \tilde{Q}_\pm +\mathcal{O}(\mu^2), \quad d_\pm = -i(\bar{U}\cdot \xi \mp \sqrt{gg_b}) - \mu\gamma.
\end{equation}
We now pick the forward branch $\tilde{\psi} = \tilde{Q}_-$ and consider $W^\mu(\tilde{\psi})$. Differentiating in time gives
\begin{equation*}
    \partial_t^\mu W^\mu(\tilde{\psi}) = W^\mu(\partial_t\tilde{\psi},\tilde{\psi}) + W^\mu(\tilde{\psi},\partial_t\tilde{\psi}) = W^\mu(\Opw^\mu(d_-) \tilde{\psi},\tilde{\psi}) + W^\mu(\tilde{\psi}, \Opw^\mu(d_-) \tilde{\psi}) +\mathcal{O}(\mu^2).
\end{equation*}
Using the expansions (cf. \cite{gerard1997homogenization}) 
\begin{align*}
    W^\mu(\Opw^\mu(a)u,v) &= a W^\mu(u,v) + \frac{\mu}{2i}\{a,W^\mu(u,v) \} + \mathcal{O}(\mu^2), \\
    W^\mu(u,\Opw^\mu(a)v) &=  W^\mu(u,v)\overline{a} + \frac{\mu}{2i}\{W^\mu(u,v),\overline{a}\} + \mathcal{O}(\mu^2), 
\end{align*}
we find that 
\begin{equation*}
    \partial_t^\mu W^\mu(\tilde{\psi}) = -\mu \{ (\bar{U}\cdot \xi + \sqrt{gg_b}),W^\mu(\tilde{\psi}) \} -  2\mu\gamma W^\mu(\tilde{\psi}) + \mathcal{O}(\mu)
\end{equation*}
Next, we note that $W^\mu(\tilde{\psi}) = W^\mu(\psi) + \mathcal{O}(\mu)$. Omitting the constant scaling, and using again the above expansions on $W^\mu(\sqrt{\Gw^\mu}\varphi,\sqrt{\Gw^\mu} \varphi)$ we therefore get
\begin{equation}
   E_W =  W^\mu(\tilde{\psi}) + \mathcal{O}(\mu) = gW^\mu(\eta) + W^\mu(\varphi,\Gw^\mu \varphi) + 2\sqrt{g}\text{Im}W^\mu(\eta,\sqrt{\Gw^\mu} \varphi) + \mathcal{O}(\mu).
\end{equation}
Recalling that $\sigma = \sqrt{gg_b}$ and  $\omega = \sigma + \bar{U}\cdot \bm{k}$, we note that \[ \{\omega,\sigma\} = \{\sigma,\sigma\} + \{\bar{U}\cdot \bm{k},\sigma\} = - \{\sigma,\bar{U}\cdot \bm{k}\}.\]
We therefore write $\gamma = -\frac{\partial_t \sigma + \{\omega,\sigma\} }{2\sigma}$, and get after dividing by $\mu$ the following evolution equation for $E_W$ 
\[
\partial_t E_W + \big\{ \omega, E_W \big \} =  \frac{\partial_t \sigma + \{\omega,\sigma\} }{\sigma}E_W + \mathcal{O}(\mu). 
\]
Last, the spectral action equation follows by rearrangement, 
since  \[\partial_t \frac{E_W}{\sigma} = \frac{\sigma \partial_t E_W - \partial_t \sigma E_W }{\sigma^2 } \quad\text{and}\quad  \big\{ \omega, \frac{E_W}{\sigma}\big \} = \frac{\sigma \{\omega,E_W\} - E_W\{\omega,\sigma\} }{\sigma^2}. \]
\end{proof}

\section{Conclusion}
The initial question that led us to writing this article was a desire to clearly understand the route from the full Euler equations for surface waves to the conservation law for wave action. In the pursuit of an answer, we have established a rigorous framework for linear surface gravity waves propagating over variable currents and bathymetry. Starting from the Zakharov-Craig-Sulem surface formulation justified for background currents with weak, localized vorticity, we demonstrated the well-posedness of the governing equations using the theory of hyperbolic systems of pseudo-differential operators. The central mathematical mechanism enabling our subsequent asymptotic analysis is the semiclassical Weyl quantization of the DN operator. We showed that this quantization is asymptotically accurate for slowly varying environments and, crucially, preserves the self-adjoint structure on $L^2$ required for consistent energy dynamics. 

Using this exact pseudo-differential framework, we derived a novel equation for the evolution of the total wave energy. This non-asymptotic formulation explicitly isolates the production term, $\mathcal{P}= -\bm{u}\cdot S(\bar{\bm{U}})\bm{u}$ which governs the physical energy exchange between the wave velocity and the background flow's strain tensor. We then systematically recovered various classical wave models as specific asymptotic limits of our leading-order system. By expanding the asymptotic DN operator, we derived an evolution equation for the phase-averaged energy density that directly corresponds to the classical wave action equation and its geometric ray-tracing characteristics.

Further analysis of the asymptotic DN operator allowed us to rigorously incorporate diffractive effects into the phase-averaged models. From this, we derived the mild-slope equation for time-harmonic waves and a linear Schrödinger-type equation that captures the dispersive spreading of wave packets. Finally, to analyze phase-space dynamics, we applied the Wigner transform to the diagonalized system and derived the action balance equation. These theoretical derivations were supported by numerical experiments. The simulations quantitatively verified the exact total energy evolution, demonstrated the precision of the Schrödinger model in capturing dispersive decay over variable bathymetry, and illustrated the phase-space resolution of wave turning against accelerating currents.\\
As the derived models are central to linear wave-current-bathymetry interaction, we believe our asymptotic surface variable model constitutes a truly unified approach to the topic.\\
Building on this framework, a natural, but highly non-trivial, next step is to incorporate weakly non-linear wave interactions, while trying to retain a simple and useful model. Moreover, extending the diffractive phase-space models to domains with physical boundaries using diffractive geometric optics would significantly increase the applicability of the Schrödinger model. Also, allowing for background currents with strong vorticity and considering rotational wave perturbations gets very complex also in the linear setting \cite{liu2025spectra}, and if the fluid is viscous, even initially irrotational waves may develop vorticity due to interaction with the bottom \cite{riquier2026irrotational}.

\section*{Acknowledgments}
A.K. acknowledges the financial support from internal grant Nr. 102158 at Simula Research Laboratory. T.H. acknowledges the financial support from the Norwegian Space Agency through the CoastCurr project (Nb. 74CO2501)

\section*{Appendix: Wave simulation}
To support this article, we have made available the Python/Jupyter Notebook code for numerical study of wave-current-bathymetry interaction, available at \href{https://github.com/jfkirkeby/WaCuBa}{https://github.com/jfkirkeby/WaCuBa}. The code has the following features: 
\vspace{2mm}
\begin{itemize}
    \item[$\circ$] Solves the Cauchy problem for \eqref{eq: cauchy} with prescribed initial conditions $(\eta_0,\varphi_0)$ and variable bathymetry $b(X)$ and current $\bar{U}(X)$ on a rectangular domain. Returns $\eta(t,X),\varphi(t,X) $ and $\mathcal{E}(t,X)$ and additional wave features. 
    \item[$\circ$] Uses ray tracing to compute wavenumber fields $\bm{k}(X)$ for given bathymetry $b(X)$ and current $\bar{U}(X)$.
    \item[$\circ$] Computes all intrinsic wave properties and solves the wave action equation \eqref{eq: wave action 1} and the Schrödinger equation \eqref{eq: eta-schrödinger}. 
    \item[$\circ$] Allows for flexible and easy comparison of results from different models. 
\end{itemize}
\vspace{2mm}
We solve both the wave system and the energy and Schrödinger equations using a standard Fourier pseudo-spectral method (cf. \cite{trefethen2000spectral}); we express our unknowns in a truncated Fourier basis, e.g., $\eta_N(t,X) = \sum_{|\bm{k}| \leq N} \eta_{\bm{k}}(t)e^{i\bm{k}\cdot X}$, compute spatial derivatives in the $\bm{k}$-domain, and transform back to physical space for multiplication by vector fields and time stepping. As the current $\bar{U}$ is assumed to be smooth (and there are no non-linear terms), we do not enforce de-aliasing. For the DN operator, we use the truncated Fourier-Galerkin method developed in \cite{andrade2018three}. We precompute the bathymetry dependent part of the $\mathcal{G}(b)$, and we also implement absorbing boundary conditions \cite{bodony2006analysis}. For time integration of the PDEs we use the standard Runge-Kutta 4 scheme and for wavenumber computation, we incorporate the open source ray tracing module \cite{halsne2023ocean}. The solver has been verified numerically by considering convergence as a function of grid size/Fourier modes. Although the wave simulation is the backbone in our numerical experiments, the publicly available code does not, for reasons of readability, support all functionality used to produce the results in the paper. For more details, we refer to the user guide accompanying the code. 
\section*{Appendix: Proofs}
\begin{proof}{\Cref{prop: potpert}}
We first show that $\bm{\omega} = \mathcal{O}(\varepsilon \delta)$ and that $\mathrm{supp}(\bm{\omega})$ is compact for any finite time. By assumption $\bm{u} = \mathcal{O}(\varepsilon)$ and the linearized vorticity equation is 
\begin{equation}
(\partial_t  + \bar{\bm{U}}\cdot \nabla_{X,z} )\bm{\omega} - (\bm{\omega} \cdot \nabla_{X,z}) \bar{\bm{U}} = r(\bm{u},\bar{\bm{U}},\bar{\bm{\omega}}) + \mathcal{O}(\varepsilon^2),
\label{eq: vort evo}
\end{equation}
with $  r = (\bar{\bm{\omega}}\cdot \nabla_{X,z})\bm{u} - (\bm{u}\cdot \nabla_{X,z})\bar{\bm{\omega}} = \mathcal{O}(\varepsilon \delta)$  and $ \mathrm{supp}(r) \subset \{(X,z): |X| < R\}. $ Next, let $\chi(t,(X_0,z_0))$ be the characteristic vector field of $\bar{\bm{U}}$ given by $\dot{\chi(t;(X,z))}=\bar{\bm{U}}(t,\chi(t;(X,z))), \chi(0;(X,z)) = (X,z)$. 
Writing $\bm{\omega}(t) = \bm{\omega}(t,\chi(t,(X,z)))$ etc., we have   \[\frac{\mathrm{d}}{\mathrm{d}t}\bm{\omega}(t) = (\bm{\omega}(t) \cdot \nabla_{X,z}) \bar{\bm{U}}(t) + r(t). \]
 Since $r$ is compactly supported and  $\bm{\omega}|_{t=0} = 0$ it follows that \[\mathrm{supp}(\bm{\omega}(t,\cdot)) \subset V(t) = \{ (X,z) \subset \Omega(0,b) : |X| < R + \|\bm{U}\|_{L^\infty} t \}.\]
Moreover, with $M_1 = \sup_{t \leq T}\|\nabla_{X,z} \bar{\bm{U}}(t,\cdot)\|_{L^\infty} $ and $M_2 = \sup_{t \leq T}\|r(t,\cdot)\|_{L^\infty}$ we get 
\[ \frac{\mathrm{d}}{\mathrm{d}t}|\bm{\omega}(t)| \leq M_1 |\bm{\omega}(t)| + M_2 \quad \implies \quad  |\bm{\omega}(t)| \leq \frac{M_2}{M_1}\left( e^{M_1t} -1\right). \]
By assumption $T = \mathcal{O}(1)$ and so $\sup_{t\leq T} \|\bm{\omega}(t,\cdot)\|_{L^\infty} \leq \frac{M_2}{M_1}\left( e^{M_1t} -1\right) = \mathcal{O}( \varepsilon \delta) $.
Next, we differentiate \eqref{eq: vort evo}, and find $D^\alpha \bm{\omega}$ ($|\alpha| = 1)$ satisfies \[(\partial_t  + \bar{\bm{U}}\cdot \nabla_{X,z} )D^\alpha\bm{\omega} + (D^\alpha\bar{\bm{U}}\cdot \nabla_{X,z} )\bm{\omega} -  (D^\alpha\bm{\omega} \cdot \nabla_{X,z} \bar{\bm{U}}) =  r^1(\bm{\omega},\bm{u},\bar{\bm{U}},\bar{\bm{\omega}}). \]
Since $D^\alpha \bm{u} = \mathcal{O}(\varepsilon)$ and $\bm{\omega} = \mathcal{O}(\varepsilon \delta)$, $r^1$ contains only bounded terms, and one can check that $  r^1 = \mathcal{O}(\varepsilon \delta)$ and $ \mathrm{supp}(r^1) \subset \{(X,z): |X| < R\}.$ Since $|(D^\alpha\bar{\bm{U}}\cdot \nabla_{X,z} )\bm{\omega}| \leq C\delta \max_{|\alpha| = 1}|D^\alpha \bm{\omega}|$,  taking the maximum over indices $\alpha$ and applying the same argument as above therefore shows that $D^\alpha \bm{\omega} = \mathcal{O}(\varepsilon \delta)$, and repeating one more time gives $D^\alpha \omega = \mathcal{O}(\varepsilon \delta)$ for $|\alpha| = 2$.

We now consider the so-called div-curl problem to reconstruct $\bm{u}$ from $\omega$ in $\Omega(0,b)$, following the works \cite{melinand2017coriolis, castro2015well}. We first decompose the horizontal component of $\bm{u}$ at the surface. Writing $\bm{u}_h = (u_1,u_2)_0$, then by the Helmholtz-Hodge decomposition there exists a unique decomposition 
\[ \bm{u}_h = \nabla_X \varphi + \bm{v}, \quad \nabla_X \cdot \bm{v} = 0,\]
where $\varphi$ and $\bm{v}$ are given by 
\[ \varphi = \Delta_{X}^{-1}( \nabla_X \cdot \bm{u}_h ) \quad \text{and} \quad   \bm{v} = \nabla_X^\perp\Delta_{X}^{-1} (\nabla_X^\perp \cdot \bm{u}_h).    \]
Above, $\Delta_{X}^{-1}$ denotes the fundamental solution to the Laplace operator, and $\nabla_X^\perp = (-\partial_{x_2},\partial_{x_1})^\top$. Note that $\nabla_X^\perp  \cdot \bm{u}_h = (\omega_3)_0$.
We now introduce a harmonic extension $\phi$ of $\varphi$ by 
\[ \Delta_{X,z} \phi = 0 \quad \text{in} \quad \Omega(0,b), \quad \partial_\nu \phi|_{z= -b} = 0, \quad  \phi|_{z=0} = \varphi,\]
and set $\tilde{\bm{u}} = \bm{u} - \nabla_{X,z} \phi$. Hence $\tilde{\bm{u}}$ is the rotational component of $\bm{u}$.  At the surface we have $\nabla_X \cdot \tilde{\bm{u}}_h = \nabla_X \cdot ( \bm{u}_h - (\nabla_{X} \phi)_0 )= 0$, and $\tilde{\bm{u}}_h = \bm{v}$. In terms of $\bm{\omega}$, we have
\[ \tilde{\bm{u}}_h= \nabla_X^\perp\Delta_{X}^{-1} (\omega_3)_0  = \int_{\R^2} \frac{(X-Y)^\perp}{|X-Y|^2} \omega_3(Y,0) \mathrm{d}Y, \quad \text{with} \quad  (x_1,x_2)^\perp =(-x_2,x_1).\]
Since $\bm{\omega} $ is supported in $V(t)$ we obtain the bound $\|\tilde{\bm{u}}_h\|_{L^\infty(\R^2)} \leq c(V(t)) \|\omega\|_{L^\infty}$ and 
$|\tilde{\bm{u}}_h| = \mathcal{O}(|X|^{-1}) $ as $|X| \to \infty$. 
With the tangential boundary, we get the following div-curl system for $\tilde{\bm{u}}$:
\begin{equation}
    \begin{dcases}
        \nabla_{X,z} \cdot \tilde{\bm{u}} = 0 & \quad \text{in} \quad \Omega(0,b),  \\
        \nabla_{X,z} \times \tilde{\bm{u}} = \bm{\omega} & \quad \text{in} \quad \Omega(0,b),\\
        \tilde{\bm{u}}_h= \nabla_X^\perp\Delta_{X}^{-1} (\omega_3)_0 &\quad \text{at} \quad z = 0, \\
        \tilde{\bm{u}} \cdot \nu = 0 &\quad \text{at} \quad z = - b.
    \end{dcases}
\end{equation}
Since $\bm{\omega}$ is in $H^1$, the existence and uniqueness (in $H^1$) of a solution $\tilde{\bm{u}}$ to the above equation follows from Theorem 2.8 in \cite{melinand2017coriolis}. 

We now establish a pointwise bound on $\tilde{\bm{u}}$ by decomposing the velocity field into a whole-space potential and a harmonic correction. Let $\bm{u}_\omega$ be given by 
\[ \bm{u}_\omega(X,z) = \frac{1}{4\pi} \int_{\R^3} \frac{\bm{\omega}((Y,z')) \times ((X,z)-(Y,z'))}{|(X,z)-(Y,z')|^3} \mathrm{d}Y\mathrm{d}z',  \]
where $\bm{\omega}$ is smoothly extended to zero on $\R^3 \setminus \overline{\Omega(0,b)}$.
Since $\bm{\omega} \in L^\infty(\R^3)$ and $\mathrm{supp}(\bm{\omega}) \subset V(t)$ we get immediately that  $\|\bm{u}_\omega\|_{L^\infty(\Omega(0,b))} \leq C(V) \|\omega\|_{L^\infty(\Omega)}$  and that  $ |\bm{u}_\omega| = O(|X|^{-2})$ as $|X| \to \infty$. Differentiating the convolution operator and moving the derivative onto $\bm{\omega}$ shows that the same bounds hold for $D^\alpha \bm{u}_\omega$ with $|\alpha| \leq 2$. 
We define the correction $\bm{u}^* = \tilde{\bm{u}} - \bm{u}_\omega$. Since $\nabla_{X,z} \times \bm{u}^* = 0$ and $\nabla_{X,z} \cdot \bm{u}^* = 0$ and $\Omega(0,b)$ is simply connected, there exists a potential $\Phi^*$ such that $\bm{u}^* = \nabla_{X,z}\Phi^*$. The boundary conditions for $\bm{u}^*$ are determined by the mismatch between the boundary data:
\begin{equation}
    \begin{cases}
    \bm{u}_h^* = \tilde{\bm{u}}_h - (\bm{u}_\omega)_h & \text{at }  z = 0, \\
    \bm{u}^* \cdot \nu = (\tilde{\bm{u}} - \bm{u}_\omega) \cdot \nu = - \bm{u}_\omega \cdot \nu &  \text{at }  z = -b.
\end{cases}
\end{equation}
Since \[\nabla_X \Phi^*_0 =\tilde{\bm{u}}_h - (\bm{u}_\omega)_h \implies \Delta_X^2 \Phi^*_0 = \nabla_X \cdot (\tilde{\bm{u}}_h - (\bm{u}_\omega)_h) = - \nabla_X \cdot (\bm{u}_\omega)_h, \]
we set $\varphi^* = -\Delta^{-1}(\nabla_X \cdot (\bm{u}_\omega)_h)$ and $\psi^* = - (\bm{u}_\omega \cdot \nu)|_{z=-b} $, and  consider the boundary value problem 
\[ \Delta_{X,z} \Phi^* = 0 \quad \text{in } \Omega(0,b), \quad \Phi^*|_{z=0} = \varphi^*, \quad \partial_\nu \Phi^*|_{z=-b} = \psi^*.  \]
Since both $\varphi^*, \psi^* \in H^1$, it follows by an application of the Lax-Milgram theorem that there exists a unique solution $\Phi^* \in H^1$ to the above problem (cf. Ch. 2, \cite{waterwavesprob}).  We now obtain pointwise estimates on $\Phi^*$ from the boundary data using boundary integral equations. 
We now seek a Green's function $G$ satisfying \begin{equation}
    \begin{cases} \Delta_{X,z} G(X,z,Y',z') = -\delta((X,z) - (Y',z')),  \quad \forall (X,z),(Y',z') \in \mathbb{R}^2 \times (-b_{\text{\tiny{max}}},0), \\
 G(X,z,Y',z')|_{z=0} = 0, \quad \partial_z G(X,z,Y',z')|_{z=-b_{\text{\tiny{max}}}} = 0,
\end{cases}
\label{eq: greens}
\end{equation}
with decay $ G = \mathcal{O}(1/|X - Y|)$ as $ |X|,|Y| \to \infty.$ We know that $G = \frac{1}{4\pi|(X,z)-(Y,z')|} + H$, where $H$ is some harmonic function enforcing the boundary conditions. However, due to the simple geometry, we can use separation of variables and obtain 
\[G(X,z,Y',z') = \frac{1}{\pi b_{\text{\tiny{max}}}}\sum_{n=0}^\infty\cos(k_n(z+b_{\text{\tiny{max}}}))\cos(k_n(z'+b_{\text{\tiny{max}}}))K_0(k_n|X -Y|), \]
with $k_n = (n+1/2)\pi/b_{\text{\tiny{max}}}$. Here $K_0$ is the modified Bessel function of the second kind, and it satisfies $K_0(r) \leq C e^{-r} $ as $r \to \infty$.
Hence, the infinite strip Green's function actually decays like $G = \mathcal{O}(e^{-k_0|X - Y|})$ as $|Y| \to \infty$. 

We define the total bottom and variable bottom as $\Gamma_b = \{(X,-b(X)): X \in \R^2\}$ and $\tilde{\Gamma}_b = \{ (X,-b(X)): X \in \mathrm{supp}(b - b_{\text{\tiny{max}}})\}$, respectively. By Green's identity,  
\begin{align*}
    \Phi^*(X,z) &= \int_{\R^2 \cup \Gamma_b} G(X,z,Y',z') \partial_{\nu(Y')} \Phi(Y',z')  - \partial_{\nu(Y')} G(X,z,Y',z') \Phi(Y',z') \mathrm{d}S(Y') \\
    & = -\int_{\R^2} \partial_{z'} G(X,z,Y',0) \varphi^* \mathrm{d}X + \int_{\Gamma_b} G(X,z,Y',-b(Y')) \psi^*(Y') \mathrm{d}S(Y') \\
    &  -\int_{\tilde{\Gamma}_b} \partial_{\nu(Y')} G(X,z,Y',-b(Y')) \Phi^*(Y',-b(Y')) \mathrm{d}S(Y')
    \label{eq: greens}
\end{align*}
We now introduce, respectively, the single- and double-layer potentials for harmonic functions (Our use of layer potentials follows Ch. 6, \cite{kress1989linear}, and these potentials will be reused in later proofs.).
\begin{align}
    S\rho(X,z) &= \int_{\Gamma_b} G(X,z,Y',-b(Y')) \rho(Y') \mathrm{d}S(Y'), \\
    D \rho(X,z) & =\int_{\tilde{\Gamma}_b}  \partial_{\nu(Y')}G(X,z,Y',-b(Y'))  \rho(Y') \mathrm{d}S(Y'), \\
    D_0 \rho(X,z) & =\int_{\R^2}  \partial_{z'} G(X,z,Y',0)  \rho(Y') \mathrm{d}Y'.
    \label{eq: potential}
\end{align}
In terms of the above operators we may therefore write 
\begin{equation}
    \Phi^*(X,z) = - D_0 \varphi^*(X,z) + S \psi^*(X,z) - D \rho(X,z), \quad \rho = \Phi^*(X,z)|_{z = -b}.
    \label{eq: phi pot}
\end{equation}
Due to the decay of $G$ we find that 
\[ \|D_0 \varphi^*(X,z)\|_{L^\infty} \leq C\|\varphi^*\|_{L^\infty}, \quad  (X,z) \in \R^2 \times [-b(X),0),  \]
and for $S$, the singularity of $G$ is integrable on $\Gamma_b$, and therefore 
\[ \|S \psi^*(X,z)\|_{L^\infty} \leq C \|\psi^*\|_{L^\infty}, \quad (X,z) \in \R^2 \times [-b(X),0]. \]
Using now the jump relation for $D$ (Theorem 6.18, Ch. 6, \cite{kress1989linear}), we find that $\rho$ satisfies the Fredholm equation 
\begin{equation}
    \rho(X) - 2 D\rho (X) = - 2g(X), \quad X \in \tilde{\Gamma}_b,  
    \label{eq: fredholm}
\end{equation}
with $ g(X) = - D_0 \varphi^*(X,-b(X)) + S \psi^*(X,-b(X))$. Since $\tilde{\Gamma}_b$ is compact, $D: C(\tilde{\Gamma}_b) \to C(\tilde{\Gamma}_b)$ is compact (Ch. 2, \cite{kress1989linear}). As the nullspace of \eqref{eq: fredholm} is trivial  (Theorem 6.21, \cite{kress1989linear}), there is a unique $\rho$ satisfying  $\|\rho\|_{L^\infty} \leq C \|g\|_{L^\infty} \leq C(\|\varphi^*\|_{L^\infty} + \|\psi^*\|_{L^\infty})$. Consequently, equation \eqref{eq: phi pot} and the above estimates show that \[\|\Phi^*\|_{L^\infty} \leq C(\|\varphi^*\|_{L^\infty} + \|\psi^*\|_{L^\infty}).\] 
The bound on $\bm{u}^* = \nabla_{X,z} \Phi^*$ now follows from Schauder estimates. Let $\Omega_R \subset \overline{\Omega(0,b)}$ be a smooth, bounded domain, possibly intersecting the boundary, and define $\Gamma_B$ and $\Gamma_0$ to be the parts of $\partial \Omega_R$ intersecting with the bottom and $z= 0$, respectively.
Then for $0< \gamma < 1 $ $\Phi^*$ satisfies (Theorem 6.6, \cite{gilbarg1977elliptic}) 
\[\|\Phi^*\|_{C^{2,\gamma}(\Omega_R)} = C(\Omega_R)\left(\|\Phi^*\|_{L^\infty(\Omega_R)} + \|\varphi^*\|_{C^{2,\gamma}(\Gamma_0)} + \|\psi^*\|_{C^{1,\gamma}(\Gamma_B)} \right).  \   \]
The constant $C(\Omega_R)$ depends on $b$ but since $b$ is smooth and constant outside some compact set, $C(\Omega_R)$ is bounded. Since $\varphi^*$ satisfies $ \Delta_X \varphi^* = - \nabla_X \cdot  (\bm{u}_{\omega})_h $ it follows that $\|\varphi^*\|_{C^{3}(\Gamma_0)} \leq C \|(\bm{u}_{\omega})_h\|_{C^{2}(\Gamma_0)} $, and similarly $\|\psi^*\|_{C^{2}(\Gamma_B)} \leq C \|(\bm{u}_{\omega})_h\|_{C^{1}(\Gamma_0)}$. As have already shown that $ \|D^\alpha \bm{u}_\omega \|_{L^\infty} = \mathcal{O}(\varepsilon \delta)$ it follows that 
\[\|\varphi^*\|_{C^{2,\gamma}(\Gamma_0)} \leq \|\varphi^*\|_{C^{3}(\Gamma_0)} = \mathcal{O}(\varepsilon \delta), \quad  \|\psi^*\|_{C^{1,\alpha}(\Gamma_B)} \leq \|\psi^*\|_{C^{2}(\Gamma_B)} = \mathcal{O}(\varepsilon \delta).\]
Consequently, we invoke the global upper bound on $\|\Phi^*\|_{L^\infty} $ and conclude that 
\[
\|\Phi^*\|_{C^{2}} \leq C\left( \|\varphi^*\|_{L^\infty} + \|\psi^*\|_{L^\infty} +\|\varphi^*\|_{C^{3}} + \|\psi^*\|_{C^{2}} \right) = \mathcal{O}(\varepsilon \delta).
\]
Finally, we have \[|\tilde{\bm{u}} | \leq |\bm{u}^*| + |\bm{u}_\omega| = \mathcal{O}(\varepsilon \delta), \quad |D\tilde{\bm{u}} | \leq |D\bm{u}^*| + |D\bm{u}_\omega| = \mathcal{O}(\varepsilon \delta).  \]
For $\partial_t \tilde{\bm{u}}$, we note that $\partial_t \bm{\omega} = \mathcal{O}(\varepsilon \delta)$ (from \eqref{eq: vort evo}), and therefore $ \partial_t \bm{u}_\omega = \mathcal{O}(\varepsilon \delta)$. Carrying out the exact same analysis with $\partial_t \Phi^*$ then yields $\partial_t \tilde{\bm{u}} = \mathcal{O}(\varepsilon \delta)$.    
\end{proof}

\begin{proof}(\Cref{prop: G_pert})
Since $\tanh(|t|) \leq 1$, we have that  
\begin{equation}
    C(1 + \tau) \leq \gamma(\tau) \leq (1+\tau), \quad C = \min_{\tau \geq 0} \frac{\gamma(\tau)}{1+ \tau} > 0.
    \label{eq: bounds}
\end{equation} 
Moreover, $\tanh(z)$ is holomorphic on the strip $\{z \in \mathbb{C}: |\text{Im} z| \leq \pi/3 \}$, and so it follows that for any 
$\tau \in \mathbb{R}$, we there is a constant $M$ such that 
\[ \left|\frac{\mathrm{d}^k}{\mathrm{d}\tau^k} \tau\tanh(b_{\text{\tiny{max}}}\tau) \right| = Mk!.\]
Hence, for $k \in \mathbb{N}_0 $, $|\partial_\tau^k \gamma(\tau)| \leq c_k$ for some constant $c_k$. For $p\in \mathbb{R}$, it now follows by recursion that  \[\left|\partial_\tau^k(\gamma^p(\tau) \right| \leq (\gamma^{p-k}(\tau)C_k \leq \Tilde{C}_k (1+\tau)^{p-k}. \]
The last inequality holds since if $p-k < 0$ we substitute the lower bound in \eqref{eq: bounds}, while if $p-k \geq 0$ we may use the upper bound. 
Since $\mathcal{G}^p= \text{Op}(\gamma^p)$ is a PDO of order $p$, the Sobolev mapping properties is then standard (cf. \cite{alinhac2007pseudo}). The fact that $\mathcal{G}^p$ is invertible with inverse $\mathcal{G}^{-p}$ is also readily established. Since $\gamma^p(|\xi|) \geq c (1+|\xi|)^p$ for all $\xi$, it is clear that $\mathcal{G}^p$ is one-to-one. To see that $\mathcal{G}^p$ is onto, let $f \in H^{s-p}$. Assume there is some $h \in \mathcal{S}'$ such that $g^p(|\xi|) \hat{h}(\xi) = \hat{f}(\xi)$. Taking $\hat{h}(\xi) = g^{-p}(|\xi|)\hat{f}(\xi)$, we find \[\|h\|_{H^s}^2 = \int_{\mathbb{R}^2}(1+|\xi|)^{2s}|\gamma^{-p}(|\xi|)\hat{f}(\xi)|^2 \mathrm{d}\xi \leq C\int_{\mathbb{R}^2}(1+|\xi|)^{2(s-p)}|\hat{f}(\xi)|^2 \mathrm{d}\xi \leq C \|f\|_{H^{s-p}}^2,\]
and by construction $\mathcal{G}^ph = f$. Last, \[\mathcal{G}^\alpha \mathcal{G}^pf = \mathcal{F}^{-1}\left(\gamma^\alpha(|\xi|)\mathcal{F}(\mathcal{G}^pf)(\xi)\right)  = \mathcal{F}^{-1}\left(\gamma^{\alpha+ p}(|\xi|)\hat{f}(\xi)\right) = \mathcal{G}^{\alpha +p} f. \]
\end{proof}

\begin{proof}( \Cref{prop: G + K})
The solution to \eqref{eq: DN-def} with constant depth $b_{\text{\tiny{max}}}$ is given by
 \[ \Phi_{b_{\text{\tiny{max}}}}(X,z) = \int_{\mathbb{R}^2} e^{iX\cdot \xi} \frac{\cosh(|\xi|(b_{\text{\tiny{max}}} + z))}{\cosh(b_{\text{\tiny{max}}}|\xi|)} \widehat{\varphi}(\xi)\mathrm{d}\xi.\]
Hence the the difference $\phi  = \Phi - \Phi_{b_{\text{\tiny{max}}}}$ satisfies 
\begin{equation}
    \Delta \phi = 0 \quad \text{in } \R^2 \times (-b(X),0), \quad \phi|_{z= 0} = 0, \quad \partial_\nu \phi|_{z = - b(X)} = g(X),
    \label{eq: harmonic diff}
\end{equation} 
where\footnote{We disregard the normalization in $\partial_\nu$ as it is only a scaling by a smooth, positive function.} $g(X) = \partial_z \Phi_{b_{\text{\tiny{max}}}}|_{z = -b(X)} + \nabla_X b(X) \cdot \nabla_X \Phi_{b_{\text{\tiny{max}}}}|_{z = -b(X)}$ . Clearly, $\text{supp }(g) = \tilde{\Gamma}_b$. Defining 
{\small
\[ m_z(X,\xi) = \frac{|\xi|\sinh(|\xi|(b_{\text{\tiny{max}}} - b(X)))}{\cosh(|\xi|b_{\text{\tiny{max}}}) } \quad \text{and} \quad  m_X(X,\xi) = i\xi \cdot \nabla_X b(X)  \frac{\cosh(|\xi|(b_{\text{\tiny{max}}} - b(X)))}{\cosh(|\xi|b_{\text{\tiny{max}}})}, \]}
we write $ g(X) = \text{Op}(m_z)\varphi(X) + \text{Op}(m_X)\varphi(X)$. We have that \[ |m_z(X,\xi)| \leq C_1 |\xi|e^{-|\xi|(b_{\text{\tiny{max}}} - b(X))} \quad \text{and} \quad |m_(X,\xi)| \leq C_2 |\nabla_X b(X) \cdot \xi| e^{-|\xi|(b_{\text{\tiny{max}}} - b(X))}. \]
Hence $\text{Op}(m_z)\varphi(X) + \text{Op}(m_X)\varphi(X)$ is smooth whenever $b(X) < b_{\text{\tiny{max}}}$. Next, since $\nabla_X b = 0$ at $b(X) = b_{\text{\tiny{max}}}$, we have   $ \text{Op}(m_X)\varphi(X) = \text{Op}(m_z)\varphi(X) = 0 $ when $b(X) = b_{\text{\tiny{max}}}$, and since both $m_z$ and $m_X$ depend continuously on $b(X)$, we conclude that $g(X)$ is continuous. 
We now construct a solution $\phi$. Let again $G$ be the Green's function satisfying \eqref{eq: greens}, and let the double layer potential $D$ be as before.  Moreover, we restrict the domain the single layer potential $S$ to $\tilde{\Gamma}_b$
\begin{equation}
    S\rho(X,z) = \int_{\tilde{\Gamma}_b} G(X,z,X',-b(X')) \rho(X') \mathrm{d}S(X'), \\
    \label{eq: potential}
\end{equation}
Using $S$, a unique solution $\phi$ to \eqref{eq: harmonic diff} such that $ \phi = o(1)$ as $|X| \to \infty$ is given by
\[ \phi(X,z) = S \rho (X,z)  \quad \text{in } \Omega, \]
where $\rho$ is the unique solution to the Fredholm equation 
\begin{equation}
    \rho(X) - 2 D\rho (X) = - 2g(X) , \quad X \in \tilde{\Gamma}_b. 
\end{equation}
In addition, we have $\|\rho\|_{L^\infty} \leq C \|g\|_{L^\infty} \leq C \|b\|_{C^1}\|\varphi\|_{L^2}$. For a proof of these results, cf. Theorems 6.28-6.30 in \cite{kress1989linear}. We now want to estimate $\partial_z \phi|_{z = 0}$. 
As $G$ is the sum of a harmonic function an the Newtonian potential, we have that $K(X,X',z,z') = \partial_z G \sim |(X-X',z-z')|^{-2}$ and $D^\alpha_X K \sim |(X-X',z-z')|^{-{2+|\alpha|}} $, which is $C^\infty$ for $z \neq z'$. Consequently, 
\[ D^\alpha_X \partial_z \phi|_{z=0}(X) = \int_{\tilde{\Gamma}_b} D^\alpha_X K(X,0,X',-b(X')) \rho(X') \mathrm{d}S(X').\] 
It now follows Young's inequality for integral operators (cf. Theorem 0.3.1 in  \cite{sogge2017fourier}) that 

\[ \|D^\alpha_X \partial_z \phi|_{z=0}\|_{L^2} \leq C \|\rho\|_{L^2} \leq C  \|b\|_{C^1}\|\varphi\|_{L^2},\quad \text{for any multi-index} \quad  |\alpha| \geq 0. \]
Hence $\partial_z \phi|_{z=0} \in H^k$ for any $k \in \N_0$ and hence for any $s \in \R$. Setting $\mathcal{K}(b)\varphi = \partial_z \phi|_{z=0}$
we therefore conclude that $\mathcal{K}(b) \in OPS^{-\infty}$. 
\end{proof}

\begin{lemma}
    For $m \in \R $ let  $a \in S^m$ and $\operatorname{Op}(a)$ be a PDO (Weyl or Kohn-Nirenberg quantization). Assume $k \in \N_0$ such that $k -m > 0$ and that $u \in C^k_b(\R^n)$. Then there is some constant $C > 0$ such that 
    \[ \|\operatorname{Op}(a) u\|_{L^\infty} \leq C \|u\|_{C^k}.\]
    \label{prop: pointwise}
\end{lemma}
\begin{proof}
The essential results used in the proof can be found in a suitable presentation in \cite{de2024full}, see Theorem 2.11 and the preceding discussion. For $s \in R$, let $C^s_*$ denote the Hölder-Zygmund space. For any $m \in \R$ and $a \in S^m$, it holds that 

\[ \|\operatorname{Op}_{\text{\tiny{KN}}}(a)\|_{C^{s}_*} \leq C  \|u\|_{C^{s+m}_*}.\]
For $k \in \N_0$ and $\alpha \in (0,1]$, let $C^{k,\alpha}$ be the standard Hölder space of $k$ times differentiable Hölder continuous functions. For $s = k + \alpha \neq \N_0$, the Hölder-Zygmund space and Hölder space agree, i.e., $C^s = C^{k,\alpha}$. For $u \in C^0_b$ we therefore have $\|u\|_{L^\infty} \leq C\|u\|_{C^s_*}$ for $s > 0$. Moreover, $C^{l+1}_b$ is continuously embedded in $C^{l,\beta}$ for $\beta \in (0,1]$. Now take $k \in \N_0$ and  $s \notin \N_0$ such that  $ k-m \geq s > 0$. Then 
\[\| \operatorname{Op}_{\text{\tiny{KN}}}(a) u\|_{L^\infty} \leq C \|\operatorname{Op}_{\text{\tiny{KN}}}(a) u \|_{C^s_*} \leq  C \|u \|_{C^{s+m}_*} \leq C\|u \|_{C^{k}_b} \quad \text{for } u \in C^k_b. \]
The results also holds for $\operatorname{Op}_{\text{\tiny{W}}}(a)$, since $\operatorname{Op}_{\text{\tiny{W}}}(a) = \operatorname{Op}_{\text{\tiny{KN}}}(a) + \operatorname{Op}_{\text{\tiny{KN}}}(r)$ with $r \in S^{m-1}$ (Ch. 4 in \cite{zworski2012semiclassical}).   
\end{proof}

\begin{proof}{(\Cref{prop: G est})}
We first show the approximation result for $\Gkn$ and then show that this implies that it also holds for $\Gw$. We start by defining the approximate solution to \eqref{eq: DN-def} by
\[ \Phi_a^\mu(X,z) = \frac{1}{(2\pi \mu)^2}\int_{\mathbb{R}^2} e^{iX\cdot \xi/\mu} \frac{\cosh(|\xi|(z + b(X))}{\cosh(|\xi|b(X))} \widehat{\varphi}_\mu(\xi)\mathrm{d}\xi,\]
where $\widehat{\varphi}_\mu(\xi) = \int_{\R^2}e^{-i\cdot \xi/\mu} \varphi(Y)\mathrm{d}Y$ Note that $\Gkn^\mu(b) \varphi  = \partial_z \Phi_a^\mu |_{z=0}.$
We find that 
\[ \Delta_{X,z}^\mu \Phi_a(X,z) = \frac{1}{(2\pi \mu)^2}\int_{\mathbb{R}^2} e^{iX\cdot \xi/\mu} R(X,z,\xi)\widehat{\varphi}_\mu(\xi)\mathrm{d}\xi, \]
where 
\[ R(X,z,\xi) = |\xi|\sinh(|\xi|z)\text{sech}^2(b|\xi|)\left( \mu^2\Delta_X b -2\mu|\xi| |\nabla_X b|^2\tanh(b|\xi|)  + \mu2i \xi\cdot \nabla_X b  \right). \]
The approximate solution satisfies 
\[ \Delta_{X,z}^\mu \Phi_a = f, \quad \Phi_a|_{z=0} = \varphi, \quad \partial_\nu^\mu \Phi_a|_{z = -b} = \nabla_X^\mu \Phi_a|_{z=-b}\cdot \nabla_X^\mu b, \]
where $f$ is given by
\[ f(X,z) = \frac{1}{(2\pi \mu)^2}\int_{\mathbb{R}^2} e^{iX\cdot \xi/\mu}R(X,z,\xi) \widehat{\varphi}_\mu(\xi)\mathrm{d}\xi. \]
We now consider the symbol $R$. First, note that since derivatives of $b$ vanish for $X \in R^2\setminus \tilde{\Gamma}_b$, we have $f(X,z) = 0$ for $X \in R^2\setminus \tilde{\Gamma}_b$, and so $f$ has compact support. Next, we have that $|\sinh(|\xi|z)|\text{sech}^2(b|\xi|) \leq \frac{1}{\cosh(b|\xi|)} \leq  C e^{-b|\xi|}$  for all $(X,z)$ and therefore $R \in S^{-\infty}$ and so $f \in C^\infty$. 
We now estimate $f$ pointwise. Set $r_{b,\alpha}(\xi) = \sinh(\alpha b |\xi|)\operatorname{sech}^2(b|\xi|)|\xi|^{-1}$ with $\alpha = z/b \in [-1,0]$ With the change of variables $\zeta = \xi/\mu$, we write 
{\small
\begin{equation*}
    f_\alpha(X) = \frac{\mu^3}{(2\pi)^2}\int_{\mathbb{R}^2} e^{iX\cdot \zeta} \left( \Delta_X b  - 2 |\nabla_X b|^2 |\zeta| \tanh(\mu b|\zeta|) + i2 (\nabla_X b \cdot \zeta) \right) r_{b,\alpha}(\mu\zeta)\mu\zeta \cdot \zeta  \hat{\varphi} \mathrm{d}\zeta
\end{equation*}
}
We estimate each term in the above expression, starting with the term $f_{\alpha,1}$ proportional to $\Delta_X b$. By the Fourier convolution theorem, we get 
\begin{align*}
f_{\alpha,1}(X,z) &= \frac{\mu^3\Delta_X b }{(2\pi)^2}\int_{\mathbb{R}^2}  e^{iX\cdot \zeta} r_{b,\alpha}(\mu\zeta) \mu \zeta\cdot \zeta \hat{\varphi}(\zeta) \mathrm{d}\zeta \\
&= -i \mu^3 \Delta_X b \left( (R_{b,\alpha,1} \ast \partial_{x_1} \varphi ) (X) + (R_{b,2} \ast \partial_{x_2} \varphi ) (X) \right)
\end{align*}

where \[ R_{b,\alpha,i}(Y) = \frac{1}{(2\pi)^2}\int_{\mathbb{R}^2} e^{iY\cdot \zeta}r_{b,\alpha}(\mu\zeta)\mu\zeta_i \mathrm{d}\zeta, \quad i = 1,2, \]
with $b = b(X)$ fixed. By Young's inequality, we have \[\|(R_{b,\alpha,i} \ast \partial_{x_1} \varphi )\|_{L^\infty} \leq \|R_{b,i}\|_{L^1} \|\partial_{x_i} \varphi \|_{L^\infty},\] and so we need to estimate $\|R_{b,\alpha,i}\|_{L^1}$. Using the fact that $\sinh(|x|)/|x|$ and $\operatorname{sech}^2(|x|)$ are analytic for all $x \in \R$ and that $r_{b,\alpha}(\zeta)$ exponentially decaying, we may conclude that $r_{b,\alpha}(\zeta) \in \mathcal{S}$, i.e., it is a Schwartz function. As $\mathcal{S}$ is closed under multiplication with polynomials, $r_{b,\alpha}(\zeta)\mu\zeta_i $ is also Schwartz function, and since $\mathcal{F}: \mathcal{S} \to \mathcal{S}$  it follows that $R_{b,\alpha,i} \in \mathcal{S}$ (Ch. 7, \cite{rudin1991functional}). Therefore, $R_{b,\alpha,i} \in L^1$. Moreover, using the change of variable $\zeta' = b\mu \zeta $ 
\[R_{b,\alpha,i}(Y) = \frac{1}{(\mu b)^2}G_\alpha(Y/(\mu b)) \quad \implies \quad  \|R_{b,\alpha,i}\|_{L^1} = C_\alpha,\]
i.e., $\|R_{b,\alpha,i}\|_{L^1}$ is independent of $b$ and $\mu$, but depends continuously on $\alpha \in [-1,0]$. 
We proceed similarly with 
\begin{align}
        f_{\alpha,2}(X) = \frac{-2\mu^3|\nabla_X b|^2}{(2\pi)^2}\int_{\mathbb{R}^2} e^{iX\cdot \zeta} |\zeta| \tanh(\mu b|\zeta|) r_{b,\alpha}(\mu\zeta)\zeta \cdot \mu\zeta  \hat{\varphi} \mathrm{d}\zeta.
\end{align}
It is straight forward to check that $|\zeta| \tanh(\mu b|\zeta|) r_{b,\alpha}(\mu\zeta)\zeta_i \in \mathcal{S}$. Overloading notation, we have now that 
\[R_{b,\alpha,i}(Y) = \frac{1}{(2\pi)^2}\int_{\mathbb{R}^2} e^{iY\cdot \zeta} |\zeta| \tanh(\mu b|\zeta|) r_{b,\alpha}(\mu\zeta)\zeta_i \mathrm{d}\zeta = \frac{1}{(\mu b)^3}\tilde{G}_\alpha(Y/(\mu b)).  \]
Hence, $\|R_{b,\alpha,i} \|_{L^1} = \frac{1}{\mu b}\tilde{C}_\alpha$, where is again independent of $b$ and $\mu$. We find $f_{\alpha,3}$ by the exact same procedure. In total we get that 

\[ |f(X,z)| = C_z \left( \mu^3 |\Delta_X b(X)| + \frac{\mu^2|\nabla_X b(X)| }{b} \left(1 + |\nabla_X b(X)| \right)\right)\|\nabla_X \varphi\|_{L^\infty},\]
where we have reintroduced the $z$-dependence through $C_z = C_\alpha = C_{z/b}$. \\
\noindent
A similar computation as above shows that 
At the bottom we have $\partial_z \Phi_a|_{z = -b} = 0$, and we define $ g(X) = \partial_\nu^\mu \Phi_a = \left(\nabla_X^\mu b \cdot \nabla_X^\mu \Phi_a\right)|_{z = -b(X)}.$ We have 
{\small
\[ \left(\nabla_X^\mu b \cdot \nabla_X^\mu \Phi_a\right)|_{z = -b(X)} = \frac{1}{(2\pi\mu)^2}\int_{\mathbb{R}^2} e^{iX\cdot \xi/\mu} \left(\frac{i\mu\nabla_X b \cdot \xi}{\cosh(b|\xi|)} + \mu^2\frac{|\nabla_X b|^2 |\xi| \sinh(b |\xi|)}{\cosh^2(b|\xi|)}\right)\widehat{\varphi}_\mu(\xi)\mathrm{d}\xi,\]
}
and by the same method as above, we get 
\[ |g(X)| \leq C\mu^2 |\nabla_X b(X)|(1 + \mu|\nabla_X b(X)|) \|\nabla_X \varphi\|_{L^\infty}. \]
Clearly, $g(X) = 0$ for $X \notin \tilde{\Gamma}_b$.   
We now consider the difference $\phi = \Phi - \Phi_a$, where $\Phi$ is the exact solution to \eqref{eq: DN-def}. $\phi$ satisfies 
\begin{equation}
    -\Delta_{X,z}^{\mu} \phi = f \quad \text{in } \Omega, \quad \phi|_{z=0}= 0, \quad \partial_\nu^\mu \phi|_{z = -b} = -g. \label{eq: diff_poisson2} 
\end{equation} 
We now follow the same strategy as in the proof of  \Cref{prop: G + K} to obtain pointwise bounds on $\phi$. To avoid working with the anisotropic Green's function, we note that if $\phi$ satisfies \eqref{eq: diff_poisson2} then $ \tilde{\phi}(X,z) = \phi(\mu X,z)$ satisfies  $\eqref{eq: diff_poisson2}$ with $\mu = 1$ and $\tilde{g}(X) = g(\mu X)$, $\tilde{f}(X,z) = f(\mu X,z)$ and $\tilde{b}(X) = b(\mu X)$. 

Let $G$ be the Green's function defined in \eqref{eq: greens}. In addition to the single- and double layer potentials in \eqref{eq: potential} used in proof of \Cref{prop: potpert} and \Cref{prop: G + K}, we introduce the volume potential   
\begin{align*}
    (V\tilde{f})(X,z) &= \int_{\tilde{\Gamma}_b\times (-\tilde{b}(X),0)} G(X,z,X',z') \tilde{f}(X',z') \mathrm{d}V(X',z'). 
\end{align*}
Using $V$ and $S$, a unique solution $\tilde{\phi}$ to \eqref{eq: diff_poisson2} (with $\mu = 1$) such that $ \phi = o(1)$ as $|X| \to \infty$ is given by
\[ \tilde{\phi}(X,z) = (V \tilde{f}) (X,z) + (S \rho) (X,z)  \quad \text{in } \Omega, \]
where $\rho$ is the solution to the Fredholm equation 
\begin{equation}
    \rho(X) - 2 D\rho (X) = - 2\left(\tilde{g}(X) + \partial_\nu (V \tilde{f}) (X,-\tilde{b}(X))\right) , \quad X \in \tilde{\Gamma}_b. 
\end{equation}
By the same consideration as in  \Cref{prop: G + K}, and since $L^\infty$ norms are invariant under scaling, we have $\|\partial_\nu (V \tilde{f})|_{z=-b} \|_{L^\infty} \leq C\|\tilde{f}\|_{L^\infty} \leq C\|f\|_{L^\infty}$, and  $\|\rho\|_{L^\infty} \leq C \left(\|g\|_{L^\infty} + \|f\|_{L^\infty}\right).$
For $\partial_z \tilde{\phi}|_{z = 0}$, Hölder's inequality gives
$\|\partial_z Vf|_{z=0}\|_{L^\infty} \leq C \|f\|_{L^\infty}$. For the bottom term, we also obtain additional dependence on the depth. For the Newtonian potential we have\footnote{Strictly speaking, $G^\mu$ is the sum the Newtonian potential and an harmonic function, but using the method of images it is not hard to show that the harmonic function does not alter the decay in our situation.}  ${\partial_z G \sim |(X-X',z-z')|^{-2} \leq b_{\text{\tiny{min}}}^{-2}}$, and since $\rho$ has compact support it follows that  \[\|\partial_z S \rho\|_{L^\infty} \leq b_{\text{\tiny{min}}}^{-2} C \|\rho \|_{L^\infty} \leq b_{\text{\tiny{min}}}^{-2}\tilde{C} \left(\|g\|_{L^\infty} + \|f\|_{L^\infty}\right).\]
As  $\| \partial_z \phi \|_{L^\infty} = \| \partial_z \tilde{\phi} \|_{L^\infty}$, we may therefore conclude that 
\begin{align*}
 \|\partial_z \phi|_{z = 0}\|_{L^\infty} & \leq  C\left(b_{\text{\tiny{min}}}^{-2}(\|f\|_{L^\infty} + \|g\|_{L^\infty}) + \|f\|_{L^\infty}\right),
\end{align*}
The result for $\Gkn$ now follows, since with $b_{\text{\tiny{min}}} \geq 1$ and $\|\Delta_X b\|_{L^\infty} \leq \delta$, the dominant term is  
\[ \|\partial_z \phi|_{z = 0}\|_{L^\infty} = \|\mathcal{G}^\mu(b)\varphi - \Gkn^\mu(b) \varphi \|_{L^\infty} \leq C\mu^2 \frac{\delta(1+ \delta)}{b_{\text{\tiny{min}}}} \|\nabla_X \varphi \|_{L^\infty} + \mathcal{O}(\mu^3)\] 
We now write $k(X,z,\xi) = \frac{\cosh(|\xi|(z + b(X))}{\cosh(|\xi|b(X+Y))} $, and write the potential corresponding to $\Gw^\mu$ as 
\[\Phi_{a,\text{\tiny{W}}}^\mu(X,z) = \frac{1}{(2\pi \mu)^2}\int_{\mathbb{R}^2} e^{iX\cdot \xi/\mu} k((X+Y)/2,z,\xi) \varphi(Y)\mathrm{d}Y \mathrm{d}\xi. \]
It is straight forward to check that $\Phi_{a,\text{\tiny{W}}}^\mu$ satisfies $\Phi_{a,\text{\tiny{W}}}^\mu|_{z=0} = \varphi$ and $\partial_z \Phi_{a,\text{\tiny{W}}}^\mu|_{z=-b} = 0$. We also have that 
\[ \Delta_X^\mu Op_{\text{\tiny{W}}}^\mu(k) = Op_{\text{\tiny{W}}}^\mu( -|\xi|^2k - i\mu \nabla_X k \cdot \xi + \frac{\mu^2}{4} \Delta_X k).\] 
From this it follows that 
\[ \Delta_{X,z}^\mu \Phi_{a,\text{\tiny{W}}}^\mu = Op_{\text{\tiny{W}}}^\mu(- i\mu \nabla_X k \cdot \xi + \frac{\mu^2}{4} \Delta_X k) \varphi.\]
We now use that for a symbol $a$, the Weyl and Kohn-Nirenberg quantizations are related through the formula  $Op_{\text{\tiny{W}}}^\mu(a) = Op_{\text{\tiny{KN}}}^\mu(e^{\frac{i\mu}{2} \nabla_X \cdot \nabla_\xi} a)$ (Ch. 4,\cite{zworski2012semiclassical}). Writing $R_{\text{\tiny{W}}} = - i\mu \nabla_X k \cdot \xi + \frac{\mu^2}{4} \Delta_X k$, this implies   
\[ \Delta_{X,z}^\mu \Phi_{a,\text{\tiny{W}}}^\mu  = Op_{\text{\tiny{KN}}}^\mu(R_{\text{\tiny{W}}}) \varphi + \frac{i\mu}{2}Op_{\text{\tiny{KN}}}^\mu(\nabla_\xi \cdot \nabla_X R_{\text{\tiny{W}}} ) \varphi + \mathcal{O}_{S^{-\infty}}(\mu^3). \]
where the remainder estimate follows from the fact that differentiation of $R(X,z,\xi)$ with respect to $X$ always results in a $S^{-\infty}$ symbol. We now note that $R_{\text{\tiny{W}}} \sim R $ (i.e., the symbols differ only by multiplicative constants). Moreover, by the same method used to estimate $f$ and a more tedious calculation, one finds that \\ $\|Op_{\text{\tiny{KN}}}^\mu(\nabla_\xi \cdot \nabla_X R_{\text{\tiny{W}}} ) \varphi\|_{L^\infty} = \mathcal{O}(\mu^3 \|\nabla_X \varphi\|_{L^\infty})$, and it follows that 
 \[ \|\Delta_{X,z}^\mu \Phi_{a,\text{\tiny{W}}}^\mu\|_{L^\infty} = \|f\|_{L^\infty} + \mathcal{O}(\mu^3).\]
By the same approach, we also find
\[ \|\nabla_X^\mu b \cdot \nabla_X^\mu \Phi_{a,\text{\tiny{W}}}^\mu |_{z=-b}\|_{L^\infty} = \|g\|_{L^\infty} + \mathcal{O}(\mu^3). \]
Consequently, the same analysis of the difference $\phi_{\text{\tiny{W}}} = \Phi^\mu - \Phi_{a,\text{\tiny{W}}}^\mu $ now results in 
\[ \|\mathcal{G}^\mu \varphi - \Gw^\mu \varphi \|_{L^\infty} = \| \phi_{\text{\tiny{W}}} \|_{L^\infty} \leq C\mu^2 \frac{\delta(1+ \delta)}{b_{\text{\tiny{min}}}} \|\nabla_X \varphi \|_{L^\infty}  + \mathcal{O}(\mu^3). \]

Last, we establish the self-adjointness of $\Gw^\mu(b)$. As $g_b \in S^1$ and real-valued it satisfies the criteria of Theorem 1 in \cite{fulsche2025simple}, and so $\Opw^\mu(b)$ is essentially self-adjoint. Therefore, it has a unique, self-adjoint realization $\Gw^\mu(b)$ on $L^2$ (cf. Ch. 2,\cite{teschl2014mathematical}), with domain 
\[\mathcal{D}(\Gw^\mu(b)) = \{u \in L^2 : \Gw^\mu(b)u \in L^2 \}.\]
As $\Gw^\mu(b) : H^s \to H^{s-1}$, we have $H^1 \subset \mathcal{D}(\Gw^\mu(b)).$  Next, we write $g_b(X,\xi) = \bar{g}(\xi)  +  r(X,\xi)$ with $\bar{g}(\xi) = |\xi|\tanh(b_{\text{\tiny{max}}}|\xi|)$ and $ r(X,\xi) = |\xi|\tanh(b(X)|\xi|)- |\xi|\tanh(b_{\text{\tiny{max}}}|\xi|).$ 
Using that $\tanh(x) = 1 - 2e^{-2x} + \mathcal{O}(e^{-4x})$ we see that  \[r(X,\xi) = -|\xi|2e^{-2b(X)|\xi|}(1 -e^{-2(b_{\text{\tiny{max}}} -b(X))|\xi|}) + \mathcal{O}(e^{-4b(X)|\xi|}) \implies r(X,\xi) \in S^{-\infty}. \] 
For the translation invariant symbol $\bar{g}$ we have $\Opw^\mu(\bar{g}) = \Opkn^\mu(\bar{g})$, and it follows that $\Gw^\mu(b) = \Opkn^\mu(\bar{g}) + R$ with $R \in OPS^{-\infty}$. By the standard theory of elliptic PDOs (cf. Ch. 5.4 in \cite{alinhac2007pseudo}), there exists an approximate inverse $Q: H^s \to H^{s+1}$ such that $ Q\Opkn^\mu(\bar{g}) = I + S$ with $S \in OPS^{-\infty}$. Now, assume $\Gw^\mu(b)u  = v \in L^2$. Then $\Opkn^\mu(\bar{g})u = v - Ru $ and $u = Q v - QRu - Su $, and since $Su, Ru \in H^\infty$ we must have $u \in H^1$. Consequently, $\mathcal{D}(\Gw^\mu(b)) \subset H^1$. 
\end{proof}

\bibliographystyle{plain}
\bibliography{references}

\end{document}